\documentclass[11pt]{article}
\usepackage{my_macros}
\usepackage{fullpage}
\usepackage{algorithm}
\usepackage[noend]{algorithmic}
\usepackage{setspace}
\usepackage[dvipsnames]{xcolor}
\usepackage[semicolon, sort, square]{natbib}
\usepackage{calrsfs}
\usepackage{thmtools} 
\usepackage{thm-restate}
\usepackage{lmodern}
\usepackage{subcaption}
\usepackage{graphicx}
\usepackage{subcaption}
\usepackage{float}
\usepackage{authblk}

\makeatletter
\let\@fnsymbol\@arabic
\makeatother

\newcommand{\Pois}{\mbox{Pois}}
\renewcommand{\Exp}{{\mbox{Exp}}}

\newcommand{\hReff}{\widehat{\mathfrak{R}}_{\text{eff}}}
\newcommand{\Reff}{{\mathfrak{R}}_{\text{eff}}}

\newcommand{\Rn}{\mathfrak{R}_0}

\newcommand{\Rinf}{{\mathfrak{R}}_\infty}
\newcommand{\irt}{\eta}
\newcommand{\rrt}{\gamma}

\newcommand{\Bern}{\text{Bern}}
\newcommand{\sbm}{\mathrm{SBM}}
\newcommand{\psbm}{\mathrm{PSBM}}

\newcommand{\bp}{\mathrm{BP}}

\newcommand{\diag}{\mathrm{diag}}
\newcommand{\Mgen}{M}
\newcommand{\Msir}{C}
\newcommand{\hMgen}{\widehat{M}}
\newcommand{\hMsir}{\widehat{C}}
\newcommand{\tMsir}{C_1}
\newcommand{\ttMsir}{C_2}

\newcommand{\ttX}{\X^{(2)}}
\newcommand{\ttF}{\calF^{(2)}}
\newcommand{\ttH}{H_2}

\title{A Law of Large Numbers for SIR on the Stochastic Block Model: \\ A Proof via Herd Immunity} 

\author{
  Christian Borgs\textsuperscript{1, 3} \quad
  Karissa Huang\textsuperscript{2} \quad
  Christian Ikeokwu\textsuperscript{1}
}
\begin{document}
\maketitle
\begin{abstract}
In this paper, we study the dynamics of the susceptible-infected-recovered (SIR)  model on a network with community structure, namely the stochastic block model (SBM).  As usual, the SIR model is a stochastic model for an epidemic where infected vertices infect susceptible neighbors at some rate $\irt$ and recover at rate $\rrt$, and the SBM is a random graph model where vertices are partitioned into a finite number of communities. The connection probability between two vertices depends on their community affiliation, here scaled in such a way that the average degrees have a finite limit as the network grows. 
We prove laws of large numbers (LLN) for the trajectory of the epidemic to a system of ordinary differential equations  over any time horizon (finite or infinite), including in particular a LLN for the final size of the infection. 

Our proofs rely on two main ingredients: (i) a new coupling of the SIR epidemic and the randomness of the SBM, revealing a vector valued random variable that drives the epidemic (related to what is usually called the ``force of the infection'' via a linear transformation),  and (ii) a novel technique for analyzing the limiting behavior of the infinite time horizon for the infection, using the fact that once the infection passes the herd immunity threshold it dies out quickly and has a negligible impact on the overall size of the infection.

\footnotetext[1]{Department of Electrical Engineering and Computer Science, University of California, Berkeley.}
\footnotetext[2]{Department of Statistics, University of California, Berkeley.}
\footnotetext[3]{Bakar Institute of Digital Materials for the Planet (BIDMaP).}

\end{abstract}

\section{Introduction}
\label{sec:intro}
The susceptible-infected-recovered (SIR) model was introduced by Kermack and McKendrick in 1927, and is a simple model used to describe the behavior of a disease in a closed, finite population \cite{kermack1927}. In the simplest version of the model, the population is divided into three compartments -- susceptible, infected, and recovered  -- whose sizes can be thought of as  deterministic, time-varying, continuous variables $s=s(t)$, $i=i(t)$, $r=r(t)$ adding up to one. The dynamics of the infection are governed by constants $\irt$ and $\rrt$, and described by a well-known system of differential equations representing a ``flow of individuals'' from $S$ to $I$ and $I$ to $R$ at rates $\irt  s i$ and $\rrt i$, respectively. In the slightly more realistic stochastic setting first considered by \cite{bartlett1949evolutionary}, the population consists of $n$ individuals again divided into compartments of susceptible, infected, and recovered individuals.  The time dynamics of the model are now stochastic, with infected individuals infecting susceptible individuals at rate $\irt/n$, and infected individuals recovering at rate $\rrt$. As $n$ becomes large, the stochastic model is well approximated by the deterministic one, mathematically expressed as a law of large numbers (LLN). This implies that on any bounded time horizon, the fraction of susceptible, infected, and recovered individuals converge to deterministic functions $s(t)$, $i(t)$, and $r(t)$ given by the solutions of the usual differential equations, see, e.g., \cite{vonbahr1980threshold} for the study of a somewhat  related version in discrete time (the so-called Reed-Frost model) with a generalization to the Kermack-McKendrick SIR model, or the excellent review article \cite{darling08:dfq-markov} for a short, streamlined proof based on Gronwall’s inequality and Doob’s $L^2$-inequality.



It is well known that the homogeneity assumption inherent in the original Kermack and McKendrick model is not realistic in practice, as contact rates between individuals vary by age and other demographic factors \cite{Jacquez1988}. While this can be, to some extent, addressed by creating more compartments stratified by these demographic factors, another property which a realistic model should take into account is the underlying graph structure of the contact network.  Indeed, it is well known that the underlying graph structure 
plays a crucial role in the dynamics of the infection throughout the graph. As such, there has been interest in studying the SIR model on various random graphs including the Erd\"os-R\'enyi graph \citep{neal2003sir, ball22:sir-rewire}, graphs with a given degree sequence (configuration models) \citep{Decreusefond2012, Bohman2012, barbour2013approximating, janson14:lln-sir}, and graphs with local household structure \citep{Ball2010, house2012modelling} or tunable local clustering, \cite{Britton2008}.

In this paper, we consider the SIR epidemic on the stochastic block model (SBM), a network model with community structure. In the SBM, nodes are partitioned into $K$ communities and the distribution of edges between vertices is determined by a $K\times K$ matrix and is dependent on the group membership of the endpoints \cite{holland1983stochastic}. The SBM has appeared independently in many contexts in the statistics, computer science, and mathematics literature and there are many examples of networks with known community structure \citep{holland1983stochastic, bui1987graph, girvan2002community, bollobas_2007}. In addition there is a vast literature on the use of the stochastic block model in community detection problems \citep{abbe2018community, lee2019review}.

Despite its widespread study and use, somewhat surprisingly unlike other well-studied graph models, no rigorous work has established analogous law of large numbers results for a limiting system of ordinary differential equations for the stochastic SIR epidemic on the SBM. The rigorous derivation of these is the goal of the present paper.  

Our work is motivated and inspired by the proofs of LLNs for  the Erd\"os-R\'enyi graph $G(n, p)$ and variations of $G(n, p)$ involving the dropping and rewiring of edges as the infection spreads, as well as LLNs for graphs of a given degree sequence \citep{neal2003sir, janson14:lln-sir, ball22:sir-rewire}. Our work differs from these in one important aspect, in all these works, the dynamics of the infection process are primarily driven by one random variable, often referred to as the force of the infection. For $G(n, p)$ and its variants, this random variable is, roughly speaking, the number of edges between infected and susceptible vertices, and for the configuration model it is the number of not-yet explored or matched edges emanating from infected vertices. 

In contrast to the configuration model where the resulting limiting equations were derived heuristically before being established mathematically rigorously, to our knowledge, there has been no conjecture for the correct differential equations describing a LLN for SIR on the SBM.

To address this, we build on ideas from  \cite{ball22:sir-rewire} and \citet*{janson14:lln-sir}. First, we couple the randomness of the SBM and that of SIR by only exploring edges in the graph when they attempt to spread the SIR infection.  More precisely, once a vertex gets infected, we draw its degree (with a distribution that depends on the community type of the infected vertex), and then attach what we call active half-edges to it, each  endowed with an exponential clock governing future infection events.  Whenever these clocks click, we determine the label of the other endpoint, and ``test'' whether that endpoint is still susceptible before spreading the infection to a new  vertex.  This representation then naturally leads us to consider a LLN for three vector valued random variables: 1) the number of susceptible vertices in each community, 2) the number of infected vertices in each community and, 3) the number of not yet explored, active half-edges emanating from vertices of each community type.

Having derived this representation, both the heuristic derivation and the mathematical proof of a LLN for finite time horizons is straightforward, with proofs and quantitative error bounds following from standard techniques involving Gronwall’s inequality and Doob’s $L^2$-inequality as, e.g., laid out in \cite{darling08:dfq-markov}. However, while these bounds are uniform for all times in a finite time interval $[0,t_0]$, the size of the error bounds turn out to be exponential in $t_0$, which prevents us from using the same techniques to prove a law of large numbers for an infinite time-horizon and, in particular, for what is usually called the final size of the infection.

The standard way of overcoming such a challenge is to apply a global time rescaling to the epidemic so that the entire epidemic happens on a finite time horizon \citep{ethier1986random, daley1999epidemic, kiss2017mathematics}. However in the case of the SBM, because the infection rate differs depending on the community, it is unclear how to appropriately apply such a global time rescaling. To overcome this problem, we develop a new technique to prove the infinite time law of large numbers. This proof technique is based on the observation that the infection must eventually pass the herd immunity threshold and that after this point, the infection dies out quickly. Thus, for a large enough time $t_1$ after the infection has passed herd immunity,  any outbreak that occurs after time $t_1$ has a negligible effect on the final state of the epidemic at time infinity (and, in fact, for all times between $t_1$ and infinity).  We believe that this technique should be applicable in much more generality than  traditional time rescaling techniques allowing for the analysis of more models but will also greatly simplify proofs of the final size LLN for many models.

We close this introduction by briefly mentioning two alternative methods which could  at least in principle be used to derive some of the results established in this paper, even though to our knowledge, they have not yet been applied to the stochastic block model.
One of these methods involves the theory of local graph limits  \citep{cocomello2023exact, alimohammadi2023epidemic}, which can be used to study the final size of an epidemic starting from a constant fraction of initially infected vertices, provided these are chosen uniformly at random. Another method involves the method of forward and backward branching processes, see, e.g.,  \cite{barbour2013approximating}.  We will discuss these and other related works in Section \ref{sec:related work} below, where we put our work into the context of the existing literature.

This paper is organized as follows. In Section \ref{sec:prelim} we define our model and state our main results, including the laws of large numbers for both the finite and infinite time horizon. In Section \ref{sec:actual-prelim}, we define  a multi-graph version of the SBM and then couple the SIR epidemic to the randomness of this model. As, e.g., in the case of the multi-graph vs. simple graph version of the configuration model, conditioning the multi-graph version of the SBM to be simple will give back the original simple graph version of the model, allowing us to transfer our results for the multi-graph model to the original one.  We will also give the heuristic derivation of the LLN in this section. In Section \ref{sec:lln}, we prove the law of large numbers in the finite time case where the number of initially infected vertices grows linearly with $n$, and in Section \ref{sec:uniqueness-implicit-solution} we extend the results from Section \ref{sec:lln} to the infinite time horizon. Section \ref{sec:o-n} extends all results to the case where there is only $o(n)$ many initially infected vertices, with some of the more technical proofs deferred to an appendix.  In Section~\ref{sec:discussion},  we discuss the implications of our result, including a discussion of our notion of herd-immunity and its relation to the standard notion of the force of an infection, as well as an alternative heuristic derivation of our results using what is often called the pair-approximation.  We kept Section~\ref{sec:discussion} essentially self-contained, so after reading the summary of our results in Section~\ref{sec:prelim} and possibly Section~\ref{sec:actual-prelim}, the reader not interested in technical details might want to jump to the discussion section.


\bigskip

\noindent{\bf Acknowledgements}: We are grateful to Tom Britton and Remco Van der Hofstad for their insightful discussions and for encouraging us to pursue this write-up, particularly in the early stages of developing the concept of using herd immunity to prove a law of large numbers for the epidemic’s final phase. Christian Borgs also extends thanks to the Simons Institute for the Theory of Computing for their support during the fall semester of 2022, and all three authors appreciate the Institute’s hospitality during that time. Lastly, Karissa Huang acknowledges support from the National Science Foundation under grant DGE 2146752.



\section{Notation and Statements of Results}
\label{sec:prelim}

\subsection{Notation} 

Throughout this paper, we use $n$ to denote the number of individuals or number of vertices in the contact graph.
We use $\pto$ for convergence in probability as $n\to\infty$,
and we say that an event holds with high probability if it holds with probability tending to $1$ as $n\to\infty$. For a positive integer $K$, we use $[K]$ to denote the set $[K]=\{1,\dots,K\}$.  We usually use bold letters for vectors, and if $\vec a=(a_k)_{k\in K}$, we use $\diag(\vec a)$ for the diagonal matrix with entries $(\diag(\vec a))_{kk}=a_k$.  As usual, we will use $\N$ to denote the set of natural numbers, and $\N_0$ to denote the set of non-negative integers.

\subsection{Model Definition}
\label{sec:model}

We consider an SIR epidemic spreading on a contact network of $n$ individuals with a community structure modeled by a stochastic block model with $K$ communities. To define the model, we first
assign a label $k(v)\in [K]$ to each individual in $V=[n]$,
leading to a labeled set  $\bm V=(V,k(\cdot))$ of $n$ vertices with labels in $[K]$.  In a second step, we connect vertices  at random, with the probability of an edge between two vertices $u$ and $v$ depending on the community labels $k(u)$ and $k(v)$.   This leads to the following formal definition.

\begin{definition}[Stochastic Block Model (SBM)]\label{def:SBM}
    Let $K \in \N$,  let $\bm V=(V,k(\cdot))$ be a set of vertices with labels in $[K]$, and let $W$ be a symmetric, non-negative matrix such that each row contains at least one non-zero element. We say that a random graph $G=(V,E)$ is drawn from a {\textbf stochastic block model with affinity matrix} $W$, and write $G \sim \sbm( \bm V, W)$,
    if a pair of vertices $\{u,v\}$ is joined by an edge with probability $\frac 1n W_{k(u)k(v)}$ independently for all 
    $n\choose 2$ pairs  
    $\{u,v\}$.  We will use $n_k$ to denote the number of vertices with label $k$, i.e. $n_k=\sum_{v\in V}\ind{k(v) = k}$.  
\end{definition}

\begin{remark}
\label{asmp:W-nk-n-asmp}
i) Note that the assumption that each row $W_{k, \cdot}$ of $W$ contains at least one non-zero elements can be made without loss of generality, since otherwise vertices of color $k$ would always be isolated, in which case we can just delete them from the graph $G$.

ii) The SBM is often defined by choosing the labels of the vertices $i.i.d.$ from some random distribution $\bm \rho=(\rho_k)_{k\in [K]}$ over $[K]$. The results of this paper can easily be formulated for this version of the SBM as well; in fact, all our results hold for an arbitrary label distribution over $[n]$, as long as w.h.p., for all $k\in [K]$ the ratio $n_k/n$ stays bounded away from $0$ uniformly in $n$. 
\end{remark}

In order to analyze SIR epidemic on the stochastic block model, 
it will be convenient to first study it on a slightly different model which allows for self-loops and multiple edges.  We call it the Poisson Stochastic Block Model (PSBM).

\begin{definition}[Poisson Stochastic Block Model (PSBM)]
    Let $\bm V$ and $W$ be as in Definition~\ref{def:SBM}.
We say that a random multi-graph $G$ is drawn from a Poisson Stochastic Block Model, $G\sim \psbm(\bm V,W)$, if each pair of vertices $\{u,v\}$ (for $u, v \in [V]$) is joined by
$A_{uv}$ edges, and each vertex $v$ has 
$A_{vv}$ self-loops, where $A_{uv}=A_{vu}\sim \Pois\left(W_{k(v)k(v)}/n\right)$ independently for all $v\leq u$.  We use $D_k$ 
\begin{equation}\label{Dk}
    D_k=\sum_{\ell\in [K]} W_{k\ell}\frac{n_\ell}{n}
\end{equation}
to denote the average degree of a vertex of type $k$.
\end{definition}

\begin{remark}\label{rem:Poisson-degrees}
The degree of a vertex $v$ in $G\sim \psbm(\bm V,W)$ is distributed according to $d_v\sim\Pois(D_{k(v)})$, and its degree into the vertices with community label $\ell$ is Poisson distributed with mean $W_{k(v)\ell}\frac{n_\ell}{n}$.  Setting 
\begin{equation}\label{pellk}
  p_{k\to\ell}=\frac 1{D_k}W_{k\ell}\frac{n_\ell}{n},  
\end{equation}
 we furthermore can couple these in such a way that conditioned on $d_v$, the degree into a community with label $k$ can be obtained by choosing labels $k_1,\dots, k_{d_v}$
i.i.d. with probability $p_{k(v)\to k_i}$ and setting the degree into community $\ell$ as the number of times $\ell$ will appear among 
$k_1,\dots, k_{d_v}$.  
\end{remark}



Next,
we formally define an SIR infection on {a general multi-graph $G$, introducing at the same time some of the notation used later. Denote the number of edges between two vertices $u$ and $v$ by $A_{uv}$. The SIR model is then a continuous-time Markov process whose state at time $t$ is given by the sets of susceptible, infected, and recovered vertices, 
$V^S(t)$, $V^I(t)$ and $V^R(t)$, respectively (with their union being equal to the set of vertices $V$ in $G$).
Starting from an initial state
$V^S(0)$, $V^I(0)$, and $V^R(0)$ at time $t=0$, at time $t> 0$ each  $u\in V^I(t)$ then} infects a susceptible vertex $v$ at rate $\irt {A_{uv}}$, and recovers at rate $\rrt$, independently for all vertices $u\in V^I(t)$.  
We adopt the standard assumption that our processes are càdlàg (or rcll), so that when vertex $v$ gets infected at time $t$, $v\in V^S(t')$ for  $t'<t$ and $v\in V^I(t')\cup V^R(t')$ for  $t'\geq t$ (and similarly for the recovery of a vertex $u$ at time $t$).

We will use the notation
$$V^S_k(t)=V^S(t)\cap V_k, \;\;S_k(t)=|V^S_k(t)|\;\;\text{ and }\;\; S(t)=|V^S(t)|. $$
We define $V^I_k(t)$, $I_k(t)$ and $I(t)$, and $V^R_k(t)$,  $R_k(t)$  and $R(t)$ analogously.

\subsection{Statements of Results}

In this paper, we separately discuss two types of initial conditions: one in which the number of initially infected vertices is proportional to $n$, and one where it is $o(n)$.

We start with the former, for which we will establish laws of large numbers for the time evolution as well as the final size of the infection.  
We will need the following assumptions on the initial state of the infection:

 \begin{assumption}\label{ass:initial-SandI}
 As $n\to\infty$,
 \begin{enumerate}[i)]
\item 
\label{asmp:init-s}
$\frac{1}{n} S_k(0)\pto s_k(0)>0$   for all $k\in[K]$,
\item
\label{asmp:init-i}
$\frac{1}{n} I_k(0)\pto i_k(0)\geq 0$   for all $k\in[K]$.

 \end{enumerate}
 We will assume that the initially infected and recovered vertices are chosen at random, possibly depending on their community labels, but not 
on the realization of $\sbm(\bm V,W)$ or $\psbm(\bm V,W)$. 
\end{assumption}

We are making the assumption on the random choice of initially infected and recovered vertices to avoid trivial examples, like an initial set of infected vertices that consists of a connected component in $G$ (in which case
all that happens is the independent recovery of all initially infected vertices).  Note that equivalently, we could assume an arbitrary starting configuration
$V^S(0)$, $V^I(0)$, and $V^R(0)$, and draw $G\sim \sbm(\bm V,W)$ or $G\sim \psbm(\bm V,W)$ independently 
of this starting configuration.

\begin{theorem}[Law of Large Numbers for Finite Time]
\label{thm:LLN}
       Consider the SIR epidemic on $G\sim \psbm(\bm V,W)$ or $G\sim \sbm(\bm V,W)$, with the initial state of the infection  obeying the conditions from Assumption~\ref{ass:initial-SandI}.  
    Let $s_k(t)$, $i_k(t)$,  and $x_k(t)$ be the unique solutions to 
    \begin{equation}\label{eq:s_dfq_1}
        \frac{ds_k(t)}{dt} = -\irt s_k(t) \sum_\ell x_\ell(t) 
        {W_{\ell k}},
    \end{equation}
    \begin{equation}\label{eq:x_dfq_1}
        \frac{dx_k(t)}{dt} = -(\irt + \rrt) x_k(t) + \irt s_k(t) \sum_\ell x_\ell(t) W_{\ell k}, 
    \end{equation}
    \begin{equation}\label{eq:i_dfq_1}
        \frac{di_k(t)}{dt} = -\rrt i_k(t) + \irt s_k(t) \sum_\ell x_\ell(t) 
        {W_{\ell k}}, 
    \end{equation}
with initial conditions $s_k(0)$ and $x_k(0)=i_k(0)$, and let $0<t_0<\infty$, and let $\eps>0$ be arbitrary.  Then as $n \to \infty,$

    \begin{equation}
        \P\left( \sup_{t\in [0,t_0] } \left\| \left(\frac 1n \S(t), \frac 1n\I(t)\right)
        - \left(\s(t), \i(t)\right)\right\|_2 > \eps \right) \to 0,
    \end{equation}
    where $\vec S(t)$ is the vector $(S_1(t),\dots, S_K(t))$, and similarly for $\vec I(t)$, $\vec s(t)$ and $\vec i(t)$.
\end{theorem}

\begin{remark} \label{rmk:r_dfq}
i) Assumption~\ref{ass:initial-SandI} does not have any condition on the set of initially recovered vertices, and our theorem does not state any LLN for $\frac 1nR_k(t)$. However, if we add the assumption that  for all $k\in[K]$ ,
$n_k/n$ converges in probability to some $\rho_k$, we get the convergence of $R_k(t)/n=n_k/n-S_k(t)/n-I_k(t)/n$ to $r_k(t)=\rho_k-s_k(t)-i_k(t)$ for free.

ii) Note that the assumptions of the theorem include the case where $\i(0)=0$, in which case the solution of the differential equations are constant, $\i(t)\equiv 0$ and $\s(t)\equiv \s(0)$ for all $t$.  Note, however, that  this does not imply that the actual infection dies out.  Indeed, if $\frac{\I(0)}n\to 0$ 
but $I(0)\to\infty$ (and the basic reproduction number, $\Rn$, is larger than one), the infection will eventually take off, see Theorem~\ref{thm:one-vertex-final-size} below.  It just takes an amount of time which diverges as $n\to\infty$!  By contrast,  $\frac 1n(\S(t)-\S(0))\to 0$ and $\frac{1}n\I(t)\to 0$ uniformly for $t$ on any \emph{bounded} time interval, consistent with the statements of the theorem.
\end{remark}

The above remark raises the question of whether the bounds of Theorem~\ref{thm:LLN} can be generalized to non-compact time intervals if we assume that $\i(0)\neq 0$.  Theorem~\ref{thm:lln-final-size} below shows that this is indeed possible.  It in particular gives a law of large numbers for the final state of the infection.  

Before stating the theorem, we note that for fixed $n$, the infection will eventually die out, leaving just a set of vertices which never got infected (and hence are still susceptible) and a set of  vertices that got infected at some point, but eventually recovered. The number of vertices in the latter set is usually referred to as the \emph{final size of the infection}.  
Since the two sets are complements of each other, the final size of the infection (for the individuals in community $k$) is therefore equal to
$
I_k(0)+S_k(0)-S_k(\infty).
$
Under the convergence assumption for the intitial state, Assumption~\ref{ass:initial-SandI}, a law of large numbers for the final size of the infection is therefore equivalent to a 
law of large numbers for
    \[
        \S(\infty) = \lim_{t\to\infty} \S(t).
    \]
The next theorem generalizes Theorem~\ref{thm:LLN} to unbounded time intervals, 
and in particular relates the quantity $\S(\infty)$ to the limit $t\to\infty$ of $\s(t)$, where $\s(t)$ and
$\x(t)$ are the solutions of the differential equations \eqref{eq:s_dfq_1} and
\eqref{eq:x_dfq_1}.  Note that the 
existence of the limit $\lim_{t\to\infty}\s(t)$ follows from the monotonicity of $s_k(t)$ for all $k$, which in turn is a consequence of \eqref{eq:s_dfq_1}.

\begin{theorem}
\label{thm:lln-final-size}
        Under the additional assumptions that $W$ is irreducible and  $i(0) >0$ as $n\to\infty$, the statement of Theorem~\ref{thm:LLN} holds uniformly for all $t\in \R_+$,
        
        \begin{equation}
            \P\left( \sup_{t\in [0,\infty) } \left\| \left(\frac 1n \S(t), \frac 1n\I(t)\right)
            - \left(\s(t), \i(t)\right)\right\|_2 > \eps \right) \to 0,
        \end{equation}
    so in particular
        \[
            \P\left(\left\|\frac{1}{n}\S(\infty) -  \s(\infty)\right\|_2 > \eps\right) \xrightarrow{n\to\infty} 0.
        \]
     where $\s(\infty)=\lim_{t\to\infty}\s(t)$.
\end{theorem}


The second statement of Theorem \ref{thm:lln-final-size} gives a law of large numbers for the final size of the infection when the number of initially infected vertices is a constant fraction of the size of the graph. In Theorem \ref{thm:one-vertex-final-size}, we consider the law of large numbers 
for the regime where the number of initially infected vertices is $o(n)$.  

We start with the simple case of a single, initially infected vertex.  We need some definitions. 
\begin{definition}[Poisson Multi-type Branching Process]
\label{def:pois-bp}
Consider a $K\times K$ matrix $\Mgen$ with non-negative entries, and the
multi-type, discrete time Poisson branching process where vertices of type $k$ have $\Pois(\Mgen_{k\ell})$ children of type $\ell$.    We will use  ${\bp}_k(\Mgen)$ to denote the distribution of the labeled tree corresponding to this branching process starting from a root of type $k$.
\end{definition}

\begin{definition}[Infection Tree for arbitrary graphs]
Consider the SIR epidemic on an arbitrary graph $G$ starting from a single vertex $v$ in $G$. We use $\calT_{G, v}(t)$ to denote the infection tree on $G$ at time $t$, defined as the tree on all vertices that were infected at or before time $t$ (with labels indicating whether they are currently infected or recovered), and an edge between two vertices if it is an edge in  $G$ along which the infection has spread. If $G$ carries community labels in some set $[K]$, we add these labels as a second label to the vertices in $\calT_{G, v}(t)$.

We consider two special cases: the case where $G\sim\sbm(\V,W)$ for an infection starting from single vertex $v$ chosen uniformly at random from all vertices with label $k$, in which case we denote the infection tree by $\calT_k^{\sbm}(t)$, and the case where the infection starts at the root of a Poisson branching process $\bp_k(\Mgen)$, in which case we denote the infection tree by $\calT_k^{\bp(\Mgen)}(t)$, or simply by $\calT_k^{\bp}(t)$ when $\Mgen$ is clear from the context.
\end{definition}

\begin{remark}
It is well known and easy to see that $\calT_k^{\bp(\Mgen)}(\infty)$ can be represented as a standard, discrete time branching process as follows: starting from a vertex $v$ of type $k(v)$, choose an exponentially distributed recovery time $T_v\sim \exp(\rrt)$ and then choose $\Pois(p_{T_v}\Mgen_{k(v)\ell})$ children of type $\ell$, where
$p_{T_v}=1-e^{-\irt T_v}$.  (To see this, just note that $p_{T_v}$ is the probability that an $\Exp(\irt)$ distributed infection clock clicks before the recovery clock of the vertex $v$ trying to infect its children runs out).  

As a consequence, the expected number of  children of $v$ 
with label $\ell$ in $\calT_k^{\bp(\Mgen)}(\infty)$ is $$\Msir_{k(v)\ell}=\E_{T_v}[p_{T_v}\Mgen_{k(v)\ell}]
=
\frac{ \irt}{\irt + \rrt}\Mgen_{k(v)\ell}.
$$
Under the assumption that $\Mgen$ is irreducible, standard branching process results then imply that the probability that the infection on
$\bp_k(\Mgen)$ survives forever is non-zero if and only if the largest eigenvalue\footnote{Note that in general, $\Mgen$ can have complex eigenvalues - but by the Perron-Frobenius Theorem, it has a single, non-degenerate eigenvalue which is positive and equal to the largest eigenvalue in absolute value., i.e., the spectral radius.  We trust that the use of  $\lambda_{\max}(\Msir)$ for this eigenvalue does not cause any confusion.},
$\lambda_{\max}(\Msir)$, is larger than $1$.  Defining
\begin{equation}
\label{pi-k-def}
\pi_k=\P\left(\left|\calT_k^{\bp(\Mgen)}(\infty)\right|=\infty\right)
\quad
\text{and}\quad
\Rn=\lambda_{\max}\left(\frac{ \irt}{\irt + \rrt} \Mgen\right)
\end{equation}
we therefore have that $\pi_k>0$ 
if and only if $\Rn>1$.
\end{remark}

\begin{lemma}
\label{lem:branching-process-coupling}
    Consider the SIR epidemic on $G\sim \psbm(\bm V,W)$ or $G\sim \sbm(\bm V,W)$, with the initial state of the infection  obeying the conditions from Assumption~\ref{ass:initial-SandI}, and assume that $V^I(0)$ consists of a single vertex $v$ with label $k$.  Let $\Mgen$ be the matrix with matrix elements
    $$
    \Mgen=W\diag(\s(0)).
    $$
        Then there exists a sequence $N_n\to\infty$ as $n\to\infty$ and a
coupling of 
$\calT_{G,v}(t)$
and $\calT_k^{\bp(\Mgen)}(t)$ such that 
    \[
        \P\left(\calT_{G,v}(t)         = \calT_k^{{\bp(\Mgen)}}(t) \text{ for all } t \text{ such that }|\calT_k^{\sbm}(t)| \leq N_n\right) \to 1 \text{ as } n\to\infty.
    \]
\end{lemma}




Lemma \ref{lem:branching-process-coupling} allows us to relate the initial behavior of the epidemic to the behavior of the  continuous time branching process $\calT_k^{\bp(\Mgen)}(t)$. 
In particular, if $\Rn\leq 1$, then the branching process $\calT_k^{\bp(\Mgen)}(t)$ dies out eventually, and any outbreak on the original graph will be of constant size, and if $\Rn>1$ then $\calT_k^{\bp(\Mgen)}(t)$ survives forever with probability $\pi_k$, and there will be an outbreak of size at least $N_n$ on the original graph with  asymptotic probability $\pi_k$.

In fact, it turns out that with high probability, 
the outbreak will be either of constant size, or of order $\Theta(n)$.
The next theorem makes this precise, and also gives the relative size of an outbreak if it occurs. 
More precisely, it shows that conditioned on there being a large outbreak, the probability that a random vertex in $V_k$ gets eventually infected is given by the survival probability of what is known as the backward branching process, see, e.g., \citep{barbour2013approximating, barbour:couplings-for-bp,diekmann00:math-epi}.

In our context, the backward branching process is the Poisson branching process $\bp_k(\Msir)$ with transition matrix
\begin{equation}\label{backward-def}
    \Msir =\frac{ \irt}{\irt + \rrt}W\diag(\s(0)),
\end{equation} 
and its survival probability is
\begin{equation}
\label{theta-k-def}
\theta_k=\P\left(|\bp_k(\Msir)|=\infty\right).
\end{equation}

\begin{theorem}
\label{thm:one-vertex-final-size}
    Assume that $W$ is irreducible and consider the SIR epidemic on $G \sim \sbm(\V, W)$ or $G \sim \psbm(\V, W)$ obeying Assumptions \ref{ass:initial-SandI} with $\i(0)=0$ and assume further that
        \[
       1-\prod_k (1-\pi_k)^{I_k(0)}\pto \pi
    \]
for some $\pi\in [0,1]$.  Then 
   $\frac{\I(0)+\S(0)-\S(\infty)}n$,
       converges in probability to a random vector which is equal to 
       $ \text{\rm diag}(\vec\theta)  \s(0)$ with probability $\pi$ and equal to $0$ with probability $1-\pi$.
\end{theorem}


Note that the theorem covers both the case where the vector $(I_k(0))_{k\in [K]}$ has a finite limit, and the case where $I(0)\to\infty$.   In the first case,  $\pi<1$, while in the second, we have that $\pi=1$ (provided $\Rn>1$), leading to an outbreak with probability tending to $1$ as $n\to\infty$.

As a consequence of Lemma \ref{lem:branching-process-coupling} and Theorem \ref{thm:one-vertex-final-size} we have the following corollary, which in particular implies that w.h.p.,  the epidemic starting from a single vertex either dies out after only infecting a finite number of vertices, or grows to size $\Theta(n)$.
\begin{corollary}
\label{cor:final-size}
     Consider the setting from Theorem~\ref{thm:one-vertex-final-size}, with $V^I(0)$ consisting of a
single infected vertex of label $k$, and let $N_n$ be an arbitrary sequence of integers with $N_n\to\infty$ and
$N_n/n\to 0$ as $n\to\infty$.
Then the following hold:

i) The probability that $S(0)-S(\infty)>N_n$ converges to $\pi_k$.

ii) If $\Rn\leq 1$, the final size $S(0)-S(\infty)$ converges in distribution to the size of $\calT_k^{ \sbm}(\infty)$, showing in particular that the limiting distribution of $S(0)-S(\infty)$
has exponentially decaying tails if $\Rn<1$.

iii) Assume $\Rn>1$.  When conditioned on being at most $N_n$, the final size converges in distribution to 
$|\calT_k^{ \sbm}(\infty)|$ conditional on extinction, showing in particular that the conditional final size distribution has exponentially decaying tails.

iv) If $\Rn>1$ and we condition on the event that the final size is at least $N_n$, then  $\frac{\S(0)-\S(\infty)}n$,
       converges in probability to 
       $ \text{\rm diag}(\vec\theta)  \s(0)$, showing in particular that with high probability, the final size of the infection is of order $\Theta(n)$ if we condition on it being at least $N_n$.
\end{corollary}




    

\begin{remark}
    Since $W$ is symmetric, even if it is not irreducible, it can be decomposed into isolated irreducible sub-components. As long as Assumption \ref{ass:initial-SandI} is true in each sub-component, then Theorem \ref{thm:lln-final-size} and Theorem \ref{thm:one-vertex-final-size} can be applied separately in each sub-component. 
\end{remark}

\subsection{Related Work}
\label{sec:related work}

In this section we place our results and methods in the context of the existing literature. As discussed in the introduction, for both $G(n, p)$ and the configuration model, the dynamics of the infection are primarily driven by a scalar-valued random variable determining the force of infection. For $G(n, p)$ the force of the infection is the same for all susceptible vertices and proportional to the number of active half-edges, while for the configuration model, the force of the infection on vertices of degree $d$  depends on $d$ and is proportional to $d$ times the number of active half-edges. Due to this fact, for $G(n, p)$ and the configuration model, a global time-rescaling can be applied to the stochastic epidemic such that the entire epidemic happens over the course of a finite time horizon. While our methods are inspired by \citep{neal2003sir, ball22:sir-rewire, janson14:lln-sir}, such a global time-rescaling cannot be applied in the case of the stochastic block model due to the fact that the force of infection is vector-valued rather than scalar-valued. Thus we develop new, fairly general, tools using the idea of herd immunity to derive the LLN for the final size in the stochastic block model. 

Another method for deriving the trajectory of the epidemic curve is to analyze the forward and backward branching processes of the epidemic. In \cite{barbour2013approximating}, the authors study the forward and backward branching processes in the case of graphs with bounded degree and apply their results to determine the epidemic curves for $G(n, p)$ and the Volz configuration model \cite{volz08:sir}. \cite{bhamidi2014front} generalize the approach of \cite{barbour2013approximating} to the setting of sparse random graphs with degrees having finite second moments under a slightly less general epidemic setting. It is plausible that this approach could be generalized to the stochastic block model using multi-type Crump-Mode-Jagers continuous time branching processes%
\footnote{We thank J\'ulia Komj\'athy for discussing this approach after one of us presented  the current work at a conference on random networks.}
to prove a LLN for the epidemic curve. However, to our knowledge this has not yet been done nor has it been used even on a heuristic level to derive the correct differential equations describing the LLN, while these fall out naturally from our approach.

There also exist several methods of analyzing the SIR trajectories using local graph approximations, provided the initially infected set of vertices is a constant proportion of all vertices chosen uniformly at random.
In \cite{cocomello2023exact}, the authors study SIR and SEIR models with (possibly) time-varying rates on locally tree-like graphs, and derive a law of large numbers result. In \cite{alimohammadi2023epidemic}, the authors propose a local algorithm for approximation of the epidemic trajectory and leverage local graph limit theory to show that the SIR evolution on a sequence of graphs with a local limit converges to the same process in the graph limit. However, in these approaches, it is not easy to connect the differential equation phase of the epidemic to sub-linear initial conditions due to the fact that at the stopping time where a fixed fraction of the vertices has been infected, the state of the epidemic is not that of an epidemic started at uniformly chosen set of infected vertices.

By contrast, our work allows  for general initial conditions of the epidemic (specifying the number of active half-edges emanating from each vertex separately), and in particular can be used to connect the initial phase starting from a sub-linear number of infected vertices to the phase where the differential equation approximation can be used to prove a LLN.  But of course, the results from  \cite{cocomello2023exact} and \cite{alimohammadi2023epidemic}, while not allowing to analyze the case with $o(n)$ initially infected vertices, are much more general as far as the underlying graph model is concerned, and in fact do not even require an underlying stochastic graph model as long as the sequence has a local limit.


\section{Preliminaries}\label{sec:actual-prelim}
\subsection{Poisson Stochastic Block Model Coupling}
\label{sec:coupling}

We will start with a simple lemma which will give us a coupling between the stochastic block model and the Poisson stochastic block model.  Using this coupling, our proof of Theorem~\ref{thm:LLN} for a graph drawn from $\psbm(\bm V,W)$ will imply Theorem~ \ref{thm:LLN} for the standard stochastic block model  $\sbm(\bm V,W)$.

\begin{lemma}
\label{lem:coupling}
Let $K $, $W$, $k(\cdot)$, and  $n_k$ be as in Definition~\ref{def:SBM}. Assume that $n>\max_{k,\ell}W_{k,\ell}$, and let $W'_{k(u)k(v)} = \frac{W_{k(u)k(v)}}{1-W_{k(u)k(v)}/n}$, and consider a multi-graph $G'\sim \psbm(\bm V,W')$. If we condition on $G'$ being simple, then the conditional distribution of $G'$ is given by 
$\sbm(\bm V, W)$. Furthermore, the probability that $G'$ is simple is  bounded away from $0$, uniformly in $n$.  
\end{lemma}

As a consequence, all statements holding with high probability for $G'\sim \psbm(\bm V, W')$ also hold with high probability for $G\sim \sbm(\bm V, W)$, reducing the proof of any convergence in probability statements for 
$G\sim \sbm(\bm n, W)$ to those for $G'\sim \psbm(\bm n,W')$.  Thus a LLN for the (easier to analyze) SIR epidemic on the Poisson stochastic block model will imply the same LLN on the standard stochastic block model.

\begin{proof}
  If $G'$ is a multi-graph with edge-multiplicities $A_{uv}$, then the probability of drawing $G'$ from 
$ \psbm(\bm V,W')$ is equal to
$$\prod_{v\leq u} \frac {p_{k(u)k(v)}^{A_{uv}}}{A_{uv}!}e^{-p_{k(u)k(v)}}.
$$
where $p_{kk'} = n^{-1} W'_{kk'}$.
As a consequence,
$$\P(G' \text{ simple}) =  \prod_{u < v}\left(e^{-p_{k(u)k(v)}} + e^{-p_{k(u)k(v)}}p_{k(u)k(v)}\right)\prod_{u=v}e^{-p_{k(u)k(v)}}. $$
Furthermore, for any simple graph $G=(V,E)$
 \[
       \P(G' = G \mid G' \text{ simple}) =
       \frac{\P(G' = G)}{\P(G' \text{ simple})}
    =\prod_{uv\in E}\frac{p_{k(u)k(v)}}{1+p_{k(u)k(v)}} = \frac{W_{k(u)k(v)}}{n},
    \]
showing that the conditional distribution of $G'$ is 
given by $\sbm(\bm V, W)$.
To lower bound $\P(G' \text{ simple})$ we use that
$(1+p)e^{-p}\geq e^{-p^2/2}$, a bound which follows, e.g., from the fact that the derivative of
$\log (1+p)-p+p^2/2$ is non-negative for all $p\geq 0$.
Setting $p=\max_{k,k'}p_{kk'}$, we therefore get
$$\P(G' \text{ simple}) \geq  e^{-{n\choose 2}\frac{p^2}2-np \geq e^{-\frac{(np)^2}4-np}} .$$ Since $pn$ is uniformly bounded away from $0$ as $n\to\infty$, this completes the proof.
\end{proof}

\begin{remark}
    While the coupling of the full graph $G\sim\sbm(\bm V,W)$ to $G'\sim\psbm(\bm V,W')$ requires moving from $W$ to the modified matrix $W'$ and conditioning on $G'$ to be simple, it is not hard to show that if we are only interested in a small neighborhood of a given set of vertices, we can couple  
    $G\sim\sbm(\bm V,W)$ to $G'\sim\psbm(\bm V,W)$ such that within the small neighborhood, $G$ and $G'$ are identical with high probability, \emph{without having to condition $G'$ or moving from $W$ to $W'$} (see Lemma~\ref{lem:coupling-psbm-bp-to-sbm-bp} below, where we consider the set of vertices infected in the early stage of the infection).
\end{remark}

\subsection{Coupling the SIR dynamic and the Poisson Stochastic Block Model}
\label{sec:exploration}

We study the SIR epidemic on  a graph $G$ 
drawn from  the Poisson stochastic block model
by dynamically revealing the edges of $G$. {This turns out to be much easier for the Poisson stochastic block model than the original stochastic block model, due mainly to the independence inherent in the Poisson degrees as expressed in Remark~\ref{rem:Poisson-degrees}.

We keep track of the exploration procedure by dynamically updating what one might call a decorated exploration forest on $V$, defined as a forest on $V$, augmented by a set of half-edges attached to the leaves of the trees.  Vertices will be marked infected or recovered, and we refer to the number of half-edges attached to a leaf $v$ as it's active degree (denoted by $d_v(t)$), and to the half-edges as active half-edges.

The exploration process starts with a set of roots, consisting of the initially infected set of vertices, $V^I(0)$, plus $d_v(0)\sim\Pois(D_{k(v)})$ half-edges attached to each $v\in V^I(0)$, chosen independently for all $v$.  Given the exploration forest at some time $t$, for all active half-edges we independently explore a half-edge at rate $\irt$ and independent from that, i.i.d. for all infected vertices, move an infected vertex to the recovered state at rate $\rrt$.  When an active half-edge emanating from a vertex $u$ is explored, we remove it,  choose a label $\ell$ according to $p_{k(u)\to\ell}$, and then choose a vertex $v$ uniformly from all vertices with label $k(v)=\ell$.  If this vertex is
susceptible, we infect it, add the edge $uv$ to the exploration tree, and add $\Pois(D_\ell)$ active half-edges to $v$.  When a vertex $v$ recovers, we remove all its active half-edges, i.e., we set $d_v(t)$ to $0$.}
\begin{lemma}\label{lem:SIR-SBM-Coupling}
The above defined exploration process and the SIR-process on a random multi-graph  $G\sim\psbm(\bm V,W)$ can be coupled in such a way that for all times $t$, the sets of susceptible, infected, and recovered vertices are the same in both, with the forest traced out by the exploration process at time $t$ being equal to the infection forest at time $t$.
Furthermore, an edge $uv$ in the infection forest (with $uv$ oriented away from the roots) will appear in the exploration forest precisely at the time $t'\leq t$ when $v$ was infected by $u$.
\end{lemma}

\begin{proof}
We start the proof with a cautionary remark: while the degree of any given vertex is given by $d_v\sim \Pois(D_{k(v)})$, it is not true that this holds independently for any set of vertices containing more than one vertex, since the degrees of two vertices $u$ and $v$ share the randomness of $A_{uv}=A_{vu}$.  

After this cautionary remark, we proceed with the actual proof.  Consider a vertex $u$ infected at time $t$, and let $V_+$ consist of all vertices not yet discovered at this point.  We can then couple $d_u(t)$ to the Poisson random variables $A_{uv}=A_{vu}$ defining the multi-graph $G$ plus a second, independent copy
$\wt A_{uv}$ as follows:
$$
d_u(t)=\sum_{v\in V_+} A_{uv}+\sum_{v\in V\setminus V_+}\wt A_{uv}.
$$
We will refer to the half-edges corresponding to the second sum as compromised, and the vertices they point to as fake copies of the original vertex $v$.
Furthermore, we will call a half-edge virulent at time $t$ if its endpoint has not yet been explored and the hidden endpoint of the edge corresponding to it lies in $S(t)$.  Finally, after the initial infection of $u$, whenever one of the half-edges corresponding to the first sum is explored, we will update the other half-edges pointing to the same end-point to be compromised as well. 

After this setup, the proof of the lemma is straightforward: by construction, virulent half-edges will be not compromised, and for all vertices $u$, the total rate at which one of the Poisson clocks of  virulent half-edges  emanating from  $u$ clicks is just the rate at which $u$ infects susceptible vertices.  Finally, when a non-compromised half-edge is explored, its endpoint is a real, not fake vertex, showing that the distribution of discovered, newly infected vertices is correct as well.
This concludes the proof.
\end{proof}

\begin{lemma}
\label{lem:coupling-psbm-bp-to-sbm-bp}
Let $N_n$ be a sequence such that $N_n/n\to 0$ as $n\to\infty$, and fix a set of initially susceptible, infected, and recovered vertices. Then the SIR-process on the two graphs $G\sim\sbm(\bm V,W)$ and $G'\sim \psbm(\V,W)$ 
can be coupled in such a way that with probability tending to $1$ as $n\to\infty$,
the sets of susceptible, infected, and recovered vertices are the same as long as
$|(V^I(t)\cup V^R(t))\setminus V^R(0)|\leq N_n$.
\end{lemma}
\begin{proof}
On a high level, the lemma can be shown to hold by using the exploration process from the proof of Lemma~\ref{lem:SIR-SBM-Coupling} for both models, and using that we are only exploring  $O(nN_n)$ possible edges to determine the course of the infection up to the time where $|V^I(t)\cup V^R(t)\setminus V^R(0)|=N_n$. The fact that the total variation distance between a $\Bern(W_{k(u)k(v)}/n)$ and a $\Pois(W_{k(u)k(v)}/n)$ random variable is $O(1/n^2)$ then implies the statement of the lemma.

Formally, we proceed as follows.  Draw,  independently for all oriented pairs  $(u,v)\in V\times V$, the following random variables
\begin{enumerate}
\item a sample $M_{uv}\sim\Bern(W_{k(u)k(v)}/n)$
\item a sample $M_{uv}'\sim \Pois(W_{k(u)k(v)}/n)$
\item an infinite sequence of i.i.d. infection times $T_{uv}^{(i)}\sim \Exp(\irt)$, $i=1,2,\dots$,
\end{enumerate}
with $M_{uv}'$ and $M_{uv}$ coupled in such a way that 
$M_{uv}=M_{uv}'$ with probability $1-O(1/n^2)$, and $T_{uv}^{(i)}$ drawn independently of the coupled random variables $M_{uv}'$ and $M_{uv}$.  Finally, draw independent samples of recovery times $T_u\sim \Exp(\rrt)$ for the vertices in $V$.

We then use the following coupling of the two processes:
starting from the set of initially infected vertices, we attach two sets of (oriented) active edges
to the vertices in $V^I(0)$: one red edge $uv$ to each of the vertices $v\in V$ with $M_{uv}=1$ and $M'_{uv}$ orange edges between $u$ and $v$ for all vertices $v$ with $M'_{uv}>0$.  We use the infection time $T_{uv}^{(1)}$ on the red edges $uv$ if $M_{uv}=1$, and the times
$T_{uv}^{(1)}, \dots ,T_{uv}^{(M'_{uv})}$ on the orange edges $uv$ if $M'_{uv}\geq 1$.  Using the random times $T_u$ for recovery for both processes, we see that the infections process on $G$ and $G'$ is identical until the first time when there exists a newly infected vertex $u$ and a vertex $v\in V$ such that $M_{uv}\neq M'_{uv}$.

The rest is straightforward: order the vertices in $V\setminus V^{I}(0)$ according to their infection times on $G$, and let $A_i$ be the event that all edge choices up to the point where $i$ got infected are such that the number of red and orange edges for each explored pair $uv$ is the same (including the edges out of $i$).  We then get that
$$
\P(A_i\mid A_{i-1})=\prod_{v\in V}\P(M_{iv}=M'_{iv})=(1-O(1/n^2))^n=1-O(1/n)
$$
with the initial probability of starting off with 
$M_{uv}=M'_{uv}$ for all edges out of $V^I(0)$ being
$$
\P(A_0)=\prod_{u\in V^I(0)}\prod_{v\in V}\P(M_{uv}=M'_{uv})=1-O(|V^I(0)|/n).
$$
Since $|V^I\cup V^R|$ is unaffected by recovery events, and grows by one precisely when a new vertex gets infected, we see that 
$$
\P(A_i)=1-O(\tilde N_i/n)
$$
where $\tilde N_i=|V^I(0)|+i$ is the size of $|V^I(t)\cup V^R(t)\setminus V^R(0)|$ when $i$ gets infected in $G$.  This implies the statement of the lemma.
\end{proof}

When analyzing the SIR epidemic on a graph
drawn from $\psbm(\bm V,W)$, we will henceforward work with the exploration process laid out at the beginning of this subsection, and will use 
$$
X_{k}(t)=\sum_{v:k(v)=k} d_v(t)
$$
to denote the number of active half-edges attached to vertices of type $k$ at time $t$. When proving Theorem~\ref{thm:LLN}, we will actually prove a LLN jointly for $S_k(t)/n$, $I_k(t)/n$ and
$X_k(t)/(D_k n)$, with $x_k(t)$ being the limit of $X_k(t)/(nD_k)$.  Note that conditioned on $I_k(0)$, the initial number of active half-edges $X_k(0)$ is equal in distribution to $\Pois(D_kI_k(0))$, so in particular, $X_k(0)/(D_k n)$ has the same expectation
as $I_k(0)/n$, motivating the initial condition
$x_k(0)=i_k(0)$ in the theorem.

To motivate the differential equations \eqref{eq:s_dfq_1}, \eqref{eq:x_dfq_1}, and \eqref{eq:i_dfq_1}, we note that at any given time
\begin{itemize}
    \item at rate $\rrt X_k$, we lose one
    active edge contributing to $X_k$ to recovery,
    \item at rate $\irt X_k$, we  explore one of the active half-edges contributing to $X_k$.
   \end{itemize}
When we explore an active half-edge contributing to $X_k$, $X_k$ decreases by $1$.  To determine the rate of increase of $X_k$, we note that when a half-edge labeled $\ell$ is explored, 
with probability $p_{\ell k}=W_{\ell k}n_k/(nD_\ell)$, we choose an endpoint with label $k$, and with probability $S_k/n_k$, this endpoint then hits a susceptible vertex, in which case $S_k$ decreases by $1$ and $X_k$ increases by  $\Pois(D_k)$.
    Therefore, at an overall rate of $\irt \frac{S_k}{n}\sum_\ell X_\ell \frac{W_{\ell k}}{D_\ell}$ we
    decrease $S_k$ by $1$ and increase $X_k$ by  $\Pois (D_k)$.
 Defining  
    $$
    \hat s_k(t) = \frac{S_k(t)}{n}, 
    \quad\hat x_k(t) = \frac{X_k(t)}{nD_k},\quad\text{and}\quad
    \hat i_k(t) = \frac{I_k(t)}{n},
    $$
we see that the expected rate of change for $\hat s_k$, $\hat x_k$ and $\hat i_k$, conditioned on the current value of these quantities, is
equal to
$$
-\irt \hat s_k(t)\sum_\ell \hat x_\ell(t) W_{\ell k},
\qquad
 -(\irt + \rrt) \hat x_k(t) + \irt \hat s_k(t) \sum_\ell \hat x_\ell(t) W_{\ell k}\quad\text{and}\quad -\rrt \hat i_k(t) + \irt \hat s_k(t) \sum_\ell \hat x_\ell(t) 
        {W_{\ell k}},$$
respectively.  Thus 
we expect that these quantities will obey a LLN given in terms of solutions to the differential equations
\eqref{eq:s_dfq_1}, \eqref{eq:x_dfq_1}, and \eqref{eq:i_dfq_1}.

\section{Law of Large Numbers on Bounded Time Intervals}
\label{sec:lln}


In this section, we prove Theorem~\ref{thm:LLN}. To this end, we first state and prove a version of the LLN that allows for more general initial conditions than those implicit in Theorem~\ref{thm:LLN}.  More specifically, we generalize the setting discussed in the last section, with $X_k(0)\sim \Pois(D_kI_k(0))$ to a setting where $X_k(0)$ and $I_k(0)$ can be specified separately.
We will later use this more general theorem to discuss an epidemic starting from a single infected vertex at time $0$.

\subsection{Formal statement for arbitrary initial conditions}\label{sec:lnn-for-gen-initial-cond}

We will write the differential equations \eqref{eq:s_dfq_1}, \eqref{eq:x_dfq_1} and \eqref{eq:i_dfq_1} in the compact form
$d\bm y_t/dt = \bm b(\bm y_t)$, where $\bm y_t\in \R_+^{3K}$ is the vector
$(s_1(t),\dots,s_K(t),i_1(t),\dots,i_K(t),x_1(t),\dots, x_K(t))$
 and

\begin{equation}\label{vector-field}
    b^{j}(\y) = \begin{cases}
     -\irt y^j {\sum_{k = 1}^{K}} y^{k+2K} W_{k j} & j = 1, \dots, K \\
     -\rrt y^j + \irt y^{j-K} \sum_{k = 1}^{K}y^{k+{2K}} W_{k (j-K)}& j = K+1, \dots, 2K \\
     -(\irt + \rrt)y^j + \irt y^{j-2K}{\sum_{k = 1}^{K}} y^{k{+2K} }W_{k (j-2K)}& j = 2K+1, \dots, 3K.
    \end{cases}
\end{equation}
We will consider solutions to these equations with initial conditions in the set
$$
U_0=\left\{\bm y\in [0,\infty)^{3K}\colon
\sum_{k=1}^K (y^k+y^{K+k}) \leq 1, ~ \sum_{k=1}^K( y^k+y^{k+2K})\leq 2\right\},
$$
which includes the initial conditions from Theorem~\ref{thm:LLN}, where we set $x_k(0)=i_k(0)$ leaving just the condition
$\sum_{k=1}^K (y^k+y^{K+k}) \leq 1$.

\begin{remark}
\label{rmk:si-sx-decreasing}
 Note that any solution of the differential equation $d \bm y_t/dt = \bm b(\bm y_t)$ will stay within $ U_0$ if 
$\bm y_0\in U_0$ by the fact that (i) $d(y^k_t+y^{k+K}_t)/dt\leq 0$ and $d(y^k_t+y^{k+2K}_t)/dt\leq 0$ for all $k\in [K]$, and (ii)
$dy^j_t/dt\geq 0$ if $y^j_t=0$.   
Using furthermore that the vector field $\bm b$ is Lipschitz in $U_0$, we see that for $\bm y_0\in U_0$,
the differential equation  $d \bm y_t/dt = \bm b(\bm y_t)$ has a unique solution $\bm y_t\in U_0$ for all $t\in [0,\infty)$.
\end{remark}

Our main theorem in this section gives a LLN for an arbitrary initial configuration of susceptible vertices, infected vertices, and active half-edges. 
We will formulate it for a deterministic initial configuration, and later use this theorem to discuss random initial conditions.  We use $\bm{\hat y}_t$ to denote the random vector
       $$\bm{\hat y}_t=\left(\frac{S_1(t)}{n}, \dots,\frac{S_K(t)}n, \frac{I_1(t)}{n},\dots,\frac{I_K(t)}{n},  \frac{X_1(t)}{nD_1}
    ,\dots, \frac{X_K(t)}{nD_K}\right).
    $$

\begin{theorem}[Law of Large Numbers for Finite Time]
\label{thm:lln}
 Let $L$ be the Lipschitz constant 
   of 
    the vector field \eqref{vector-field} on
$
 U=\left\{\bm y\in [0,\infty)^{3K}\colon
\|\bm y\|_2\leq 4\right\}
$, let $t_0 < \infty$, and choose $\delta$ such that
$0<\delta\leq \frac 13e^{-t_0L}$.
Choose $\bm y_0\in U_0$, and let
    $\bm y_t$ be the unique solution of 
    $$\frac {d \bm y_t}{dt} = \bm b(\bm y_t)$$
   with initial condition $\y_0$.
    Consider the coupling of the SIR epidemic on $\psbm(\bm V,W)$
    as defined in Section \ref{sec:prelim},
    starting from an arbitrary initial configuration at time $0$ obeying the conditions
    $$\max_{v\in [n]}d_v(0)\leq \log n
\quad\text{and}\quad
\|\bm y_0 - {\bm {\hat y}}_0\|_2\leq \delta.
$$
 If $\min_k D_k\geq D_0>0$, then  
    \begin{equation}\label{lnn-bd}
           \P\left( \sup_{t \in [0,t_0]} \left\| \bm{\hat y}_t-\bm{ y}_t\right\|_2 > 3\delta e^{t_0L}\right)\leq 
            \frac \zeta n\left(1+\frac{t_0(1+\log n)^2}{\delta^2}\right),
             \end{equation}
  where $\zeta$ is a constant depending on $\irt$, $\rrt$, $\|W\|_\infty$ and $D_0$.
   \end{theorem}

\subsection{Proof of Theorem~\ref{thm:lln}}

The proof of Theorem~\ref{thm:lln} is based on tools that are standard when proving laws of large numbers for continuous time Markov processes: Doob’s $L^2$-inequality and Gronwall’s lemma. We first introduce some notation, consider the continuous-time  Markov process defined via the exploration process introduced in Section~\ref{sec:exploration},
\begin{equation}\label{Ynt-def}
    Y^n(t) \defeq \left((V^S_k(t))_{k\in [K]}, (V^I_k(t))_{k\in[K]}, 
    (d_v^A(t))_{v \in[n]}\right),
\end{equation}
with $V^S_k(t), V^I_k(t), V^R_k(t)$ as in \ref{sec:model}, and $d_v^A(t)$ denoting the active degree of vertex $v$ at time $t$.
The state space of $\{Y^n(t): t\geq 0\}$ is given by
\begin{multline}
    \calE^n \defeq \bigg\{\left((V_k^S)_{k \in [K]}, (V_k^I)_{k \in [K]}, 
    (d_v^A)_{v \in V}\right) \colon\forall k,~ V_k^S, V_k^I\subseteq V_k\text{ and }
      V_k^S \cap V_k^I=\emptyset;\\
   \forall v \in V,~d_v^A \in \N_0 \text{ with }d_V^A=0\text{ if } v\notin \cup_{k\in [K]} V_k^I\bigg\},
\end{multline}
\noindent where $V_k \defeq \{v \in [n] : k(v) = k\}$. We denote the elements of $\calE^n$ by $\xi$,
and use $q(\xi,\xi')$ to denote that rate at which $\xi$ transitions to $\xi'$.
We  consider the following projection onto the coordinates for which we aim to establish the LLN,
\begin{align*}
    \y = (y^1, \dots, y^{3K}) : \calE^n &\to \R^{3K} \\
    \xi & \to \frac{1}{n}\left(S_1, \dots, S_K, I_1, \dots, I_K, \frac{X_1}{D_1}, \dots, \frac{X_K}{D_k} \right)
\end{align*}
where here we overload the notation using $S_k = |V_k^S|$, $I_k = |V_k^I|$ and $X_k = \sum_{v \in V_k^I}d_v^A$. In order to prove the LLN with the help of Doob's $L^2$-inequality, we need a bound on the variance of the jumps, 
\[
    \alpha(\xi) = \sum_{\xi' \neq \xi}|\y(\xi') - \y(\xi)|^2 q(\xi, \xi').
\]
\begin{lemma}\label{lem:variance-bd}
Fix $W$, $\rrt$ and $\irt$, and assume that $\min_kD_k\geq D_0$ for some $D_0>0$.  Then there exists a constant 
$\zeta_0=\zeta_0(\|W\|_\infty, D_0,  \rrt,\irt)$ such that 
$$
\alpha(\xi)\leq \frac{\zeta_0}n\left(1+\max_v d_v^A\right)^2 \text{ for all } \zeta\in \calE^n.
$$
\end{lemma}

\begin{proof}
We write $\alpha(\xi)$ as a sum of three contributions,
$$
\alpha(\xi)=\alpha_R(\xi)+\alpha_A(\xi)+\alpha_I(\xi),
$$
with the first term coming from  transitions
$\xi\to\xi'$ corresponding to the recovery of a vertex, the second comes from the exploration of an active half-edge which does not lead to the infection of a new vertex, and the last one comes from infection events.

   We start by analyzing the contribution to $\alpha_R(\xi)$.
   Vertices recover at rate $\rrt$. Each time a vertex recovers, the number of infected vertices decreases by $1$ and the number of active infected edges decreases by the active degree of the recovered vertex. Thus the contribution due to recoveries can
   be written as
     \begin{align*}
        \alpha_R(\xi) &= \rrt \sum_k\sum_{v \in V_k^I} \left(n^{-2} + \frac{(d_v^A)^2}{n^2D_k^2}\right)
        \leq \frac{\rrt}n\left(1+
        D_0^{-2}\max_v (d_v^A)^2\right).
         \end{align*}

Next we analyze transitions where a vertex gets explored, but no vertex gets infected, in which case all that happens is that one of the active degrees decreases by $1$.  
Since   
half-edges of type $\ell$ get explored at rate $\irt X_\ell$ and infect no vertex with probability
$$
1-\sum_{k=1}^K\frac {S_kW_{k\ell}}{nD_\ell}\leq 1,
$$
we can bound
$$
\alpha_A(\xi)
\leq \irt \sum_{\ell=1}^KX_\ell\frac 1{n^2D_\ell^{2}}
= 
\irt \sum_{\ell=1}^K\sum_{v\in V_\ell^I}d_v^A\frac 1{n^2D_\ell^{2}}
\leq\frac {\irt }{nD_0^2}\max_v d_v^A.
$$

Finally, consider the transitions where a vertex of type $\ell$ infects a vertex of type $k$.  In this case the number of infected vertices in community $k$ increases by $1$, the number of susceptible vertices in community $k$ decreases by $1$, and the number of free, infected, half-edges increases by $\Pois(D_k)-1$ if the explored half-edge was of the same type.  If it was of a different type, $\ell$ then $X_\ell$ decreases by $1$
and $X_k$ increases by $\Pois(D_k)$.  Thus the contributions
to $\alpha(\zeta)$ from infections can be written as
  \[
        \alpha_I = \sum_{k=1}^K \alpha_{I, k}.
    \]
    where
\[ 
    \alpha_{I, k} = \irt \frac{S_k}n\sum_{\ell}X_\ell\frac{W_{\ell k} }{D_\ell}\left(n^{-2}+n^{-2}+\frac 1{n^2D_k^2}\E[(Z_k-\delta_{k,\ell})^2 +(1-\delta_{k,\ell})]\right),
\]   
with $Z_k\sim \Pois(D_k)$.  Thus
\begin{align*}
    \alpha_{I, k} &\leq \irt n^{-2}\frac{S_k}n\sum_{\ell}X_\ell\frac{W_{\ell k} }{D_\ell}\left(2+\frac 1{D_k^2}\E[Z_k^2+1]\right)\\
    &=\irt n^{-2}\frac{S_k}n\sum_{\ell}X_\ell\frac{W_{\ell k} }{D_\ell}\left(3+D_k^{-1}+D_k^{-2}\right)\\
    & \leq \frac{\irt}n\frac{S_k}n(3D_0^{-1}+D_0^{-2}+D_0^{-3})\|W\|_\infty
 \max_vd_v^A
\end{align*} 
 As a consequence,
 \[
        \alpha_I  \leq \frac{\irt}n
        (3D_0^{-1}+D_0^{-2}+D_0^{-3})\|W\|_\infty
 \max_vd_v^A
    \]
\end{proof}


Next we prove the following lemma, which follows immediately from the construction in Section~\ref{sec:actual-prelim}.

\begin{lemma}\label{lem:Poisson-bounds}
  Consider the coupling of the SIR epidemic on $\psbm(\bm V,W)$
    as defined in Section \ref{sec:prelim}.
Then the following  statements hold with with  probability at least $1- \frac 1ne^{7\|W\|_\infty}$.

i) For an infection starting with a single, infected vertex $v$ with $d_v(0)\sim \Pois(D_{k(v)})$, 
  $$
  \sup_{t\geq 0}\max_{v\in [n]}d_v(t)\leq \log n.
  $$  

ii) For an arbitrary initial configuration
$$\sup_{t\geq 0}\max_{v\in [n]}d_v(t)\leq \max\left\{\max_{v\in [n]}d_v(0),\log n\right\}.
 $$ 
 \end{lemma}

\begin{proof}
i) Let $Z_v$, $v\in V$, be independent $\Pois(D_{k(v)})$ random variables.  
Recall the construction of the coupled process and note that $d_v(t)$ is decreasing in $t$, we see that 
$\max_{v:k(v)=k}d_v(t)
$
is stochastically dominated by $\max_{v\in [n]} Z_v$.
We know that
$$
\P(Z_v\geq \log n)=\P(e^{2Z_v}\geq n^2)\leq \frac 1{n^2}\E[e^{2Z_v}]
=\frac 1{n^2}e^{(e^2-1)D_{k(v)}}\leq \frac 1{n^2}e^{7\|W\|_\infty},
$$
where we used that $D_k\leq \|W\|_\infty$ for all $k$.
The result  follows by a union bound .

The proof of (ii) is a straightforward generalization of the proof of (i).
\end{proof}

\begin{proof}[Proof of Theorem~\ref{thm:lln}]

Given our preparations, this proof follows immediately from Theorem 4.1 in \cite{darling08:dfq-markov}, which in turn is based on Gronwall's inequality and Doob's $L^2$ inequality.
Recall the definition of $U$ from Theorem~\ref{thm:lln},
$$
 U=\left\{\bm y\in [0,\infty)^{3K}\colon
\|\bm y\|_2\leq 4\right\},
$$
and set $\eps =3\delta e^{t_0L}$.
Note that $U$  is a superset of $ U_0$ since $\bm y\in  U_0$
implies that $\|\bm y\|_1\leq 3$ and $\|\bm y\|_2\leq \|\bm y\|_1$.
Furthermore, we have that
\begin{equation}
\label{asmp:darling-boundary}
    \text{for all } t\in [0,\infty)
    \text{ and all } \xi \in \calE^n,~ \norm{\y(\xi) - \y_t}_2 \leq \eps \text{ implies that } \y(\xi) \in U.
\end{equation}
To see this, note that
$\bm y_0\in U_0$ implies that $\bm y_t\in U_0$, which in turn implies that $\|\bm y_t\|_2\leq \|\bm y_t\|_1\leq 3$.  The claim then follows from the triangle inequality and the fact that   $\eps\leq 1$ by our assumption that $\delta\leq \frac 13 e^{-t_0L}$.
%
This puts us in the setting of  \cite{darling08:dfq-markov}, Section 4.  

Theorem 4.1
 of \cite{darling08:dfq-markov} then states that for all $A\in \R_+$
 \begin{equation}\label{darlin-thm4.1-bd}
        \P\left( \sup_{t \leq t_0} \| \hat{\y}_t -  \y_t\|_2 > \eps \right) \leq \frac{4A t_0}{\delta^2 }+\P\left(\int_0^{T\wedge t_0} \alpha(Y^{n}(t))\; dt > At_0\right),
    \end{equation}
where $T := \inf\{t \geq 0: \hat{\y}_t \not \in U\}$. 
Set $$
A=\frac {\zeta_0}n(1+\log n)^2,
$$
where $\zeta_0$ is the constant from  Lemma~\ref{lem:variance-bd}.  
Combining the fact that $d_v(0)\leq \log n$
by the assumption of the theorem with the first statement of Lemma~\ref{lem:Poisson-bounds} (ii) and Lemma~\ref{lem:variance-bd}, we then have that
\begin{equation}
\label{eq:darling-2nd-moment-bound}
    \P\left(\int_0^{T\wedge t_0} \alpha(Y^{n}(t))\; dt > At_0\right)\leq \P\left(\sup_{t\in [0,t_0]}\max_{v\in [n]}d_v(t)\geq \log n\right) 
\leq \frac 1ne^{7\|W\|_\infty}
\end{equation}
This completes the proof.

\end{proof}

\begin{remark}
\label{rmk:drift-error-darling}
Theorem 4.1 of \cite{darling08:dfq-markov} covers the more general situation where the drift vector of 
the continuous time Markov-chain is only approximately equal to $\bm b(\hat {\y}_t)$.  To state this generalization, we assign a \emph{drift vector}
$\bm \beta(\xi) $ to each $\xi\in  \calE^n$ via
\[
    \bm \beta(\xi) = \sum_{\xi' \neq \xi}(\y(\xi') - \y(\xi)) q(\xi, \xi').
\]
In the setting of Theorem~\ref{thm:lln}, we had
$\bm \beta (Y^{n}(t))=\bm b(\hat{\y}_t)$ for all times $t$.  If this is violated, Theorem 4.1 of \cite{darling08:dfq-markov} gives an additional error term of
\begin{equation}
\label{eqn:starting-condition-error}
    \P\left(\int_0^{T\wedge t_0} \|\bm b(\hat{\y}_t)- 
\bm \beta(Y^{n}(t))\|_2
dt >\delta\right)
\end{equation}
on the right hand side of \eqref{darlin-thm4.1-bd}.  We will need this generalization in the proof of Theorem~\ref{thm:LLN}.
\end{remark}

\subsection{Proof of Theorem~\ref{thm:LLN}}



Recall that Theorem \ref{thm:lln} required that $\norm{\hat \y_0 - \y_0} \leq \delta$, while Theorem \ref{thm:LLN} allows for arbitrary 
initial conditions as long as they obey Assumption~\ref{ass:initial-SandI}.  To prove Theorem \ref{thm:LLN}, we therefore need to bound the probability 
\[
    \P\left(||\hat\y_0 - \y_0||_2 > \delta \right),
\]
where $\y_0$ is the initial conditions as defined in the statement of Theorem \ref{thm:LLN}.
Using the fact that $\x(0) = \i(0)$, we have
\begin{align*}
    \norm{ \hat\y_0 - \y_0}_2 &\leq \norm{\hat \s_0 - \s_0}_2 + \norm{\hat \i_0 - \i_0}_2 + \norm{\hat \x_0 - \i_0}_2 \\
    &\leq \norm{\hat \s_0 - \s_0}_2 + 2\norm{\hat \i_0 - \i_0}_2 + \norm{\hat \x_0 - \hat \i_0}_2.
\end{align*}
By Assumption~\ref{ass:initial-SandI}, the first two terms go to $0$ in probability, furthermore $\E[\hat\x_0] = \E[\hat \i_0]$, thus by standard concentration bounds, the last term therefore also goes to $0$ in probability. Showing that
for any $\delta>0$,
\begin{equation}
    \P\left(\norm{\hat \y_0 - \y_0}_2 > \delta \right) \to 0
    \quad\text{as}\quad n\to\infty.
\end{equation}
Next, as in the proof of Lemma~\ref{lem:Poisson-bounds}, 
\[ \P\left(\max_{v\in [n]}d_v(0)>\log n\right) \leq \frac 1ne^{7\|W\|_\infty},
 \]
which goes to $0$ as $n\to\infty$ as well. Finally, 
by Assumption \ref{ass:initial-SandI} \ref{asmp:init-s} it follows that $n_k/n$ is bounded away from $0$ uniformly in $n$. Combining that fact with the assumption that $\sum_\ell W_{k \ell} > 0$ for all $k \in [K]$, we know that
$\min_k D_k\geq D_0$ for some constant $D_0>0$ that does not depend on $n$.
Therefore, under the conditions of Theorem~\ref{thm:LLN}, those of Theorem \ref{thm:lln}
hold with probability tending to $1$ as $n\to\infty$,
reducing Theorem \ref{thm:LLN} to Theorem \ref{thm:lln} when $G \sim \psbm(\V, W)$.

To finish the proof of Theorem \ref{thm:LLN} when
$G \sim \sbm(\V, W)$, we use Lemma \ref{lem:coupling}, which states that
there exists a $W'$ such that the distribution of $G'\sim \psbm(\V, W')$ conditioned on being simple is that of $G$, and then apply 
 Theorem \ref{thm:lln}  to SIR epidemics on $G'$.  More precisely, we will use the extension described
 in Remark~\ref{rmk:drift-error-darling} to take into account that the drift vector ${\bm\beta}(\xi) $ of the continuous time Markov chain for the coupled SIR dynamics on $G'$ is given in terms of $W'$,
\begin{equation}
\footnotesize
    {\bm\beta}(\xi) = \left(\left(-\irt \frac{S_k}{n}\sum_\ell \frac{X_\ell}{n}\frac{W'_{\ell k}}{D_\ell}\right)_{k \in [K]}, \left(-\rrt \frac{I_k}{n} + \irt \frac{S_k}{n}\sum_\ell \frac{X_\ell}{n}\frac{W'_{\ell k}}{D_\ell}\right)_{k \in [K]}, \left(-(\irt + \rrt)\frac{X_k}{nD_k} + \irt \frac{S_k}{n}\sum_\ell \frac{X_\ell}{n} \frac{W'_{\ell k}}{D_\ell}\right)_{k \in [K]}\right),
\end{equation}
while we want to establish the law of large numbers for the differential equations given in terms of the vector field \eqref{vector-field} defined in terms of $W$.

%
We need an upper bound on the term in  \eqref{eqn:starting-condition-error}. Recall that $W'_{\ell k} = W_{\ell k}/(1- W_{\ell k}/n)$, and assume that $n\geq 2\|W\|_\infty$. For any $\xi \in \calE^n$ we have 
\begin{align*}
    \norm{\b(\hat \y(\xi)) - \bm \beta(\xi)}_2 &\leq \norm{\b(\hat \y(\xi)) - \bm \beta(\xi)}_1 
    \leq \sum_{k \in [K]}3\norm{\irt \hat s_k \sum_\ell \hat x_\ell\left(W'_{\ell k} - W_{\ell k}\right)}_1 \leq \frac{\zeta}{n}\norm{\sum_\ell \hat x_\ell}_1
\end{align*}
for some constant $\zeta$ depending on $\|W\|_\infty$ and $\irt$.
For $n$ sufficiently large, we therefore have that
\begin{align*}
    \P\left(\int_0^{T \wedge t_0} \norm{\b(\hat\y_t) - \bm \beta(Y^n(t))}_2 ~ dt > \delta \right) &\leq \P\left(\int_0^{T \wedge t_0} \frac{\zeta}{n}\norm{\hat \x(t)}_1 ~ dt > \delta\right) \\
    = \P\left(\frac{\zeta t_0}{n} \norm{\hat \x(t)}_1 > \delta\right) 
    &\leq \P\left(D_0^{-1}\sup_{t \geq 0}\max_{v \in [n]} d_v(t) > \frac{\delta n}{ \zeta t_0}\right) \\
    &\leq n^{-1}e^{7\norm{W'}_\infty}
   \leq n^{-1}e^{14\norm{W}_\infty}
\end{align*}
where the second to last inequality follows from a proof analogous to that of Lemma \ref{lem:Poisson-bounds}. 
This proves the desired law or large numbers for 
the epidemic on $G' \sim \psbm(\V, W')$.  Conditioning on $G'$ being simple, the statement of Theorem \ref{thm:LLN} for $G \sim \sbm(\V, W)$ follows with the help of Lemma \ref{lem:coupling}.

\section{Law of Large Numbers for the Final Size}
\label{sec:uniqueness-implicit-solution}

In this section, we will prove 
Theorem~\ref{thm:lln-final-size}, which gives a law of large numbers for the final state of the SIR epidemic on the stochastic block model.

\subsection{Solution to Differential Equations}

We start by establishing a few elementary facts about the solutions of the differential equations.  Throughout this section, $\bm b(\cdot)$
will be the vector field defined in \eqref{vector-field}.

\begin{lemma}
\label{lem:diff-eq-sol-cts-finite-t}
    For all finite $t$, the solution
    $\bm y_t$  to 
    the differential equation $d \bm y_t/{dt} = \bm b(\bm y_t)$
 is  continuous in the initial conditions $\y_0\in U_0$.
 
\end{lemma}
\begin{proof}
The proof is standard and follows immediately from Gronwall's inequality and the fact that $\bm b(\cdot)$ is Lipschitz continuous.
\end{proof}

\begin{lemma}\label{lem:dfq_limit}
Let $\bm y_t=(s_1(t),\dots,s_K(t),i_1(t),\dots,i_K(t),x_1(t),\dots, x_K(t))$ be the solution to the differential equation
$d \bm y_t/{dt} = \bm b(\bm y_t)$ with initial condition $\y_0\in U_0$.  Then
\begin{enumerate}[(a)]
    \item \label{cond:s1-limit} 
    As $t\uparrow \infty$, $s_k(t)\downarrow s_k(\infty)\geq 0$.
    \item \label{cond:x-limit} 
    $\lim_{t \to\infty} x_k(t) = 0$ 
    \item \label{cond:i-limit} 
    $\lim_{t \to\infty} i_k(t) = 0$
    \item\label{sinfty>0} $s_k(\infty)>0$ if and only if $s_k(0)>0$
    \item \label{x>0forallt} Assume that $W$ is irreducible, $s_k(0)>0$ for all $k$, and that $\vec x(0)\neq 0$.  Then 
    $$x_k(t)>0,\;i_k(t)>0\; \text{ and } \;\frac {ds_k(t)}{dt}<0
    $$
   for all $ 0<t<\infty$ and all $k\in [K]$.\end{enumerate}
\end{lemma}

\begin{proof}
The proof follows from a straightforward analysis of the differential equations \eqref{eq:s_dfq_1}, \eqref{eq:x_dfq_1}, and \eqref{eq:i_dfq_1} and is given in Appendix~\ref{app:diff-equ}.
\end{proof}


\begin{lemma} \label{lem:implicit-theta}
For all initial conditions $\y_0\in U_0$, the limit $\s(\infty)=\lim_{t\to\infty}\s(t)$ obeys the implicit equation
\begin{equation}\label{eq:implicit-eqn-w-starting-conditions}
       s_k(\infty) = s_k(0)\exp\left\{-\frac{\irt}{\irt + \rrt} \sum_\ell W_{\ell k}\Big(s_\ell(0)-s_\ell(\infty)+x_\ell(0)\Big)\right\}.
\end{equation}
\end{lemma}

\begin{proof}
Defining  $\chi_\ell(t)= \int_{0}^t x_\ell(\tau) d\tau$, we
can integrate \eqref{eq:s_dfq_1} to express $s_k(\infty)$ as
\begin{equation}
    s_k(\infty) = s_k(0)\exp\left\{-\irt \sum_\ell \chi_\ell(\infty)W_{\ell k}\right\}.
\end{equation}
Next we express the term $(\irt+\rrt)x_k(t)$ in \eqref{eq:x_dfq_1} as 
$(\irt+\rrt)\frac {d\chi_k(t)}{dt}$ and then use  \eqref{eq:s_dfq_1} and \eqref{eq:x_dfq_1} to conclude that the derivative of $s_k(t)+x_k(t)+(\irt+\rrt)\chi_k(t)$ is $0$, showing in particular that $s_k(\infty)+(\irt+\rrt)\chi_k(\infty)=s_k(0)+x_k(0)$.  Inserted into the above expression for $s_k(\infty)$ this proves the statement of the lemma.
\end{proof}
\begin{remark}
Assume that $s_k(0)>0$ for all $k$, and set $q_k=s_k(\infty)/s_k(0)$.
We can then rewrite the implicit equation \eqref{eq:implicit-eqn-w-starting-conditions} as an implicit equation for $q_k$,
 \begin{equation}\label{qk-with-x}
    q_k = e^{-a_k}\exp\left(-\sum_\ell \Msir_{\ell k}\left( 1 - q_\ell\right)\right)
    =e^{-a_k}G_k(\q)
\end{equation}
where  
$a_k=\frac \irt{\irt+\rrt}\sum_{\ell}W_{\ell k}x_\ell$, $\Msir= \frac{\irt}{\irt + \rrt} W\diag(\s(0))$, and  $G_k(\vec q)$ is the generating function for the off-spring distribution  of a Poisson multi-type branching process with mean matrix $\Msir$ starting from a root of type $k$. As we will see in the next section, for $x(0)>0$, the implicit equation \eqref{qk-with-x}
has a unique solution which is the survival probability of a certain ``backward'' branching process expressing the probability that a random vertex of type $k$ gets infected in the course of the infection.
\end{remark} 


\subsection{Backward Branching Process}
\label{sec:backward-BP}
In the previous section we derived the implicit equation \eqref{qk-with-x} by considering the infinite time limit of the solutions to the differential equations \eqref{eq:s_dfq_1}, \eqref{eq:x_dfq_1}, and \eqref{eq:i_dfq_1}.  As we will see, they can also be understood as an implicit equation for the survival probability of a certain branching process.  To motivate the definition of this branching process, we note that in order for an initially susceptible vertex to be eventually infected, it either has to be infected directly by a vertex infected at time $0$, or it has to be infected via a chain of infections going through initially susceptible vertices.  

Consider an active edge contribution to $X_\ell(0)$.  The probability that it directly infects a random vertex $v$ of label $k$ is $\frac\irt{\irt +\rrt}\frac{W_{\ell k}}{nD_\ell}$.  Treating these infection events as approximately independent, we see that the probability that $v$ is not directly infected is approximately equal to
$$
\prod_\ell\left(1-\frac\irt{\irt +\rrt}\frac{W_{\ell k}}{nD_\ell}\right)^{X_\ell}
\approx \exp\left(-\frac\irt{\irt +\rrt}\sum_\ell W_{\ell k}x_\ell\right) = e^{-a_k}.
$$
On the other hand, if a vertex $w$ that is connected to $v$ via an edge gets infected at some point during the course of the SIR epidemic, its probability to transmit the infection to $v$ is equal to $\frac \irt{\irt+\rrt}$.
Since $v$ has $\Pois(W_{k\ell}{\hat s}_\ell(0))\approx \Pois(W_{\ell k}s_\ell(0))$ many neighbors of of label $\ell$, who themselves can again be either directly of indirectly infected, we are lead to consider the following ``backward'' branching process.

Starting from a susceptible vertex $v$ of label $k$, we directly infect it with probability $1-e^{-a_k}$ (represented by a child of type $K+1$ in the backward branching process), while with probability $e^{-a_k}$ we give it $\Pois\left(\frac \irt{\irt+\rrt}W_{\ell k}s_\ell(0)\right)$ children of color $\ell$ which could transmit an infection to $v$.  In later generations, we proceed iteratively  with the same offspring distribution for parents of type $k\in [K]$, while parents which are infected stay infected (formally, with probability $1$, each parent of type $K+1$ has exactly one child of type $K+1$).  As argued above, the survival probability of the backward branching process starting from a root of type $k$ should then be asymptotically equal to the probability that a randomly chosen vertex from the vertices with label $k$ that are susceptible at time $0$ gets infected at some time during the epidemic.

While we don't derive this directly, we will prove it indirectly by showing that for $\sum_\ell a_\ell>0$ the implicit equation \eqref{qk-with-x} has a unique solution, and that this solution is the survival probability of the backward branching process starting from a  root  of type $k$, see Appendix~\ref{app:branching} for details.

\subsection{$\Reff$ and herd immunity}
\label{sec:herd-immunity}
To motivate the next definition and lemma, we note that the differential equation \eqref{eq:x_dfq_1} can be written in the form 
     \begin{equation}\label{dxdt-matrix-form}
        \frac{d\x(t)}{dt} = (\irt + \rrt)\x(t)(\Msir(t)-\1), 
    \end{equation}
where $\vec{x}(t) = (x_1(t), \dots, x_K(t))$, $\1$ is the $K\times K$ identity matrix and $\Msir(t)$ is the matrix 
    \[
        \Msir(t) = \frac{\irt}{\irt + \rrt} W\diag(\s(t)).
    \]
Using $\lambda_{\max}(\Msir(t))$ to denote the largest eigenvalue of the matrix $\Msir(t)$, we define 
    \begin{equation}
        \Reff(t) \defeq \lambda_{\max}( \Msir(t)).
    \end{equation}

\begin{lemma}\label{lem:dfq_subcritical}
$\Reff(t) $ is a continuous, non-increasing function of $t$.  If we assume that $W$ is irreducible, $s_k(0)>0$ for all $k$, and $\x(0)\neq \bm 0$, then $\Reff(t)$ is strictly monotone  and $\Reff(t)<1$ for all large enough $t$.
\end{lemma}

\begin{proof}
We start by proving monotonicity and continuity. Let  $\norm{\cdot}_F$ be the frobenius norm.
By Gelfand's formula for the spectral radius, $\Reff(t)=\lim_{m\to\infty}\norm{\Msir^m(t)}_F^{1/m} $.
Since $s_k(t)$ is weakly monotone decreasing in $t$ for all $k$, it follows that $(\Msir(t))_{\ell k} $ is weakly monotone as well. Furthermore, the elements of $\Msir(t)$ are non-negative for all $t$ so $\norm{\Msir^m(t)}_F$ is weakly monotone, implying the desired weak monotonicity of $\Reff(t)$.
To establish continuity, we use  continuity of eigenvalues as a function of the matrix entries, \cite{bhatia1990bounds}, combined with the continuity of $s_k(\cdot)$.

To prove strict monotonicity, we note that the eigenvalues of $\Msir(t)$ are equal to those of the symmetric  matrix $\Msir'(t)$ with entries
$(\Msir'(t))_{k\ell} = \frac{\irt}{\irt + \rrt} \sqrt{s_k(t)}W_{k\ell}\sqrt{s_\ell(t)}$, giving the alternative representation
$$
 \Reff(t)=\max_{\v}\v^T \Msir'(t)\v
$$
where the maximum goes over all normalized vectors $\v$, which by Perron-Frobenius can be chosen to have all positive entries.  Since $s_k(t)$ is strictly monotone by Lemma \ref{lem:dfq_limit} \ref{x>0forallt}, the entries of $\Msir'(t)$ are strictly monotone, which implies that $\Reff(t)$ is strictly monotone.

Finally, assume towards contradiction that $\Reff(t)\geq 1$ for all $t$.  By continuity, the same holds for the largest eigenvalue of $\Msir(\infty)$
, $\Reff(\infty)$.  Let $\v$ be the eigenvector of $\Msir(\infty)$ corresponding to the eigenvalue $\Reff(\infty)\geq 1$.  Again by Perron-Frobenius, we may assume that $\v$ has all positive entries.  Since the entries of $\Msir(t)$ are monotone decreasing in $t$, we have that $\Msir(t)\v\geq \Msir(\infty)\v\geq \v$ component-wise.  Inserted into \eqref{eq:x_dfq_1}, this implies
$
        \frac{d\vec{x}(t)}{dt}\cdot \vec{v} 
\geq 0
$  
for all $0\leq t<\infty$ which is a contradiction, 
as  $\x(0)\cdot \vec v> 0$ by the assumption that $\x(0)\neq 0$ and 
$\x(t)\cdot \v\to 0$ by Lemma~\ref{lem:dfq_limit}.
\end{proof}

\begin{remark}\label{rem:herd-immunity}
It is not hard to see that in the differential equation representation for the epidemic, $\Reff(t)=1$ is  the herd immunity threshold of the epidemic.  Indeed, if $\Reff(t_0)<1$ and $\v$ is the corresponding right eigenvector with all positive entries, then
$$
\frac{d\x(t)}{dt}\cdot\v\leq (\irt+\rrt)(\Reff(t_0)-1) \x(t)\cdot\v
\quad\text{for all}\quad t\geq t_0
$$
showing that $\x(t)\cdot \v$ and hence $\|\x(t)\|_1$ is decaying exponentially in $t-t_0$ for any initial condition $\x(t_0)$. As a consequence, a small infection at time $t_0$ will only have small overall impact on the final size of the infection.  By contrast, if $\Reff(t_0)>1$, then $\x(t)\cdot \v$ will grow exponentially in $t-t_0$ until  $\Reff(t)=1$.   Thus no matter how small $\x(t_0)$ is, the infection will always have a sizable effect on the final size, since the infection won't start to recede  before it has infected enough susceptible vertices to drive $\Reff(t)$ below one.
\end{remark}

In our subsequent proofs, it will be useful to have a version of Lemma~\ref{lem:dfq_subcritical} which gives a lower bound on $\Reff$ that holds uniformly over a suitable set of initial conditions. 

\begin{lemma}
\label{claim:reff_compactness}
    Consider an arbitrary compact set  $\calC$ of the intitial conditions such that for all initial conditions in  $\calC$,
    $\x(0)\neq 0$ and $s_k(0)>0$ for all $k$, and assume that $W$ is irreducible.
    Then there exists $\delta>0$ and $t_0<\infty$ such that $\Reff(t_0)<1-\delta$ for all initial conditions in $\calC$.
\end{lemma}

\begin{proof}
    Consider any starting conditions $(\tilde \s(0), \tilde \x(0)) \in \calC$. By  Lemma~\ref{lem:dfq_subcritical}, we can  choose $\tilde t_0<\infty$ and  $\tilde\delta>0$ such that $\Reff(\tilde t_0) < 1- 2\tilde\delta$. By Lemma \ref{lem:diff-eq-sol-cts-finite-t} and the fact that eigenvalues are continuous in the elements of the matrix $C(\tilde t)$, there exists an open ball $B_{\tilde r}$ around $(\tilde \s(0), \tilde \x(0))$ such that $\Reff(\tilde t_0) < 1-\tilde \delta$ for all initial conditions in $B_{\tilde r} \cap \calC$. This uncountable collection of balls $B_{\tilde r}$ indexed by starting conditions in $\calC$ gives an open cover of the set $\calC$ with the property that (i) each starting condition $(\s(0), \x(0)) \in \calC$ is in some ball $B_{\tilde r}$ and (ii) for any starting condition in $B_{\tilde r}$ we have that $\Reff(\tilde t_0) < 1-\tilde\eps$  for some $\tilde t_0<\infty$ and $\tilde \delta>0$. Since $\calC$ is compact, this open cover has a finite subcover, giving the desired result.
\end{proof}

We will also use a similar lemma for the difference between $\y(t)$ and its asymptotic value $\y(\infty)$.

\begin{lemma}
\label{claim:y_compactness}
   Let $W$ be irreducible and let $\calC$ be as in Lemma~\ref{claim:reff_compactness}.  Then there exists $\delta_1>0$ such that
   $\min_ks_k(\infty)\geq \delta_1$ for all initial conditions in $\calC$.  In addition, for all $\eps>0$ there exists $t_1<\infty$ such that for all initial conditions in $\calC$,
   $$
   \|\y_{t}-\y_\infty\|_2\leq\eps\quad\text{for all}\quad t\geq t_1.
   $$
  \end{lemma}
\begin{proof}
Let $t_0$ and $\delta$ be as in Lemma~\ref{claim:reff_compactness}, and let $C'(t)$ be the symmetric matrix defined in the proof of Lemma~\ref{lem:dfq_subcritical}.  For $t\geq t_0$, we then have
$$
\x(t)\leq \x(t_0)e^{(\irt+\rrt) (C(t_0)-1)(t-t_0)}=\x(t_0)A^{-1}e^{ (\irt+\rrt)(C'(t_0)-1)(t-t_0)}A
$$
where $A$ is the diagonal matrix with entries $\sqrt {s_k(t_0)}$.  Since $1\geq  {s_k(t_0)}>0$ for all initial conditions in $\calC$, continuity and compactness imply the existence of some $\delta'>0$ such that $1\geq  {s_k(t_0)}\geq\delta'$ for all initial conditions in $\calC$.  Combined with the fact that $\y_{t_0}\in U_0$ implies $\|\x(t_0)\|_2\leq\|\x(t_0)\|_1\leq 2$, we conclude that  for all $t\geq t_0$,
\begin{equation}\label{uniform-x-bound}
    \|\x(t)\|_2\leq \frac 2{\delta'}\|e^{(\irt+\rrt) (C'(t_0)-1)(t-t_0)}\|_{2\to 2}
=\frac 2{\delta'} e^{(\irt+\rrt)(\Reff(t_0)-1)(t-t_0)}\leq \frac 2{\delta'} e^{-(\irt+\rrt)\delta(t-t_0)},
\end{equation}
where $\|\cdot\|_{2\to 2}$ is the operator norm from $\ell_2$ to $\ell_2$.

Observe that by the above bound, $\int_{t_0}^\infty \|\x(t)\|_2 \, dt \leq \frac2{(\irt+\rrt)\delta\delta'}
$, which, by \eqref{eq:implicit-eqn-w-starting-conditions} and the lower bound on $s_k(t_0)$ shown earlier, in turn implies that
$s_k(\infty)$ is bounded away from zero uniformly for all initial conditions in $\calC$.  Furthermore, again by \eqref{eq:implicit-eqn-w-starting-conditions} and the fact that $1-e^{-x}\leq x$, we have that for all $t\geq t_0$
\begin{equation}\label{uniform-s-bound}
\|\s(t)-\s(\infty)\|_2\leq
\irt\|W\|_\infty \int_{t}^\infty\|\x(t)\|_2\, dt
\leq \frac{2\irt\|W\|_\infty }{(\irt+\rrt)\delta\delta'} e^{-(\irt+\rrt)\delta(t-t_0)}.   
\end{equation}

To complete the proof, we choose $\tilde t_1$ such that the right hand sides of
\eqref{uniform-x-bound} and \eqref{uniform-s-bound} are smaller than $\eps/4$ for all $t\geq \tilde t_1$ and then use the fact that $\i(t)\to 0$ for all initial conditions in $\calC$ and another continuity \& compactness argument to conclude that there exists a finite time $t_1\geq \tilde t_1$ such that $\|\i(t_1)\|_2\leq \eps/4$
for all initial conditions in $\calC$. Combined with the fact that $i_k(t)+s_k(t)$ is increasing in $t$, we infer that for $t\geq t_1$ we have
$$
\|\i(t)\|_2\leq \|\i(t_1)\|_2+\|\s(t_1)- \s(t)\|_2\leq\frac \eps 4 + \|\s(t_1)-\s(\infty)\|_2\leq \frac \eps 2
$$
to complete the proof.
\end{proof}

Next we introduce the analog of $\Reff$ for the actual SIR-infection on the Poisson stochastic block model.  In particular, given the state of the infection at time $t$ with $S_k(t)$ susceptible vertices with label $k$, we define
 \begin{equation}\label{C(t)}
        \hMsir(t) = \frac{\irt}{\irt + \rrt} W\diag(\hat\s(t))
        \quad\text{where}\quad \hat\s(t)=\frac 1n \S(t),
    \end{equation}
and set
    \begin{equation}
        \hReff(t) \defeq \lambda_{\max}(\hMsir(t)).
    \end{equation}
The threshold $ \hReff(t) =1$ can then be seen as the threshold for herd immunity for the (stochastic) SIR epidemic on the Poisson stochastic block model.
    
The following lemma shows that after passing this herd immunity threshold, a small initial outbreak (with a only a few active edges) will leave most of the vertices which are still susceptible untouched, while also not increasing the number of active half edges by more than a constant factor. It is a major ingredient in our proof of Theorem~\ref{thm:lln-final-size}.

\begin{lemma}\label{lem:herd-immune}
Consider an SIR-epidemic on $\psbm(\vec V, W)$ and assume that at time $t_0$ there are $X_k(t_0)$ active edges and $S_k(t_0)$ susceptible vertices of label $k$.  If $\hReff(t_0)<1$, then
$$
\E[\|\hat\s(t_0)-\hat\s(\infty)\|_2\mid \calF(t_0)]\leq  \frac 1{\min_k\sqrt{\hat s_k(t_0)}}\frac {\hReff(t_0)}{1-\hReff(t_0)} \|\hat \x(t_0)\|_2
$$
and
$$
\E\bigg[\sup_{t\geq t_0}\|\hat\x(t)\|_2\;\bigg|\:\calF(t_0)\bigg]\leq  \frac 1{\min_k\sqrt{\hat s_k(t_0)}}\frac {1}{1-\hReff(t_0)} \|\hat \x(t_0)\|_2,
$$
where the expectation is conditioned on the state of the epidemic at time $t_0$.
\end{lemma}

\begin{proof}
Recall the coupling of the epidemic and the Poisson stochastic block model from Section~\ref{sec:exploration}.  We will want to estimate the contribution from each active half-edge at time $t_0$ to the the difference $S_k(t_0)-S_k(\infty)$, which is nothing but the total number of newly infected vertices emanating from this half-edge after time $t_0$.  Assume that the half-edge in question has label $\ell$. With probability $\frac{\irt}{\irt+\rrt}$
the infection clock on this half-edge will click at some time $t$ before the recovery clock, at which point it will infect a vertex of label $\ell'$ with probability
$
\frac {W_{\ell\ell'}}{nD_\ell} S_{\ell'}(t),
$
for an overall probability of
$$
\frac{\irt}{\irt+\rrt}\frac {W_{\ell \ell'}}{nD_\ell} S_{\ell'}(t)\leq \frac 1{D_\ell}{\hMsir}_{\ell \ell'}(t_0)
$$
of infecting a vertex of label $\ell'$.  Each of these will give
$\Pois(D_\ell')$ new active half-edges of label $\ell'$, who in turn will each lead to a number of new infected vertices of label $\ell''$
with probability at most $ \frac 1{D_\ell'}{\hMsir}_{\ell' \ell''}(t_0)$.  Thus the expected number of infected vertices emanating from one of these  is bounded by $D_{\ell'}\frac 1{D_\ell'}{\hMsir}_{\ell' \ell''}(t_0)=
{\hMsir}_{\ell' \ell''}(t_0)$.  Continuing by induction, we see that in expectation, the total number of infected vertices of label $k$ in generation $m$, $m=1,\dots$ is bounded above
by
$$
\sum_\ell  \frac {X_\ell(t_0)}{D_\ell} ({\hMsir}(t_0)^{m})_{\ell k}.
$$
Summing up the contributions from all generations and observing that the total change in $S_k$ from time $t_0$ to the point where the infection dies out is just the total number of newly infected vertices after time $t_0$,
we see that in expectation we have the bound
$$
\E[\hat s_k(t_0)-\hat s_k(\infty)\mid \calF(t_0)]
\leq \sum_{m=1}^\infty (\hat \x(t_0){\hMsir}(t_0)^m)_k
=\left(\hat \x(t_0) \frac{{\hMsir}}{1-{\hMsir}}\right)_k
$$
Next we write $\hMsir$ as $\hMsir= {\widehat Q}^{-1} \widehat{M}\widehat Q$ where $\widehat Q$ is the diagonal matrix with entries
$\widehat Q_{kk}=\sqrt{\hat s_k(t_0)}$ and
$\widehat M=\frac{\irt}{\irt + \rrt} \widehat Q W\widehat Q$ to 
write $ \frac{{\hMsir}}{1-{\hMsir}}$ as
$\widehat Q^{-1}\frac{{\widehat M}}{1-{\widehat M}}\widehat Q$.  This gives the bound

\begin{align*}
    \E[\|\hat \s(t_0)-\hat \s(\infty)\|_2\mid \calF(t_0)]
    &\leq \E[\|\hat \s(t_0)-\hat \s(\infty)\|_1\mid \calF(t_0)]\\
&\leq \sum_k\left(\hat \x(t_0) \frac{{\hMsir}}{1-{\hMsir}}\right)_k
=\sum_k\left(\hat \x(t_0) \widehat Q^{-1}\frac{{\widehat M}}{1-{\widehat M}}\widehat Q\right)_k
\\
&\leq \|\widehat Q^{-1}\|_\infty \hat \x(t_0) \frac{{\widehat M}}{1-{\widehat M}}\v_M
\end{align*}
where $\v_M$ is the column vector with entries $M_k$.  Using the fact that $\widehat M$ is symmetric with the same eigenvalues as $\hMsir(t_0)$, we bound the operator norm of $\frac{\widehat M}{1-\widehat M}$ by $\frac{\hReff(t_0)}{1-\hReff(t_0)}$ to conclude that
\begin{align*}
    \E[\|\hat \s(t_0)-\hat \s(\infty)\mid \calF(t_0)\|_2&\leq
\|\widehat Q^{-1}\|_\infty \| \hat x(t_0)\|_2 \frac{\hReff(t_0)}{1-\hReff(t_0)}\|\v_M\|_2.
\end{align*}
Bounding 
$\|\v_M\|_2^2=\sum_k\hat s_k(t_0)\leq 1$
and observing that $\|\widehat Q^{-1}\|_\infty=\frac 1{\min_k\sqrt{\hat s_k(t_0)}} $ this completes the proof of the first inequality.

To prove the second inequality, we note 
$X_k(t)$ can be upper bounded by 
$X_k(t_0)$ plus the total number of active half-edges created after time $t_0$.  Proceed as in the proof of the bound on $\hat S(\infty)$, and denoting the number of new active half-edges of label $k$ generated in generation $m$ by $Z_{m,k}$, we get
$$
\E[Z_{m,k}\mid (Z_{m-1,\ell})_{\ell\in [K]}]\leq D_k\sum_{\ell}
\frac{\irt}{\irt+\rrt}Z_{m-1,\ell}\frac {W_{\ell k}}{nD_\ell}S_k(t_0)
=D_k \sum_{\ell}
\frac{Z_{m-1,\ell}}{D_\ell}\hMsir_{\ell k},
$$
which by induction implies that
$$
\E\Bigg[\frac{Z_{k,m}}{nD_k}\;\bigg|\;
\calF(t_0)
\bigg]
\leq \sum_\ell \frac{X_{\ell}(t_0)}{nD_\ell}(\hMsir^m)_{\ell k}
=\sum_\ell \hat x_\ell(t_0)(\hMsir^m)_{\ell k}.
$$
As a consequence,
$$
\E\Bigg[\sup_{t\geq t_0}\hat x_k(t)\;\bigg|\;
\calF(t_0)
\bigg]
\leq \sum_{m\geq 0}\sum_\ell \hat x_\ell(t_0)(\hMsir^m)_{\ell k}
=\left(\hat \x(t_0) \frac{1}{1-{\hMsir}}\right)_k
$$
Continuing as in the proof of the first bound, this implies the second bound of the lemma.
\end{proof}

\begin{remark}\label{rem:herd-immune} 
    Note that Lemma~\ref{lem:herd-immune} implies a similar statement for $I_k(t)$.  This follows immediately from the fact that both $S_k(t)$ and $S_k(t)+I_k(t)$ are non-increasing in $t$.  As a consequence, we get that for all $t\geq t_0$
    $$
    I_k(t)\leq I_k(t_0) +S_k(t_0)-S_k(t)\leq I_k(t_0) + S_k(t_0)-S_k(\infty),
    $$
    which together with the statements of Lemma~\ref{lem:herd-immune} implies that past the herd immunity threshold, a small initial outbreak (with a small number of initially active edges and infected vertices) will stay small for all $t\geq t_0$.
\end{remark}

\subsection{Final Size of the Epidemic}
\label{sec:final_size}

In this section we prove Theorem \ref{thm:lln-final-size}, which in particular implies a law of large numbers for the final size of the infection. 
As in Section~\ref{sec:lnn-for-gen-initial-cond},  we generalize the setting of Theorem \ref{thm:lln-final-size}, (with initial conditions which imply that $X_k(0)\sim \Pois(D_kI_k(0))$) to a setting where $X_k(0)$ and $I_k(0)$ can be specified separately, again using $\hat\x(t)$ to denote the state of the epidemic of the SIR model coupled to $G\sim\psbm(\V,W)$ as defined in Section~\ref{sec:exploration}.  
Theorem \ref{thm:lln-final-size} will follow in the same way as Theorem \ref{thm:LLN} followed from Theorem~\ref{thm:lln}. 

Note that in contrast to Theorem~\ref{thm:lln}, where all constants were independent of the initial conditions, the bounds in the next theorem are only  uniform over the initial conditions in a compact subset $\cal C\subset U_0$ obeying the condition
\begin{equation}\label{calC}
    \y(0)\in \calC \quad\Rightarrow\quad    \x(0)\neq 0\quad\text{and}\quad s_k(0)>0\quad\text{for all}\quad k\in [K].
\end{equation}



\begin{theorem}\label{thm:LLN-final-size}
Assume that $W$ is irreducible and $\max_kD_k\geq D_0$ for some $D_0>0$. Let $\cal C\subset U_0$ be a compact set obeying the conditions in \eqref{calC}, and let $\eps>0$.  Then there exists constants $\delta>0$ and $n_0<\infty$
such that if $\y_t$ is the solution of the differential equation $d \bm y_t/{dt} = \bm b(\bm y_t)$ with initial conditions $\y_0\in\calC$, then 
 \begin{equation}
       \P\left(\sup_{t\geq 0}\left\| \hat\y_t - \y_t\right\|_2\leq \eps\right)\geq 
     1- \eps
\end{equation}
 provided $n\geq n_0$ and  the initial configuration at time $0$ obeys the bound
    \begin{equation}
        \label{in-condition-final}
    \max_{v\in [n]}d_v(0)\leq \log n
\quad\text{and}\quad
\|\bm y_0 - {\bm {\hat y}}_0\|_2\leq \delta.
\end{equation}
\end{theorem}

\begin{proof}
The proof follows immediately from Theorem \ref{thm:lln}, the fact that $\y_t\to\y_{\infty}$ as $t\to\infty$, 
 and the fact that for $t$ large enough, the infection has passed the herd immunity threshold and thus dies out quickly, implying that
$\hat \y_t-\hat\y_{\infty}$ is small when $t$ is large. Lemmas~\ref{claim:reff_compactness}, ~\ref{claim:y_compactness},~\ref{lem:herd-immune} and
Remark~\ref{rem:herd-immune} provide the needed quantitative estimates.

Concretely, we first use Theorem \ref{thm:lln} to see that for all $t_0<\infty$, we can choose 
$n_0<\infty$ and $\delta>0$ such that if $n\geq n_0$, $\y_0\in U_0$  and the conditions in \eqref{in-condition-final}  hold, then
\begin{equation*}
      \P\bigg(\sup_{0\leq t\leq t_0}\left\| \hat\y_t - \y_t\right\|_2\leq \eps\bigg)\geq 1-\eps/4.
\end{equation*}
To prove Theorem~\ref{thm:LLN-final-size}, it is therefore enough to to show that 
\begin{equation}\label{lln-bd-t0}
     \P\bigg( \sup_{ t\geq t_0}\left\| \hat\y_t - \y_t\right\|_2\leq \eps\bigg)\geq 1-3\eps/4,
\end{equation}
uniformly for all $\y_0\in \calC$.
To this end, we use the triangle inequality to conclude that for $t\geq t_0$
 \begin{equation}\label{triangle-herd}
    \| \y_t - \hat\y_t\|_2\leq \| \y_t - \y_{t_0}\|_2 +  \| \y_{t_0} - \hat\y_{t_0}\|_2 + \| \hat\y_{t_0} - \hat \y_t\|_2,
\end{equation}
and then use the fact that
$\y_t$ converges as $t\to\infty$ to bound the first term,
Theorem \ref{thm:lln} to bound the second, and Lemma~\ref{lem:herd-immune} and Remark~\ref{rem:herd-immune} to bound the last.

The details are tedious but straightforward. First, by Lemma \ref{claim:reff_compactness}, there exists $\delta_0$ and $t_1 < \infty$ such that for all initial conditions in $\calC$, $\Reff(t_1) < 1-\delta_0$. By Lemma \ref{claim:y_compactness}, there exists $\delta_1 > 0$ such that $\min_k s_k(\infty) \geq \delta_1$, and we can choose $t_0 \geq t_1$ large enough such that for any starting condition in $\calC$,
    \[
        \|\x(t_0)\|_2\leq \frac{\eps^2}{64}\frac {1-\Reff(t_1)}{1+\Reff(t_1)} {\min_k\sqrt{s_k(t_0)}}, \quad
        \|\i(t_0)\|_2\leq \frac \eps 8
        \quad\text{and}\quad
        \|\y_\infty-\y_t\|_2\leq \frac\eps 8 \quad \text{for all }t \geq t_0.
    \]
    In particular, $\|\y_\infty - \y_t\|_2 \leq \eps/8$ for all $t \geq t_0$ implies that 
    \[
        \|\y_t - \y_{t_0}\|_2 \leq \|\y_t - \y_\infty\|_2 + \|\y_{t_0} - \y_\infty\|_2 \leq \frac{\eps}{4} \quad \text{ for all } t\geq t_0,
    \]
giving a uniform bound on the first term in \eqref{triangle-herd}.


Next we prove an upper bound on the right hand side of the bounds in Lemma~\ref{lem:herd-immune}.  To this end, we use Theorem \ref{thm:lln} to conclude that given any $\eps'>0$ we can choose $n_0$ large enough and $\delta$ small enough, such that for all $n\geq n_0$ and
under the condition \eqref{in-condition-final}, we have that with probability at least $1-\eps/4$
$$\|\hat \s(t_0)-\s(t_0)\|_2\leq\eps'\delta_1 \quad\text{and}\quad \|\hat\x(t_0)\|_2\leq \|\x(t_0)\|_2+\eps'.
$$
The first bound implies that
$\hat s_k(t_0)\leq (1+\eps') s_k(t_0)$ for all $k$, showing that
$\hMsir(t_0)\leq (1+\eps')\Msir(t_0)$
and thus $\hReff(t_0)\leq (1+\eps') \Reff(t_0)\leq (1+\eps')\Reff(t_1)$, in addition to the lower bound $\hat s_k(t_0)\geq (1-\eps')s_k(t_0)$.  By choosing $\eps'$ sufficiently small, we therefore have that with probability at least $1-\eps/4$
\begin{align*}
    \frac 1{\min_k\sqrt{\hat s_k(t_0)}}
    &\frac {1+\hReff(t_0)}{1-\hReff(t_0)} \|\hat \x(t_0)\|_2
\leq 
\frac 32\frac 1{\min_k\sqrt{s_k(t_0)}}\frac {1+\Reff(t_1)}{1-\Reff(t_1)} \|\hat \x(t_0)\|_2\\
&\leq \frac 32\frac 1{\min_k\sqrt{s_k(t_0)}}\frac {1+\Reff(t_1)}{1-\Reff(t_1)} \|\x(t_0)\|_2 +
\frac 32\frac 1{\sqrt{\delta_1}}\frac{2+\delta_0}{\delta_0}\eps'
\leq \frac 32\frac{\eps^2}{64}+\frac 12 \frac{\eps^2}{64}=\frac{\eps^2}{32}.
\end{align*}
Combined with Lemma~\ref{lem:herd-immune}, this implies that with probability at least  $1-\eps/4$.
\begin{equation}\label{expectation-bound-s-x}
\E \Big[\Big(
    \|\hat\s(t_0)-\hat\s(\infty)\|_2
    +\sup_{t\geq t_0}\|\hat \x(t)\|_2
    \Big)\Big | \calF(t_0)\Big]\leq  
    \frac{\eps^2}{32}.
\end{equation}
Applying Markov's inequality to the conditional expectation, and then using that the above bound holds with probability at least $1-\eps/4$, we see that with probability at least $ 1-\eps/2$,
$$
    \|\hat\s(t_0)-\hat\s(\infty)\|_2
    +\sup_{t\geq t_0}\|\hat \x(t)\|_2
   \leq \frac \eps 8.
$$
Combined with the fact that $\|\hat \s(t_0) - \hat\s(t)\|_2$  is monotone in $t$ and the bound from Remark~\ref{rem:herd-immune}, we get that with probability at
least $1-\eps/2,$
\begin{align*}
\|
\hat\y_{t_0} - \hat\y_t\|_2&\leq \| \hat\s({t_0})-\hat\s(t)\|_2+\|\hat \x(t_0)\|_2+\|\hat \x(t)\|_2
+\|\hat \i(t_0)\|_2+\|\hat \i(t)\|_2
\\
&\leq
2\| \hat\s({t_0})-\hat\s(\infty)\|_2+\|\hat \x(t_0)\|_2+\|\hat \x(t)\|_2
+2\|\hat \i(t_0)\|_2
\\
&\leq \frac \eps 4 +2\|\i(t_0)\|_2+2\|\hat \i(t_0)- \i(t_0)\|_2
\leq \frac\eps 2 +2\|\hat \i(t_0)- \i(t_0)\|_2.
\end{align*}
Inserting into \eqref{triangle-herd} and using that $\|\y_t-\y_{t_0}\|_2\leq \frac\eps 4$ for all $t\geq t_0$, this gives that with probability at least $1- \eps /2$
$$
\sup_{t\geq t_0}\|\y_t-\hat\y_t\|_2\leq \frac {3\eps}4 +  \| \y_{t_0} - \hat\y_{t_0}\|_2
+2\|\hat \i(t_0)- \i(t_0)\|_2
$$
Using  Theorem \ref{thm:lln} a second time (and adjusting $n_0$ and $\delta$ if needed), the remaining terms on the right can be made smaller than $\eps/4$ with probability at least $1-\eps/4$, completing the proof.
\end{proof}

\section{Law of Large Numbers for $o(n)$ Initially Infected Vertices}
\label{sec:o-n}

In this section we prove Theorem \ref{thm:one-vertex-final-size}, the law of large numbers for the final size starting from one (or $o(n)$) many initially infected vertices. In order to do so, we use a branching process approximation and martingale analysis for the initial phase of the epidemic to show that with high probability the epidemic either dies out or reaches size ${\Theta(n)}$. This puts us in the setting of Theorem \ref{thm:LLN}, which we then use alongside results from Section~\ref{sec:uniqueness-implicit-solution} to prove the LLN result.
\subsection{Branching Process Approximation for the Initial Phase}
\label{sec:infect-tree}

The infection tree of the initial phase of the epidemic on the Poisson SBM starting from one, randomly seeded, initially infected node is well-approximated by a continuous time branching process. Consider the following dynamics:
\begin{itemize}
    \item The branching process starts with one infected node $v_0$ chosen randomly from all nodes of label $k$ and $\Pois(D_{k})$ active half-edges attached to it. 
      \item Infected nodes infect down active half-edges at rate $\irt$, independently for all active half-edges. When a half-edge emanating from a vertex $u$ of label $k'$ passes down the infection, a new infected vertex $v$ with label $\ell$ chosen with probability $p_{k'\to \ell} = \frac 1{D_k}W_{k'\ell}\frac{n_\ell}{n}$ is born, and the active half-edge gets replaced by an (inactive) full edge $uv$.  Finally, $v$ is given $\Pois(D_\ell)$ new active half-edges.  
      
       \item Infected vertices recover at rate $\rrt$. Once a vertex recovers, all active half-edges emanating from this vertex get deleted, while the inactive (full) edges remain. 
\end{itemize}

Observe that the tree generated from the dynamics above, represented by vertices labelled by group membership and state at time $t$ (infected or recovered) and inactive edges between such vertices is equal to $\calT_k^{\bp(\hMgen)}(t)$, 
with 
$$
\hMgen=W\diag(\vec n/n).
$$
Throughout this section,  we will
usually omit the reference to $\hMgen$ and just use the notation $\calT_k^{\bp}(t)$.  We will further slightly abuse notation by including the currently active half-edges as part of the tree $\calT_k^{\bp}(t)$, while in Section~\ref{sec:prelim} these were not included.


\label{sec:branching_process}

Let $\calT_k^{\bp}(t)$ be the tree generated by the dynamics described above at time $t$, represented as a tree with vertices labelled by group membership and the state at time $t$ (infected or recovered), and edges labelled as active or inactive. Let $\calT_k^{\psbm}(t)$ be the SIR infection tree starting from a vertex in community $k$ on $G \sim \psbm(\V, W)$, again including the active half-edges.

We can construct the process $\calT_k^{\psbm}(\cdot)$ from the process $\calT_k^{\bp}(\cdot)$ as follows: generate a sample $(\calT_k^{\bp}(t))_{t\geq 0}$ from the coupling described at the beginning of this section, but each time an infection passes along an edge $uv$ in $\calT_k^{\psbm}(\cdot)$,
delete the edge with probability $h_\ell(t^-)$, where $h_\ell(t^-) := 1-S_\ell(t^-)/n_\ell$, $t$ is the time the edge transitions from active to inactive in $\calT_k^{\bp}(\cdot)$, $t^-$ indicates a limit from the left, and $\ell$  is the label of $v$ in $\calT_k^{\bp}(t)$.
This takes into account the fact that edges only appear in the SIR infection tree when the endpoint is connected to a susceptible individual, and couples the infection tree of PSBM, $\mathcal T_{k}^{\psbm}(t)$, to the branching process tree $\calT_k^{\bp}(t)$ in such a way that
$\calT_{k}^{\psbm}(t)$ is a sub tree of $\calT_k^{\bp}(t)$ for all $t\in [0,\infty)$. 
Since $S_\ell(t)$ decreases by $1$ precisely when a new vertex with label $\ell$ is added to $\calT_{k}^{\psbm}(t)$, we have that $h_\ell(t)$ is equal to $\frac 1{n_\ell}$ times the number of vertices of label $\ell$ in $\calT_{k}^{\psbm}(t)$, showing in particular that $\calT_{k}^{\psbm}(t)$ is a Markov Process.  

The above construction implies that
$|\calT_{k}^{\psbm}(t)| \preceq |\calT_k^{\bp}(t)|$, where $\preceq$ means stochastic domination in the usual sense. It also implies that whenever
$|\calT_{k}^{\psbm}(t)|=|\calT_k^{\bp}(t)|$, then $|\calT_{k}^{\psbm}(t')|=|\calT_k^{\bp}(t')|$, and in fact
$\calT_{k}^{\psbm}(t')=\calT_k^\bp(t')$,
for all $t'\leq t$. 


\begin{lemma}
\label{lem:exact-coupling-sir-bp}
Let $N_n$ be such that $N_n/\sqrt n\to 0$ and assume that $\liminf_{n\to\infty} \frac{n_k}n>0$ for all $k$.
   Under the above coupling
    \[
      \mathbb P\left( \calT_{k}^{\psbm}(t)=\calT_k^\bp(t) \text{ for all $t$ such that }
     |\calT_k^{\bp}(t)| \leq N_n\right)\to 1\text{ as }n\to\infty.
    \]

\end{lemma}

\begin{proof}
Order the vertices in the branching process 
$\mathcal \calT_k^\bp(\cdot)$ by their time of arrival, with $1$ being the label of the root, $2$ the label of the next vertex to arrive, etc. Let $t_i$ be the arrival time of vertex $i$,
$k(i)$ its community label, $v_i$ be the label of its parent in $\mathcal \calT_k^\bp(\cdot)$, and 
$f_\ell(i)$ be the number of vertices of group label $\ell$ in 
$\calT_k^{\bp}(\cdot)$ that have arrived \textbf{before} vertex $i$ (not 
including $i$).

We can imagine generating the SIR tree from $\calT_k^{\bp}(\cdot)$ as follows: when vertex $i$ gets born in $\calT_k^{\bp}(\cdot)$, leading to a new edge $(v_i,i)$ in $\calT_k^{\bp}(\cdot)$, we assign a random variable $Z_i\sim \Bern(1-h_{k(i)}(t_{i-1}))$
to the edge $(v_i,i)$.
If 
$Z_i=1$, then we keep the edge in the SIR tree, otherwise we remove the edge. More precisely, we obtain the tree $\calT_k^{\psbm}(t)$ by taking the connected component of the root obtained once we remove all edges with $Z_i=0$.
Our goal is to prove that
$$
\P\left(Z_i=1\text{ for all }i=1,\dots N_n\right)\to 1\text{ as }n\to\infty,
$$
since that implies the statement of the lemma.

Note that in general, the random variables $Z_i$ are not independent, even when conditioned on $\calT_k^{\bp}(\cdot)$, since $h_{k(i)}(t_{i-1})$
depends on the tree $\calT_k^{\psbm}(t_{i-1})$, which in turn depends on the previous variables $Z_1$, $\dots$, $Z_{i-1}$:
$$
h_{k(i)}(t_{i-1})=\frac 1{n_{k(i)}}\left(f_{k(i)}(i)
-\sum_{j=1}^{i-1}(1-Z_j)\1_{k(j)=k(i)}
\right).
$$
However, if we condition on $\calT_k^{\bp}(\cdot)$ and the event
$Z_1=Z_2=\dots=Z_{i-1}=1$, then $h_{k(i)}(t_{i-1})=f_{k(i)}(i)/n_{k(i)}$ 
and the conditional probability that $Z_i=1$ becomes

    \[
    1-
    f_{k(i)}(i)/n_{k(i)}\geq 1-(i-1)/n_{k(i)}.
    \]
Rewriting the probability of the event that $Z_1=Z_2=\dots=Z_{N_n}=1$
as a product of conditional probabilities, we therefore get that
    \begin{align*}
        \P\left(\prod_{i=1}^{N_n}Z_i=1\,\bigg|\, \mathcal{T}_k^{\bp}(\cdot)\right)& = \prod_{i=1}^{N_n}\left(1-\frac{f_{k(i)}(i)}{n_{k(i)}}\right) 
            \geq \prod_{i=1}^{N_n}\left(1 - \frac{i-1}{n_{k(i)}}\right) 
        \geq 1 - \sum_{i=1}^{N_n} \frac{i-1}{n_{k(i)}} \\
        &\geq 1 - \frac{N_n(N_n - 1)}{\min_kn_k} = 1 - O(N_n^2/n).
    \end{align*}
This implies 
\begin{align*}
\P\left(\prod_{i=1}^{N_n}Z_i=1\right)
= \E_{\mathcal{B}_k(\cdot)}
\left[ 
    \P\left(\prod_{i=1}^{N_n}Z_i=1\,\bigg|\,\mathcal{T}^{\bp}_k(\cdot)\right)
\right] 
 \geq 1 - \frac{N_n(N_n - 1)}{\min_kn_k}=1-O(N_n^2/n),
    \end{align*}
    proving the lemma.
\end{proof}

Lemma~\ref{lem:branching-process-coupling} is an immediate corollary of 
Lemma \ref{lem:coupling-psbm-bp-to-sbm-bp}
and
Lemma~\ref{lem:exact-coupling-sir-bp}.
\begin{proof}[Proof of Lemma~\ref{lem:branching-process-coupling}]
In view of Lemma \ref{lem:coupling-psbm-bp-to-sbm-bp}
and
Lemma~\ref{lem:exact-coupling-sir-bp}, all we need to do is to couple the two process 
$\calT_k^{\bp(\Mgen)}(t)$ and 
$\calT_k^{\bp(\hMgen)}(t)$
with
$
\Mgen=W\diag(\s)
$
and
$
\hMgen=W\diag(\vec n/n)
$
in such a way that they are equal with high probability as long as, say, the first one has not too many vertices.

Formally, we proceed as follows: we label vertices in 
the two processes by their order of arrival in the infection tree, with both of them starting at vertex $v=1$ at time $t=0$.  We then 
iteratively couple the two processes as follows: if up to the time $t_i$ when vertex $i$ got infected in $\calT_k^{\bp(\Mgen)}(\cdot)$ the two processes are identical, then we use the same random variable $T_i\sim \Exp(\rrt)$ for the recovery time of $i$. Furthermore, we
couple the set of children of $i$ in $\bp(\Mgen)$ and $\bp(\hMgen)$ by optimally coupling the random variables $\Pois(\Mgen_{k(i)\ell})$  and $\Pois(\hMgen_{k(i)\ell})$ such that they are equal for all $\ell$ with probability $1-O(\eps_n)$, where $\eps_n=\|\s-\frac{\n}n\|_1$.
If they are equal, we use the same random variables to determine the infection times along the edges emanating from $i$ for both processes; if they are not, we run the two process independently from there on.

Let $A_i$ be the event that up to the time when vertex $i$ in $\calT_k^{\bp(\Mgen)}(\cdot)$ gets infected, both process have identical histories.  Then
$
\P(A_{i+1}\mid A_i)=1-O(\eps_n),
$
showing that 
$$\P(A_i)=1-O(i\eps_n).
$$
To complete the proof, we use 
Assumption~\ref{ass:initial-SandI}, to find a deterministic sequence $N_n\to\infty$ with $n\to\infty$ such that  $\eps_nN_n\to 0$ in probability. This implies that
$$
\P(A_{N_n})\to 1 \quad\text{in probability as}\quad n\to\infty.
$$
(Note $\P(A_{N_n})$ is random due to the possibly random starting conditions allowed by Assumption~\ref{ass:initial-SandI}).  Together with Lemma \ref{lem:coupling-psbm-bp-to-sbm-bp}
and
Lemma~\ref{lem:exact-coupling-sir-bp}, this proves Lemma~\ref{lem:branching-process-coupling}.
\end{proof}
\begin{remark}
\label{rem:many-initial-vertices-coupling}
 Observe that the coupling described at the beginning of this section and the results in Lemma \ref{lem:exact-coupling-sir-bp} and Lemma~\ref{lem:branching-process-coupling} are robust to the starting conditions of the epidemic. Suppose that are $\zeta = O(1)$ many initially infected vertices. Then couple the epidemic to a forest of branching processes in the following way: start a branching process $\calT^{\bp}_{k(v)}(t)$ from each initially infected vertex $v \in \zeta$. Then proceed as described in the beginning of this section, removing edges with probability $h_\ell(t^-)$ each time an infection event occurs into community $\ell$. It is clear that Lemma \ref{lem:exact-coupling-sir-bp} immediately generalizes and the epidemic and forest are exactly equal with high probability until the size of the epidemic is of order $N_n$.
\end{remark}

Next, we will show that either the infection dies out, or there is a "large outbreak" of the SIR epidemic on the PSBM -- in which case a constant fraction of the vertices are infected, and the differential equation approximation proved in Section \ref{sec:lln} applies. To this end, it will be convenient to switch from the continuous time Markov chain defining the SIR epidemic to a discrete time Markov chain.



\subsection{Martingale Bounds and Discrete Time Markov Chain}
\label{sec:discrete-time-chain}
In this section we will prove the following theorem using a submartingale concentration bound for a suitably defined discrete time Markov chain.
Throughout this subsection we will assume that $W$ is irreducible, and that the initial conditions satisfy Assumption \ref{ass:initial-SandI}.

\begin{theorem}
\label{thm:bounds_on_x}
 Consider the SIR epidemic on $G \sim \sbm(\V, W)$ or $G \sim \psbm(\V, W)$,
 and let
  $$\tau_\eps=\min\{t:S(t) \leq S(0) -\lceil \eps S(0)\rceil\}.$$
  If $\Rn>1$, then there exist constants $\eps_0>0$ and $0<\zeta_1<\zeta_2<\infty$ such that for $0<\eps\leq \eps_0$, the following statements hold

    \begin{enumerate}[(a)]
    \item \label{unbdd_initial} if $I(0) \pto \infty$ and $I(0)/n \pto 0$ as $n\to\infty$, then with probability tending to $1$ as $n\to\infty$
    \[
     \tau_\eps < \infty\quad\text{and}\quad \zeta_1 \eps S(0) \leq \|\X(\tau_\eps)\|_1 \leq \zeta_2\eps S(0).
    \] 
    \item \label{bdd_initial} if $I(0)$ is bounded in probability and $N_n \pto \infty$, then with probability tending to $1$ as $n\to\infty$
    \[
      \exists t \text{ s.th. } S(0) - S(t) \geq N_n \quad\Rightarrow\quad\tau_\eps < \infty\quad\text{and}\quad \zeta_1 \eps S(0) \leq \|\X(\tau_\eps)\|_1 \leq \zeta_2 \eps S(0).
    \]
\end{enumerate}
\end{theorem}

 To define the discrete time Markov chain used in the proof of the theorem, we observe that the continuous time SIR epidemic on the Poisson stochastic block model is a Markov Chain where elements, $\xi$, of the state space $\Omega$ can be represented as follows:
    \[
        \xi = \{S_k, I_k^d, R_k;~ k \in [K], d \in \N\}.
    \]
    $S_k, R_k$ are the number of susceptible and recovered nodes in community $k$, respectively, and $I_k^d$ is the number of infected nodes in community $k$ with active degree $d$. First, we make a few observations. Notice that $I_k = \sum_{d=0}^\infty I_k^d$ is the number of infected nodes in community $k$. Also, $X_k = \sum_{d=0}^\infty d I_k^d$ is the number of active edges in community $k$. Since infected nodes with active degree zero have no bearing on the course of the epidemic, for our purposes, it will be useful to work with the following 
    state space that omits such vertices:
    \[
        \tilde \xi = \{S_k, A_k^d;~ k \in [K], d\in \N\}
    \]
    where $A_k^0 = R_k + I_k^0$ and $A_k^d = I_k^d$ for all $d \geq 1$. We now construct a discrete time Markov Chain based on the jumps of the continuous time Markov Chain. Let $A_k = \sum_{d\geq1} A_k^d$, 
    $A = \sum_k A_k$,
    $X_k = \sum_d d 
    A_k^d$ and 
    $ X = \sum_k X_k$.

    For this section, let $T \in \N$ denote our discrete time steps, with $T$ incrementing by $1$ whenever
    the state of the continuous time chain changes.
    The rate at which this happens is then equal to
        $\Gamma(t) = \rrt A(t) + \irt X(t)$, 
        where $t$ denotes the time  in the continuous-time epidemic. We use $(S_k(T), A_k^d(T))_{k, d}$ to represent the state of the chain at time $T$. 
We start the chain at time $T=0$ with
    one initially infected vertex of label $k_0$ and active degree $d_0\sim \Pois(D_{k_0})$, i.e.,  $S_k(0)=n_k-\delta_{kk_0}$ and $A_{k}^{d}(0)=\delta_{kk_0}\delta_{dd_0}$ where $\delta_{nm}$ is the Kronecker-Delta.
    
Whenever $T$ increments in the discrete time chain, it corresponds to one of the following events:
    \begin{itemize}
        \item with probability $\frac{\rrt A_k(T)}{\Gamma(T)}$, a node chosen uniformly at random among all infected nodes with active degree $d \geq 1$ in community $k$ recovers. In particular, a node with active degree $d \geq 1$ in community $k$ is chosen with probability $A_k^d(T) / A_k(T)$. In this case, 
        \begin{equation}\label{eqn:changes-recover}
            A_k^d(T+1) = A_k^d(T) - 1, A_k^{d-1}(T+1)= A_k^{d-1}(T) + 1 \text{ and } X_k(T+1) = X_k(T) - d.
        \end{equation}
        \item with probability $\frac{\irt X_k(T)}{\Gamma(T)}$, choose an active edge uniformly at random among all active edges in community $k$ to attempt infection across an edge. In particular, an edge is chosen with probability $dA_k^d(T) / X_k(T)$ for $d \geq 1$. In this case, 
        \begin{equation}
        \label{eqn:changes-infection}
            A_k^d(T+1) = A_k^d(T) - 1, A_k^{d-1}(T+1) = A_k^{d-1}(T) + 1, \text{ and }X_k(T+1) = X_k(T) - 1.
        \end{equation}
        \begin{itemize}
            \item In this case, we also select the endpoint that the active edge infects and this infection event is successful with certain probability. With probability $\frac{W_{k\ell}S_\ell}{nD_\ell }$, we choose a susceptible vertex in community $\ell$ to infect, and we give it $d' \sim \Pois(D_\ell)$ new active edges. In this case, the infection is successful and in addition to the changes to the system as in \eqref{eqn:changes-infection}, we also have that 
            \begin{equation}\label{eqn:changes-infection-b}
                S_\ell(T+1) = S_\ell(T) - 1, A_{\ell}^{d'}(T+1) = A_{\ell}^{d'}(T) + 1, \text{ and }X_\ell(T+1) = X_\ell(T) + d'.
            \end{equation}
            \item If the infection is not successful, no further action is taken.
        \end{itemize}
    \end{itemize}

Let $\calE^n$ be a realization of this Markov chain,  let $\calF_T$ be the 
natural filtration of this chain, and   
let $\X(T)$ be the row vector $\X(T) = (X_1(T), \dots, X_K(T))$.  Then the expectation of $\X(T+1)$ conditioned on
$\calF_{T}$ is easily calculated, giving
\begin{equation}\label{condEX}
    \E[\X(T+1)\mid \calF_t]
= \X(T)+\frac{\irt+\rrt}{\Gamma(T)}
\Big(\X(T)\tMsir(T)-
\X(T)\Big),
\end{equation}
where $\tMsir(T)$ is the matrix with matrix elements
$$
(\tMsir)_{k\ell}(T)=\frac{\irt}{\irt + \rrt}D_kW_{k\ell}\frac{S_\ell(T)}{n D_\ell}.
$$
Note that zero is an absorbing state for the chain $\calE^n$; if $X(T)$ reaches $0$ (and therefore $A(T)$ also reaches $0$), then the chain dies and remains at $0$ forever. In order to aid in the analysis, we construct another Markov Chain on the same state space, $\Omega$, that restarts whenever the chain hits zero. 
    \begin{definition}[Restarted Discrete Time Chain]
    \label{def:restarted-chain}
        Let $\calE_n$ be a draw from the Markov Chain described above. The restarted chain, $\calE_n^{(1)}$, is a Markov Chain on $\Omega$ that is identical to $\calE_n$ up to the random time $\tau$ when $\X(\tau)=0$.  If $S_{k_0}(\tau)=0$, the chain stops, otherwise 
        restart $\calE_n^{(1)}$ with the current drift probabilities according to another copy of $\calE_n$ drawn independently,
        with $S_k(\tau+1)=S_k(\tau)-\delta_{kk_0}$ and $A_{k}^{d}(\tau+1)=\delta_{kk_0}\delta_{dd_0}$ where  $d_0$ is a fresh sample from $\Pois(D_{k_0})$.  If the restarted chain hits zero again, we again restart it in the same way, until $S_{k_0}=0$, at which point the chain stops.  We denote the active half-edges in the restarted chain by $\X^{(1)}(T)$ and the natural filtration of the restarted chain by $(\calF_T^{(1)})_{T\geq 0}$.
    \end{definition}

Note \eqref{condEX} holds for the restarted chain as well, except
when $\X^{(1)}(T)=0$, in which case
\begin{equation}
\E[X_k^{(1)}(T+1)\mid \calF_T^{(1)}]=\delta_{kk_0}D_{k_0} .   
\end{equation}
%

In a later analysis, we will construct a submartingale from $\X^{(1)}(T)$ and prove a concentration result for that submartingale. Observe that at each discrete time step, the change in $|\X^{(1)}(T)|$ is determined by one Poisson random variable, which has unbounded increments. In order to prove concentration for the submartingale, it will be convenient for us to introduce the following, slightly modified, version of the discrete time chain, which has bounded increments. Here, we truncate the value of the Poisson increments by making the initially drawn degree of a newly infected vertex in community $k$ the minimum of a $\Pois(D_k)$ random variable and some parameter $\zeta_L$, which we will later choose to be conveniently large enough such that the truncated process is equal to the original process with high probability.

\begin{definition}[Restarted, Truncated Discrete Time Chain]
\label{def:trunc-chain}
    Let $\calE_n^{(2)}$ be a Markov Chain on $\Omega$, constructed in the same way as $\calE_n^{(1)}$, but when a susceptible vertex $v$ is infected as described in the beginning of this section, its active degree is given by $d' \sim \Pois(D_{k(v)}) \wedge \zeta_L$ for some parameter $\zeta_L$. Denote the active edges of this process by $\X^{(2)}(T)$ and the natural filtration by $(\calF_T^{(2)})_{T\geq 0}$.
\end{definition}


Notice that \eqref{condEX} holds for $\X^{(2)}(T)$ as well with $\tMsir$ replaced by $\ttMsir$
where
\[
    (\ttMsir)_{k\ell}(T)=\frac{\irt}{\rrt + \irt}\E[\Pois(D_k) \wedge \zeta_L]W_{k\ell}\frac{S_\ell(T)}{n D_\ell}
\]
except for when $\X^{(2)}(T) = 0$, in which case we have
\[
    \E[X_k^{(2)}(T+1) \mid \calF_T^{(2)}] = \delta_{k k_0} \E[\Pois(D_{k_0}) \wedge \zeta_L].
\]
Recall that in the context of Section \ref{sec:branching_process} and Lemma \ref{lem:branching-process-coupling}, we defined
\[
    \Rn = \lambda_{\max}\left(\frac{\irt}{\irt + \rrt} W \diag(\s(0))\right).
\]
Since similar matrices have the same eigenvalues, we can also express $\Rn$ as
$$
\Rn=\lambda_{\max}(C_0)
\quad\text{where}\quad
(C_0)_{k\ell}=\frac{\irt}{\rrt + \irt}D_kW_{k\ell}\frac {s_{\ell}(0)}{D_\ell}.
$$
By Perron-Frobenius, we can find a normalized right eigenvector $\v$ with strictly positive entries such that $C_0\v=\Rn\v$.  Furthermore, by Assumption~\ref{ass:initial-SandI}, we have 
that $\tMsir(0)\v\pto \Rn\v$ as $n\to\infty$.  Thus with probability tending to $1$,
$$
(\tMsir(0)\v)_k\geq \frac {5\Rn+1}6v_k
\quad\text{for all}\quad k\in [K].
$$

\begin{lemma}
\label{lem:pf_of_sub_m}
Assume that $\Rn>1$ and $\tMsir(0)\v\geq \frac {5\Rn+1}6\v$ component wise. Choose
$\zeta_L$ sufficiently large so that
\begin{equation}
\label{asmp:trunc-pois}
    \min_k\frac{\E[\Pois(D_{k}) \wedge \zeta_L]}{D_k}\geq \frac 56+\frac 1{6\Rn},
\end{equation}
let
    $$
   \eps\leq \eps_0=\frac{\min_kS_k(0)}{6S(0)}\left(1-\frac{1}{{\Rn}}\right),
    $$
let $\tau_2$
be the stopping time 
    \[
       \tau_2 \defeq \inf \{ T\geq 0 : S^{(2)}(T) \leq\lfloor(1-\eps)S(0)\rfloor\},
    \]
    and let $H_2(T) = \X^{(2)}(T) \v - {\zeta_0} T$ where
    \[
        {\zeta_0} =\frac 12 \min\left\{\Rn-1,D_{\min}\right\}{v_{\min}},
    \]
    with $D_{\min} = \min_k D_k$ and $v_{\min}=\min_k v_k$.
    Then 
    $H_2(T \wedge \tau_2)$ is a submartingale.
\end{lemma}
\begin{proof}
    Let $\calE_n^{(2)}$ be a realization of the Markov Chain from Definition \ref{def:trunc-chain}, and let 
    \[
        M_2(T) = H_2(T \wedge \tau_2).
    \]
    We will prove that $\E[M_2(T+1) \mid \calF_T^{(2)}] \geq M_2(T)$. Notice that $\{\tau_2 \leq T\}$ is measurable with respect to $\calF_T^{(2)}$. We consider three cases:
    \begin{enumerate}[(i)]
        \item $T \geq\tau_2$ \\
        In this case, $M_2(T+1) = M_2(T)$ and the inequality holds as an equality.
        \item $T < \tau_2$ and $\X^{(2)}(T) = \vec 0$ \\
        In this case, we have that $M_2(T) = H_2(T)$ and $M_2(T+1) = H_2(T+1)$. If $\X^{(2)} = \vec 0$, then we have
        \[
            \E[H_2(T+1) \mid \calF_T^{(2)}] = \E[\X^{(2)}(T+1) \v\mid \calF_T^{(2)}] - (T+1) {\zeta_0}\geq {\zeta_0} T = H_2(T)
        \]
        where we used that $\E[\X^{(2)}(T+1) \v\mid \calF_T^{(2)}]=\E[\Pois(D_{k_0}) \wedge \zeta_L] v_{k_0}\geq \zeta_0$ by our assumption on $\zeta_L$.
        \item $T < \tau_2$ and $\X^{(2)}(T) \neq \vec 0$  \\
        In this case, we also have that $M_2(T) = H_2(T)$ and $M_2(T+1) = H_2(T+1)$. If $\X^{(2)} \neq \vec 0$, then
        \[
            \E[H_2(T+1) \mid \calF_T^{(2)}] = \ttH(T) - {\zeta_0} + \frac{\irt + \rrt}{\Gamma(T)}\left(\ttX(T)\ttMsir(T) - \ttX(T)\right) \v.
        \]
    Let ${\tilde\eps}=1-1/\Rn$. Using the condition $T < \tau_2$, we know that $S^{(2)}(T) \geq S(0)(1-\eps)$  which implies that for all $\ell\in [K]$
        $$S_\ell(0)-S_\ell(T)\leq S(0)-S(T)\leq \eps_0 S(0)
        \leq \frac{\tilde\eps}6 S_\ell(0).
        $$
               As a consequence, $\tMsir(T)\v\geq \left(1-\frac{\tilde\eps}6\right)\tMsir(0)\v\geq  \left(1-\frac{\tilde\eps}6\right) \frac {5\Rn+1}6\v$ and 
        \begin{align*}
        \ttMsir(T)\v&\geq
        \left(\frac 56+\frac 1{6\Rn}\right)\tMsir(T)\v\geq
       \left(\frac 56+\frac 1{6\Rn}\right) \left(1-\frac{\tilde\eps}6\right) \frac {5\Rn+1}6\v\\
        &=\left(1-\frac{\tilde\eps}6\right)^3\Rn\v\geq \left(1-\frac{\tilde\eps}2\right)\Rn\v=\frac{\Rn+1}2 \v
        \end{align*} 
                component wise.  This in turn implies that 
              \begin{align*}
            \E[\ttH(T+1) \mid \ttF_T] &\geq \ttH(T) - {\zeta_0} + \frac{\irt + \rrt}{\Gamma(T)}\frac{\Rn-1}2 \ttX(T)\v \\
            &\geq \ttH(T) - {\zeta_0} + 
             \frac{\Rn-1}2v_{\min} \geq \ttH(T)
        \end{align*}
        where in the last line we used that $\Gamma(T) \leq (\irt + \rrt)\| \ttX(T) \|_1$.
%
    \end{enumerate}
\end{proof}

\begin{remark}
\label{rem:init_conditions_dis_time_chain}
    Note that the proof of the lemma never uses that we started with one initially infected vertex.  In fact, we can start from an arbitrary starting configuration $(S_k(0), A_k^d(0))_{k, d}$ and define the restarted chain as before, always restarting with one infected vertex with label $k_0$, even when the starting configuration at time $T=0 $ is arbitrary.  The previous lemma still holds.  We will use this fact in the proof of the next lemma. 
\end{remark}

\begin{lemma}[Concentration for Sub-Martingale]
\label{lem:sub_m_concentration}
    Let
   \[
        \tau_1 \defeq \inf\{T \geq 0: S^{(1)}(T) \leq \lfloor (1-\eps)S(0) \rfloor\}
   \]
 Under the assumptions of Lemma~\ref{lem:pf_of_sub_m},
    \[
        \P\left(H_1(T \wedge \tau_1) - H_1(0) \leq -\lambda\right) \leq \exp\left(\frac{-\lambda^2}{2T
        (\zeta_L+{\zeta_0})^2}\right) + \left( I(0)+(T\wedge \lceil\eps S(0)\rceil)\right)e^{7\|W\|_\infty -2\zeta_L}
    \]
    where $H_1(T) = \X^{(1)}(T)\v - \zeta_0 T$.
\end{lemma}
\begin{proof}
    Let $\calE_n^{(1)}$ be a draw from the restarted Markov Chain in Definition \ref{def:restarted-chain} and let $\calE_n^{(2)}$ be a draw from the truncated, restarted Markov Chain in Definition \ref{def:trunc-chain}. 
    Observe that
  \begin{equation}  \label{eq:concentration-split}
   \begin{aligned}
          \P\left(H_1(T \wedge \tau_1) - H_1(0) \leq -\lambda\right) 
        &\leq \P\left(\ttH(T \wedge \tau_1) - \ttH(0) \leq -\lambda\right) \\
        &+ \P\left(H_1(T\wedge \tau_1) \neq \ttH(T \wedge \tau_1)\text{ or } H_1(0)\neq \ttH(0)\right).
    \end{aligned}
    \end{equation}
    Since the increments of $\ttH(T)$ are bounded by ${\zeta_0}+\zeta_L$, the first term in \eqref{eq:concentration-split} can be bounded using Azuma-Hoeffding giving
    \[
        \P\left(\ttH(T \wedge \tau_1) - \ttH(0) \leq -\lambda\right) \leq \exp\left(\frac{-\lambda^2}{2T(\zeta_L+{\zeta_0})^2}\right).
    \]
    Now consider the second term in \eqref{eq:concentration-split}. Observe that the only case in which $H_1$ can differ from $\ttH$ is if one of the Poisson degrees are larger than $\zeta_L$ up to time $T\wedge \tau_1$. Thus, letting $Z_i \sim \Pois(\max_k D_k)$ and observing that the number of Poisson random variables drawn in $\calE_n^{(1)}$ up to time $T\wedge \tau_1$ is at most $I(0)+T\wedge (S(\tau)-S(0))\leq I(0)= T\wedge(\lceil\eps S(0)\rceil)=N_T$, we may proceed as in the proof of Lemma \ref{lem:Poisson-bounds} to conclude that
    $$  
        \P\left(H_1(T\wedge \tau_2) \neq \ttH(T \wedge \tau_2)\text{ or }H_1(0)\neq \ttH(0)\right) \leq \P\left(\max_{i = 1, \dots, N_T}
        Z_i > \zeta_L\right) \leq N_Te^{7 \|W\|_\infty-2\zeta_L}
    $$
\end{proof}
We are now ready to prove Theorem~\ref{thm:bounds_on_x}.

\begin{proof}[Proof of Theorem \ref{thm:bounds_on_x}]

  We start by proving the upper bounds on $\|\X(T)\|_1$.  To this end, we observe that for every change in $\|\X(T)\|_1$, one of two events can happen: either a vertex recovers (removing however many active edges are left attached to it), or a half-edge successfully infects another vertex and gives rise to a new Poisson random variable. Since the total number of infection events between $t=0$ and $t=\tau_\eps$
    is equal to $S(\tau_\eps)-S(0)=\lceil \eps S(0)\rceil$ and the number of initially infected vertices is equal to $I(0)$, the random variable 
        $\|\X(T)\|_1$ is stochastically dominated by a sum of $I(0)+\lceil \eps S(0)\rceil$ independent  $\Pois(D)$ random variables,
where $D:=\max_k D_k$.
Since $\eps S(0)\pto\infty$ and
$I(0)/\eps S(0)\pto 0$, Poisson concentration implies that with probability tending to $1$, $\|\X(T)\|_1\leq \zeta_2\eps S(0)$ as long as $\zeta_2>D$.

To prove the lower bounds and the finiteness of $\tau_\eps$, 
we switch to the discrete time chains $\calE_n$ defined in \eqref{eqn:changes-recover}, \eqref{eqn:changes-infection}, and \eqref{eqn:changes-infection-b}.  As we will see, the submartingale concentration bounds for the restarted chain $\calE_n^{(1)}$ in Lemma~\ref{lem:sub_m_concentration}
imply that
\begin{equation}\label{unbdd-lower}
\P\left(\X^{(1)}(T) \v> \frac{\zeta_0T}2\quad\text{if}\quad T\leq \tau_1\right)\to 1
\end{equation}
in case \ref{unbdd_initial}, and
\begin{equation}\label{bdd-lower}
\P\left(\X^{(1)}(T) \v> \frac{\zeta_0T}2\quad\text{if}\quad N_n\leq T\leq \tau_1\right)\to 1
\end{equation}
in case \ref{bdd_initial}.  This shows that with high probability, the restarted chain stays bounded away from zero up to the stopping time $\tau_1$ when  $S^{(1)}(\tau_1)= S(0) -\lceil \eps S(0)\rceil$ in case \ref{unbdd_initial}, and for all times between the time $T_0$ when  $S^{(1)}(T)=S(0) - N_n$ and the time $\tau_1$
when $S^{(1)}(\tau_1)= S(0) -\lceil \eps S(0)\rceil$
in case \ref{bdd_initial} (note that $T\geq T_0$ implies $T\geq N_n$ since each step in the discrete chain decreases $S$ by either $0$ or $1$).
But if the restarted chain does not hit zero before $S(T)\leq S(0) -\lceil \eps S^{(1)}(0)\rceil$, the original chain does not hit $0$ before this time either, showing that up to this point, the two are the same, which  in particular shows that
 $
  \tau_0=\min\{T\in\N:S(T)\leq S(0) -\lceil \eps S(0)\rceil\}<\infty
  $
or equivalently that $\tau_\eps<\infty$ in the continuous time chain. The lower bounds in both cases follow by using the lower bounds in \eqref{unbdd-lower} and \eqref{bdd-lower} at $T=\tau_1$, giving that with probability tending to $1$, 
 $$
\|\X(\tau_\eps)\|_1\geq\frac{ \X(\tau_\eps)\v}{v_{\max}}=\frac{ \X^{(1)}(\tau_1)\v}{v_{\max}}
\geq \frac{ \X(0)\v+\zeta_0\tau_1}{2v_{\max}}
\geq \frac{ \zeta_0}{2v_{\max}}\eps S(0).
$$
This gives the desired the  lower bounds with $\zeta_1=\zeta_0/2v_{\max}$. 

We are left with the proof of the bounds \eqref{unbdd-lower} and \eqref{bdd-lower}.  Let
$$M_1(T)=H_1(T\wedge\tau_1)\quad\text{with}\quad
H_1(T)=\X^{(1)}(T)\v - \zeta_0 T.
$$
Using Lemma \ref{lem:sub_m_concentration} with $\lambda = \frac 1 2 (M_1(0) + \zeta_0T) =\frac{1}{2}(\X^{(1)}(0)\v+\zeta_0 T)$ and an appropriately chosen $\zeta_L$ (to be determined later), we get that
    \begin{equation}
    \label{eq:general_bound}
    \begin{aligned}
             \P\left(M_1(T)-M_1(0)\leq -\lambda\right) 
        &\leq \exp\left(\frac{-(\X^{(1)}(0)\v + \zeta_0T)^2}{8 T(\zeta_L + \zeta_0)^2}\right) + (I(0) + T)e^{7\|W\|_\infty - 2\zeta_L}
        \end{aligned}
    \end{equation}
provided $\zeta_L$ is large enough to guarantee that Assumption \ref{asmp:trunc-pois} from Lemma~\ref{lem:pf_of_sub_m} holds.
 
Starting with case \ref{unbdd_initial}, 
we set $T_0 = \lceil I(0)^{1/3}\rceil$ and  $\zeta_L = T^{1/3} \vee T_0^{1/3}$, observing  that Assumption \ref{asmp:trunc-pois} from Lemma \ref{lem:pf_of_sub_m} holds for all large enough $n$ by monotone convergence and the fact that $\zeta_L\to\infty$. Using the fact that $\sum_{T\geq T_0}f(T)\leq \int_{T_0}^\infty f(T) dT+\sup_{T\geq T_0}f(T)$ whenever $f$ is unimodal, the bound 
(\ref{eq:general_bound}) implies that
    \begin{equation}
    \label{eq:sub_m_sum_upper_tail}
        \sum_{T \geq T_0} 
        \P\left(M_1(T)-M_1(0)\leq -\lambda\right) 
         \leq \int_{T_0}^\infty \left[\exp\left(\frac{-\zeta_0^2 T}{8(\zeta_L + \zeta_0)^2}\right) + (I(0) + T)e^{7\|W\|_\infty - 2\zeta_L}\right] dT +o(1)
    \end{equation}     
    The integral of the first term is of order $O\left(T_0^{2/3}e^{-\Theta(T_0^{1/3})}\right)$ which goes to $0$ with $n$, while the integral of the second term in the bound is of order $O\left(T_0^{2/3}(I(0)+T_0) e^{-T_0^{1/3}}\right)$.  Using the fact that $I(0)\leq T_0^3$, we see that this term goes to $0$ as well. 
    
    Finally, we sum the bound \eqref{eq:general_bound} from $0$ to $T_0-1$     
     using our choice of $\zeta_L$ and $T_0$:
    \begin{align*}
        \sum_{T < T_0}\P(M_1(T) - M_1(0) \leq -\lambda) &\leq \sum_{T < T_0} \left[\exp\left(\frac{-(\X^{(1)}(0)\v )^2}{8T(\zeta_L + \zeta_0)^2}\right) + (T_0^3 + T)e^{7\|W\|_\infty - 2\zeta_L}\right] \\
        &\leq T_0 \exp\left(\frac{-(\X^{(1)}(0)\v )^2}{8T_0(T_0^{1/3} + \zeta_0)^2}\right) + (T_0^{4}+T_0^2)e^{7\|W\|_\infty - 2T_0^{1/3}}.
    \end{align*}
    Here, the second term clearly goes to $0$ with $n$. To see that the first goes to zero, note that
    $\X^{(1)}(0)\v$ is a weighted sum of $I(0)$ independent Poisson random variables with expectation bounded from below by $I(0)D_{\min}v_{\min}$, showing  that $\X(0)\v\geq \Theta(I(0))$ with probability tending to $1$ as $n\to\infty$.
        Together, these bounds imply that
    \begin{align*}
    \P&\left(\X^{(1)}(T) \v> \frac{\X(0)\v+\zeta_0T}2\quad\text{for }\quad 0\leq T\leq \tau_1\right) \\
   &\quad \geq  1-    \P(\text{there exists } T \text{ s.th. }M_1(T) - M_1(0) \leq -\lambda) \xrightarrow{n\to\infty} 1,
    \end{align*}
    establishing the bound \eqref{unbdd-lower}.

To prove the bound \eqref{bdd-lower}, we
again use the bound \eqref{eq:sub_m_sum_upper_tail}, this time with $T_0=N_n$. Indeed, bounding  the integral of the two terms as before, and using that now  $I(0)$ is bounded,  the right hand side is again of order $o(1)$, implying that
\begin{align*}
    \P&\left(\X^{(1)}(T) \v> \frac{\X(0)\v+\zeta_0T}2\quad\text{for }\quad N_n\leq T\leq \tau_1\right) \xrightarrow{n\to\infty} 1
    \end{align*}
as required.   
\end{proof}

The arguments above prove Theorem \ref{thm:bounds_on_x} for $G \sim \psbm(\V, W')$. By conditioning on $G$ being simple, we can immediately extend the results to any $G\sim \sbm(\V, W)$ with $W$ depending on $W'$ as in Lemma \ref{lem:coupling}.

\subsection{Final Size}
\label{sec:one-vertex}



\noindent As we can see from the results obtained in the last two subsections, if the epidemic starts with a constant, non-zero number of initially infected vertices, then a constant fraction of the population ends up being infected with probability converging to the survival probability of the branching process (forest) $\bigcup_{v \in V^I(0)}\calT_{k(v)}^{\bp}$, while for an initial infection of size which goes to infinity but is  $o(n)$, this probability goes to $1$ as long as $\Rn>1$.

In this subsection, we will connect these results to those derived in Section~\ref{sec:uniqueness-implicit-solution} to prove Theorem~\ref{thm:one-vertex-final-size} on the final size of an infection starting with $o(n)$ initially infected vertices. We will need one additional lemma.

\begin{lemma}
\label{lem:sinfty-continuity}
Assume that $W$ is irreducible, and that $s_k(0)>0$ for all $k$. 
 Let $\theta_k$ be the survival probability of the backward process defined in 
\eqref{backward-def} and \eqref{theta-k-def}, let
$\tilde\s(t)$ and $\tilde\x(t)$ be the solution  to 
    the differential equations \eqref{eq:s_dfq_1} and \eqref{eq:x_dfq_1} 
 with initial conditions $\tilde\s(0)$ and $\tilde\x(0)\neq 0$, and let $\tilde\s(\infty)=\lim_{t\to\infty}\tilde \s(t)$.  If $\tilde\x(0)\to 0$ and $\tilde\s(0)\to \s(0)$,
 then $\tilde \s(0) - \tilde\s(\infty)\to \diag (\vec\theta)\s(0)$. 
\end{lemma}

We will prove the lemma in Appendix~\ref{app:branching} by expressing $\s(\infty)$ in terms of the backward process defined in Section~\ref{sec:backward-BP} and then showing that the survival probability of that backwards process converges to that of the backward process defined in \eqref{backward-def} and \eqref{theta-k-def} as $\tilde\x(0)\to 0$ and $\tilde\s(0)\to \s(0)$.




\begin{proof}[Proof of Theorem~\ref{thm:one-vertex-final-size}]
\phantom{a}\newline
We first prove the theorem when $\Rn \leq 1$, in which case 
$\vec\theta=0$, and our goal is to show that
 $\frac 1n\left(S(0)-S(\infty)\right)\to 0$ in probability as $n\to\infty$.
As it turns out, this will be an easy consequence of  Theorem~\ref{thm:lln-final-size} and the following claim, which states that the final size of the epidemic, $|V^R(\infty)|$, increases if the set of initially infected vertices increases.

\begin{claim}
Consider the SIR epidemic on an arbitrary graph $G=(V,E)$, let 
$V^I(0)$ and $V^S(0)$ be the set of initially susceptible and infected vertices, respectively, and let $V^R(\infty)$ be the final set of  recovered vertices.  If we add a set $W\subset V^S(0)$ to the set of infected vertices, then the SIR dynamic with these new initial conditions can be coupled to the original one such that the new final set of recovered vertices,
$\tilde V^R(\infty)$, is a superset of $V^R(\infty)$.
\end{claim}

\begin{proof}
    Consider the following construction of the set $V^R(\infty)$, starting from $V^R(\infty)=V^R(0)$.
    \begin{enumerate}[(i)]
        \item \label{first_step} Add all vertices $v \in V^I(0)$ to $V^R(\infty)$ and color them red initially.
        \item \label{second_step} For every red vertex, draw its recovery time $T_v \sim \Exp(\rrt)$ and for every edge incident to $v$, draw the infection time $Z_{v, i} \sim \Exp(\irt)$ where $i \in \{1, \dots, |N(v)|\}$ and $N(v)$ is the set of neighbors of $v$. Once a vertex has been explored, color it black.
        \item \label{third_step}For all $i \in N(v)$ such that $Z_{v, i} < T_v$, color them red and add $i$ to $V^R(\infty)$. Repeat step \ref{second_step}.
        \item Repeat steps \ref{second_step} and \ref{third_step} until the infection ends (i.e. there are no more new vertices reachable from step \ref{second_step}).
    \end{enumerate}
  While the order in which vertices added to 
  $V^R(\infty)$ is not necessarily the same as the one in which they get added in the course of the continuous time epidemic, it is easy to see (and well known) that the final set   $V^R(\infty)$ is correctly obtained by the above construction.  The construction also makes it obvious that 
$|V^R(\infty)|$ can only become larger if we add
vertices to the initial set of infected vertices.
\end{proof}


    Thus, we may assume that we start with more infected vertices and the final size with these new starting conditions will be an upper bound on the final size with the original starting conditions. This in turn will put us into the setting of Theorem~\ref{thm:lln-final-size}.  Concretely, choose $\hat\eps>0$ arbitrary.
    By Lemma~\ref{lem:sinfty-continuity} and the fact that $\vec\theta=0$, we can choose $\eps>0$ such that whenever
    $\|\tilde{ \x}(0)\|_2\leq\eps$ and
    $\|\tilde \s (0)-\s(0)\|_2\leq\eps$,
    we have that ${\tilde s}(0) -{\tilde s}(\infty)\leq \hat\eps/2$.
Next we 
change the initial conditions, raising the number of infected vertices from $I_k(0)=o(n)$ to $\tilde I_k(0)=\lfloor \frac\eps {2K}n\rfloor$ and decreasing $S_k(0)$
to $\tilde S_k(0)=S_k(0)+I_k(0)-\tilde I_k(0)$.  With probability tending to $1$ as $n\to\infty$,
we then have that $\tilde I_k(0)\geq I_k(0)$, implying that
$$
S_k(0)-S_k(\infty)\leq \tilde S_k(0)-\tilde S_k(\infty).
$$
Furthermore, with the help of Assumption \ref{ass:initial-SandI},
we have that 
 $\frac 1n{\tilde I_k}(0)\pto\tilde x_k(0)=\frac {\eps}{2K}$ and
 $\frac 1n\tilde S_k(0)\pto \tilde s_k(0)=s_k(0)-\frac {\eps}{2K}$, which implies that with probability tending to $1$, the solution of differential equations with starting conditions ${\tilde\s}(0)$ and ${\tilde \x}(0)={\tilde \i}(0)$ obey the bound
 $$\tilde s(0)-\tilde s(\infty)\leq \frac {\hat\eps}2.
 $$
With the help of 
Theorem~\ref{thm:lln-final-size} we conclude that with probability tending to $1$ as $n\to\infty$
$$
\frac 1n(S_k(0)-S_k(\infty))\leq\frac 1n( \tilde S_k(0)-\tilde S_k(\infty))\leq \tilde s(0)-\tilde s(\infty)+\frac{\hat\eps}2\leq \hat\eps.
$$
Since $\hat\eps>0$ was arbitrary, this implies that $
\frac 1n(S_k(0)-S_k(\infty))\pto 0$, as desired.

Next we consider the case when $\Rn > 1$ and $I(0)$ is bounded in probability.
 In this case,  the existence of the limit $\pi$ in Theorem \ref{thm:one-vertex-final-size} implies convergence of $\I(0)$ in probability, with $\pi$ being the survival probability of a branching forest consisting of the disjoint union
   of $I(0)$ independent trees $\calT^{\bp}_{k(v)}$, $v\in V^I(0)$. On the other hand, if $N_n\to\infty$, the survival probability of this branching forest is the limit of it growing to size at least $N_n$, which by  
   Lemma~\ref{lem:exact-coupling-sir-bp}
   and  Remark~\ref{rem:many-initial-vertices-coupling} is asymptotically equal to the probability that there exists some $t<\infty$ such that $S(0)-S(t)\geq N_n$ provided $N_n$ is chosen appropriately.
Consider the stopping time $\tau_\eps$ defined in Theorem~\ref{thm:bounds_on_x}.  By the second statement of the theorem, we then get that for $\eps>0$ small enough,
\begin{equation}
\pi
=\lim_{n\to\infty}\P(\exists t<\infty\text{ s.th. }S(0)-S(t)\geq N_n)=\lim_{n\to\infty}\P(\tau_\eps<\infty).
\end{equation}
Furthermore, conditioned on $\tau_\eps<\infty$,
we have that with high probability
$$
|s(0)-\hat s(\tau_\eps)|\leq 2\eps s(0)
\quad\text{and}\quad
\tilde\zeta_1\eps s(0) \leq \|\hat x(\tau_\eps)\|_1 \leq \tilde\zeta_2\eps  s(0)$$
for some constants $0<\tilde\zeta_1<\tilde\zeta_2<\infty$. Here we moved from the statements in Theorem~\ref{thm:bounds_on_x} for $S(t)$ and $\vec X(t)$ to those for 
$\hat x_k(t) = \frac{X_k(t)}{nD_k}$ and  $\hat \s(t) = \frac{\S(t)}{n}$ 
by absorbing factors of $\frac 1{D_{\max}}$ and
$\frac 1{D_{\min}}$ into $\tilde\zeta_1$
and $\tilde\zeta_2$ and 
using 
Assumption \ref{ass:initial-SandI} to control the difference between $\hat s(0)$ and $s(0)$.  Finally, note that with high probability
$$
\hat\y_{\tau_\eps}\in U_0 \quad\text{and}\quad \max_{u\in V}d_u(\tau_\eps)\leq \log n
$$
since by Lemma~\ref{lem:Poisson-bounds} these statements actually hold for all $t$.  
Defining the compact set
$$
\calC_\eps=\{\tilde \y(0)\in U_0:
|s(0)-\tilde s(0)|\leq 2\eps s(0)
\quad\text{and}\quad
\tilde\zeta_1\eps s(0) \leq \|\tilde\x(0)\|_1 \leq \tilde\zeta_2\eps  s(0)
\}
$$
and the event
$$
\calA_\eps=\{\tau_\eps<\infty, \;
\hat \y_{\tau_\eps}\in \calC_\eps\;\text{and}\; \max_{u\in V}d_u(\tau_\eps)\leq \log n\}
$$
we therefore have that
\begin{equation}
\label{pi<==>tau<infty}
\pi
=\lim_{n\to\infty}\P(\exists t<\infty\text{ s.th. }S(0)-S(t)\geq N_n)=\lim_{n\to\infty}\P(\calA_\eps).
\end{equation}
Note  that $\calC_\eps$ obeys the conditions in  \eqref{thm:LLN-final-size} provided
$\eps$ is small enough to guarantee that $|s(0)-\tilde s(0)|\leq 2\eps s(0)$ implies $\tilde s_k(0)>0$ for all $k$ (if needed, we will decrease $\eps_0$ to make sure this holds).

Let $\tilde \y_t$ denote the solution at time $t$ to the differential equations starting at $\tilde\y_0 = \hat\y_{\tau_\eps}$, and let $\hat\eps>0$ be arbitrary. In view of Lemma~\ref{lem:sinfty-continuity}, we can choose $\eps_0'>0$ such that  $|\tilde \s(0) - \tilde\s(\infty)-\diag (\vec\theta)\s(0)|\leq \eps'$ whenever $\eps\leq\eps_0'$ and $\tilde \y(0)\in\calC_\eps$. For $\eps\leq\min\{\eps_0,\eps_0',\frac 12\hat\eps\}$ we then have that 
\begin{align*}
    \P(\| (\s(0) &- \hat\s(\infty)) - \diag(\vec\theta) \s(0)\|_2 > 4 \hat \eps \mid \calA_\eps) \\
    &\leq \P\left(\| \s(0) - \tilde \s(0)\|_2 > 2\hat\eps \mid  \calA_\eps\right) + \P(\|(\tilde \s(0) - \tilde \s(\infty)) - \diag(\vec\theta) \s(0)\|_2 > \hat\eps \mid  \calA_\eps) \\
    &\qquad+ \P(\|\tilde \s(\infty) - \hat\s(\infty)\|_2 > \hat\eps \mid  \calA_\eps).
\end{align*}
The first term is $0$ since $\tilde \y_0=\y_{\tau_\eps}\in \calC_\eps$ implies
$\|\s(0)-\tilde\s (0)\|_2\leq \|\s(0)-\tilde s(0)\|_1\leq 2\eps s(0)\leq \hat\eps$, the second term is $0$ by Lemma \ref{lem:sinfty-continuity} and the fact that $\eps\leq\eps_0'$, and the last term goes to $0$ by Theorem \ref{thm:LLN-final-size}. Since $\P(\calA_\eps)\to \pi$, this proves  Theorem~\ref{thm:one-vertex-final-size} in the case when $\Rn > 1$ and $I(0)$ is bounded in probability.

When $\Rn>1$ and $I(0)$ is not bounded in probability, we in fact have that $I(0)\pto\infty$ since otherwise the limit defining $\pi$ would not exist.  But then
$$
\pi=1=\lim_{n\to\infty}(\calA_\eps),$$
where the second equality follows by the first statement of Theorem~\ref{thm:bounds_on_x}.  The above argument then again implies the statement of Theorem~\ref{thm:one-vertex-final-size}.
\end{proof}

Finally, we prove Corollary \ref{cor:final-size}.

\begin{proof}
For the proof of the corollary, we will choose a 
coupling of $\calT_{G, v}(t)$ and $\calT_k^{\bp(M)}(t)$
and a
sequence $N_n'$ with $N_n'\to\infty$ and  $N_n'\leq N_n$  such
the statements of Lemma \ref{lem:branching-process-coupling} hold for $N_n'$, i.e., 
    \[
        \P\left(\calT_{G,v}(t) = \calT_k^{{\bp(\Mgen)}}(t) \text{ for all } t \text{ such that }|\calT_k^{\sbm}(t)| \leq N_n'\right) \to 1 \text{ as } n\to\infty.
    \]

We start with the case $\Rn \leq 1$.  Then
$\pi_k = 0$ and the branching process $\calT_k^{\bp(M)}(\cdot)$ dies out in finite time, implying that with probability tending to $1$, $\calT_k^{\bp(M)}(\cdot) \leq N_n'$.  Combined with Lemma \ref{lem:branching-process-coupling}, we conclude that with probability tending to $1$, $\calT_k^{\bp(M)}(t) = \calT_k^{\sbm}(t)$ for all $t$ and 
\[
    \P(S(0) - S(\infty) > N_n)
   \leq  \P(S(0) - S(\infty) > N'_n)\xrightarrow{n\to\infty} \pi_k.
\]
This proves (ii) and (i) in the case where $\Rn \leq 1$, with the exception of the exponentially decaying tails statement for $\Rn<1$. To prove that, we first note that the offspring distribution of a Poisson branching process has exponentially decaying tails.  Since $\calT_k^{\bp(M)}(\infty)$ was constructed as a sub-tree of $\bp_k(M)$, its offspring distribution has exponentially decaying tails as well.  For $\Rn<1$, $\calT_k^{\bp(M)}(\infty)$  is therefore a sub-critical branching process whose offspring distribution has exponentially decaying tails, implying that $|\calT_k^{\bp(M)}(\infty)|$ itself has exponentially decaying tails (see Lemma \ref{lem:subcritical-tails} for a proof of this fact). 
    
Turning to the case $\Rn>1$, we first prove the following claim.
\begin{claim}
\label{claim:final-size-outside-this-set}
    Let $N_n'=o(n)$ be a sequence for which the statements  from Lemma \ref{lem:branching-process-coupling} hold, and assume that
   $\Rn > 1$. If $0<\eps<\vec\theta\cdot\vec s(0)$, then
    \[
        \lim_{n\to\infty} \P(S(0) - S(\infty) \in (N_n', \eps n)) = 0.
    \]
\end{claim}
\begin{proof}
   Recall that $\pi_k$ is the survival probability of the backward branching process. Then for an infection starting from one vertex with label $k$, $\pi = \pi_k$. 
    This implies that
    \[
        \lim_{n\to\infty} \P(S(0) - S(\infty) \leq N_n') = \lim_{n\to\infty} \P(|\calT_k^{\sbm}(\infty)| \leq N_n') = 1-\pi_k,
    \]
    or equivalently that 
    \[
        \lim_{n\to\infty} \P(S(0) - S(\infty) > N_n') = \pi_k.
    \]
    Furthermore, Theorem \ref{thm:one-vertex-final-size} implies that 
    \[
        \lim_{n\to\infty} \P(S(0) - S(\infty) \geq \eps n) = \pi_k.
    \]
    Since $N_n' = o(n)$, $\{S(0) - S(\infty) \geq \eps n\} \subseteq \{S(0) - S(\infty) > N_n'\}$ and we then see that 
    \[
        \lim_{n\to\infty} \P \left(S(0) - S(\infty) \in (N_n', \eps n)\right) = 0.
    \]
\end{proof}

Using the just proven claim, we see that for $\eps>0$ small enough,
the probability that $S(0) - S(\infty) > N_n$ is asymptotically equal to the probability that $S(0) - S(\infty) > \eps n$,
which converges to $\pi_k$ by Theorem~\ref{thm:one-vertex-final-size}.
This proves (i) in the case where $\Rn > 1$. 

To prove (iii), we
need to show that for all fixed  $N<\infty$,
\[
\P\Big(S(0) - S(\infty)=N\,\Big|\, S(0) - S(\infty) \leq  N_n\Big)\to
\P\Big(|\calT_k^{{\bp(\Mgen)}}(\infty)| =N \,\Big|\,|\calT_k^{{\bp(\Mgen)}}(\infty)|<\infty\Big).
\]
But this follows immediately from (i) and Lemma~\ref{lem:branching-process-coupling}.  Indeed,  (i) implies that
$\P(S(0) - S(\infty)\leq N_n)\to 1-\pi_k=
\P\Big(|\calT_k^{{\bp(\Mgen)}}(\infty)|<\infty)$
while Lemma~\ref{lem:branching-process-coupling} together with the fact that $N_n\to\infty$ 
implies that 
\begin{align*}
   \lim_{n\to\infty} &\P\Big(\{S(0) - S(\infty)=N\}\cap\{S(0) - S(\infty)\leq  N_n\}\Big) \\
   &=  
    \lim_{n\to\infty} \P\Big(S(0) - S(\infty)=N\Big)=\P\Big(|\calT_k^{\bp(M)}(\infty)| =N\Big) .
\end{align*}
Finally, we prove (iv). It suffices to show that
\[
    \lim_{n\to\infty} \P\left(\left|\frac{\S(0) - \S(\infty)}{n} -\diag(\vec \theta)\s(0) \right| < \eps\right) = \lim_{n\to\infty} \P\left(S(0) - S(\infty) > N_n\right).
\] 
Let $\eps \in (0, \vec \theta \cdot \s(0))$. Observe that for $N_n = o(n)$ and $n$ large enough,
the event 
\[
    \left\{ \left|\frac{\S(0) - \S(\infty)}{n} -\diag(\vec \theta)\s(0) \right| < \eps\right\}
\]
implies $\{S(0) - S(\infty) > N_n\}$.
Thus 
\begin{align*}
    \P\left(\left\{ \left|\frac{\S(0) - \S(\infty)}{n} -\diag(\vec \theta)\s(0) \right| < \eps\right\}\cap \{S(0) - S(\infty) > N_n\}\right) &= \P\left(\left|\frac{\S(0) - \S(\infty)}{n} -\diag(\vec \theta)\s(0) \right| < \eps\right) \\
    &\xrightarrow{n\to\infty} \pi_k
\end{align*}
where the convergence is true by Theorem \ref{thm:one-vertex-final-size}. Furthermore
\[
    \lim_{n\to\infty} \P(S(0) - S(\infty)> N_n) = \pi_k
\]
by (i) and this gives the desired result.
\end{proof}



\section{Discussion}
\label{sec:discussion}
In this section, we discuss our results and relate them to notions like herd immunity, the force of an infection, and the so-called pair-approximation, concluding with the discussion of possible extensions.  Throughout, we consider the setting where $s_k(0)>0$ for all $k$ and $W$ is irreducible.

\subsection{Herd Immunity}
The herd immunity threshold for a disease is the threshold at which the infection starts slowing down due to the fact that a significant proportion of the (finite) population have been infected, and new infection attempts are unlikely to lead to a successful infection. Mathematically, this is the point at which the effective reproductive number
falls below one,
and new infections die out exponentially fast.

In the deterministic, homogeneous mixing setting for the SIR epidemic, described by the differential equations 
\begin{equation}
    \frac{ds}{dt} = -\irt s i
\end{equation}
\begin{equation}
    \frac{di}{dt} = \rrt\left( \frac{\irt s}\gamma-1\right)i
\end{equation}
\begin{equation}
    \frac{dr}{dt} = \rrt i
\end{equation}
the effective reproductive number is given by $\Reff(t) = \irt s(t) / \rrt$, and the herd immunity threshold is at $\Reff(t) = 1$, below which the number of infected vertices is exponentially decreasing. 

To discuss the notions of  herd immunity and effective reproductive number for SIR on the stochastic block model, we
recall the vector form of the differential equations for the number of susceptible vertices and the number of active half-edges for SIR on the  SBM:
\begin{equation}\label{dsdt-matrix}
    \frac{d\s(t)}{dt} = -(\irt + \rrt)\x(t) C(t)
\end{equation}
\begin{equation}\label{dxdt-matrix}
    \frac{d\x(t)}{dt} = (\irt + \rrt)\x(t)(C(t)-1)
\end{equation}
where $C(t) = \frac{\irt}{\irt + \rrt} W \diag(\s(t))$. As  pointed out in Remark~\ref{rem:herd-immunity} in Section~\ref{sec:herd-immunity}, which we rephrase here to keep our discussion self-contained, the largest eigenvalue of this matrix,  $\Reff(t)$, is the appropriate generalization of the effective reproductive number to our setting, with $\Reff(t)=1$ determining the herd-immunity threshold. 

Indeed, if $\Reff(t_0)<1$ and $\v_0$ is the corresponding right eigenvector with all positive entries, then
\begin{align*}
\frac{d\x(t)}{dt}\cdot\v_0&=
 (\irt + \rrt)\x(t)(C(t)-1)\v_0\\
 &\leq  (\irt + \rrt)\x(t)(C(t_0)-1)\v_0\\
 &= (\irt+\rrt)(\Reff(t_0)-1) \x(t)\cdot\v_0
\quad\text{for all}\quad t\geq t_0
\end{align*}
showing that $\x(t)\cdot \v$ and hence $\|\x(t)\|_1$ is decaying exponentially in $t-t_0$ for any initial condition $\x(t_0)$. As a consequence, a small infection at time $t_0$ will only have small overall impact on the final size of the infection.  By contrast, if $\Reff(t_0)>1$, then $\x(t)\cdot \v$ will grow exponentially in $t-t_0$ until  $\Reff(t)=1$.   Thus no matter how small $\x(t_0)$ is, the infection will always have a sizable effect on the final size, since the infection won't start to recede  before it has infected enough susceptible vertices to drive $\Reff(t)$ below one.

Notice that here, the parameters that drive the infection are not the number of infected vertices as in the case of the standard SIR, but rather the numbers of active half-edges in each community weighted by the connectivity between communities. This quantity differs by community based on the connectivity in between communities and gives rise to $K$ parameters that govern the trajectory of the epidemic, one in each community. This implies that the trajectories of the epidemic curves in different communities may peak at different times, and furthermore that it is possible for the epidemic to be spreading in a community despite the global herd immunity threshold being reached in the population. 

We make two observations about $\Reff(t)$ and its relation to the number of active half-edges at time $t$. First, notice that if $\Reff(t) > 1$ then there exists at least one community $k$ such that $dx_k(t)/dt > 0$. In other words, before the herd immunity threshold, at least one community experiences exponential growth in the number of half-edges. Second, if $\Reff(t) < 1$, then in at least one community $k$, $dx_k(t)/dt < 0$ and there is exponential decay of half-edges in that community. We can see this as follows: let $\v_0$ be the right Perron-Frobenius eigenvector of $C(t)$. If $\Reff(t) > 1$ then
\begin{align*}
    \frac{d\x(t)}{dt}\cdot \v_0 &= (\irt + \rrt)\x(t) (C(t) - 1)\v_0 \\
    &= (\irt + \rrt) (\Reff(t) - 1)\x(t)\cdot \v_0 > 0.
\end{align*}
Since $\v_0$ has all strictly positive entries, this implies that at least one component of $d\x(t)/dt$ must be larger than $0$. Similarly, in the case where $\Reff(t) < 1$, there must exist some $k$ such that $dx_k(t)/dt$ is less than $0$.



Thus, after the herd immunity threshold the number of half-edges in at least one community is decreasing.
Note that eventually all components decay exponentially (we prove this in Lemma~\ref{lem:x-exp-decay} in Appendix~\ref{app:diff-equ}) furthermore, though a small infection at $t_0$ will never grow large this does not preclude some components of $\x(t)$ from growing even after the herd immunity threshold.  All we know is that a small infection can only grow by at most a constant factor before dying out.

\subsection{Numerical Simulations for $K=2$}
In this section, we demonstrate the relationship between the herd immunity threshold and the various maxima of the components of $\x$ and $\i$. In Figure \ref{fig:ode_sims} we illustrate the relationship between $\Reff(t)$ and $\x(t)$ in the differential equations for a two-community stochastic block model. We start with a certain proportion, $\eps$, of infected individuals in the first community and plot the curves for $x_1(t), x_2(t), i_1(t),$ and $i_2(t)$. The peaks of each curve are also plotted in relation to the first time $t$ at which $\Reff(t) < 1$. Observe that the threshold $\Reff(t) < 1$ always occurs between the first and last peak in the $\x(t)$. The first row simulates the ODE with the two communities being the same size, in which case as $\eps$ gets smaller, the trajectory of the curves for the two communities converge. The last row simulates the ODE for two communities of different sizes, with the contact matrix $W$ adjusted so that the expected number of edges in each community is the same. For small $\eps$, the peaks between the two communities are visibly separated, but as $\eps$ gets smaller, the peaks get closer.

\begin{figure}[H]
    \begin{subfigure}[h]{\textwidth}
        \centering
        \begin{subfigure}[h]{0.3\textwidth} 
            \centering
            \includegraphics[width=\textwidth]{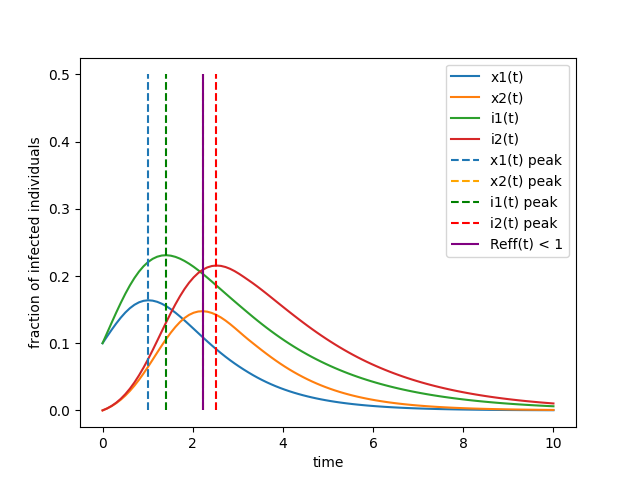} 
        \end{subfigure}
        \hfill
        \begin{subfigure}[h]{0.3\textwidth}
            \centering
            \includegraphics[width=\textwidth]{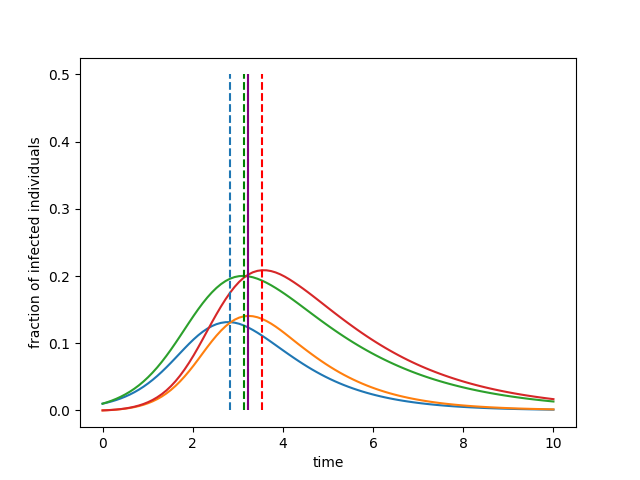} 
        \end{subfigure}
        \hfill
        \begin{subfigure}[h]{0.3\textwidth}
            \centering
            \includegraphics[width=\textwidth]{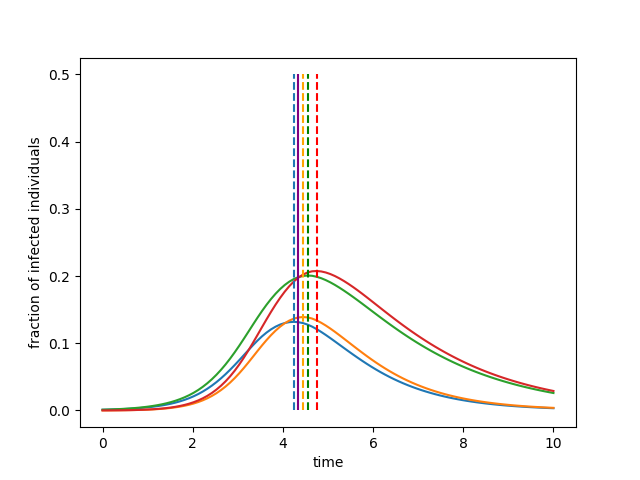} 
        \end{subfigure}
        \caption{Simulations for two communities of equal size ($s_1(0) = s_2(0) = 1/2$) and $W = \begin{pmatrix}
            10 & 1 \\ 1 & 10
        \end{pmatrix}$.} 
    \end{subfigure}

    \begin{subfigure}[h]{\textwidth}
        \centering
        \begin{subfigure}[h]{0.3\textwidth}
            \centering
            \includegraphics[width=\textwidth]{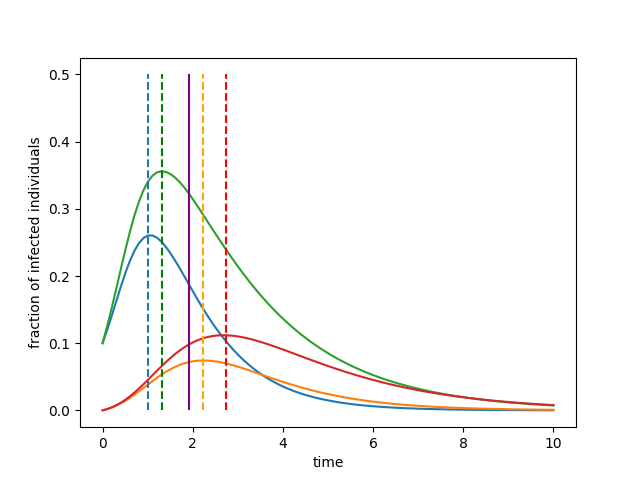} 
        \end{subfigure}
        \hfill
        \begin{subfigure}[h]{0.3\textwidth}
            \centering
            \includegraphics[width=\textwidth]{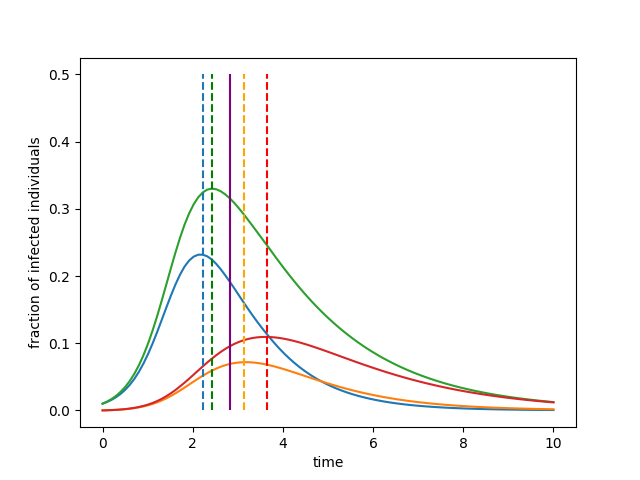} 
        \end{subfigure}
        \hfill
        \begin{subfigure}[h]{0.3\textwidth}
            \centering
            \includegraphics[width=\textwidth]{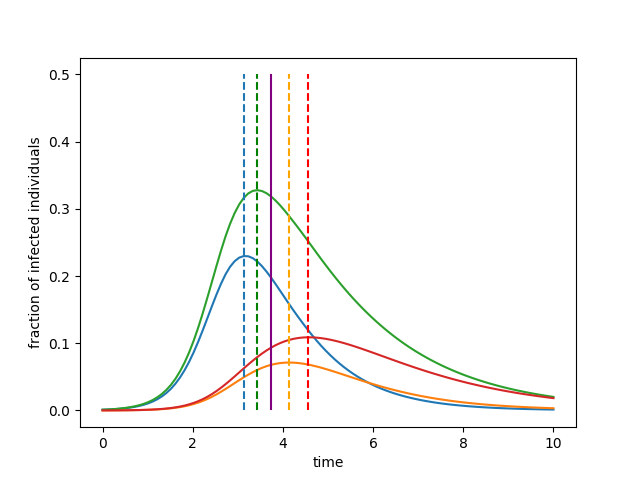} 
        \end{subfigure}
        \caption{Simulations for two communities of different sizes ($s_1(0) = 2/3, s_2(0) = 1/3$) and $W = \begin{pmatrix}
            10 & 1 \\ 1 & 10
        \end{pmatrix}$.}
    \end{subfigure}

    \begin{subfigure}[h]{\textwidth}
        \centering
        \begin{subfigure}[h]{0.3\textwidth}
            \centering
            \includegraphics[width=\textwidth]{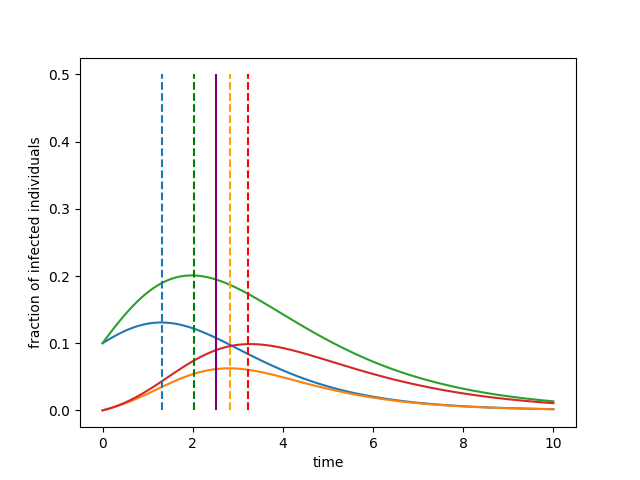} 
        \end{subfigure}
        \hfill
        \begin{subfigure}[h]{0.3\textwidth}
            \centering
            \includegraphics[width=\textwidth]{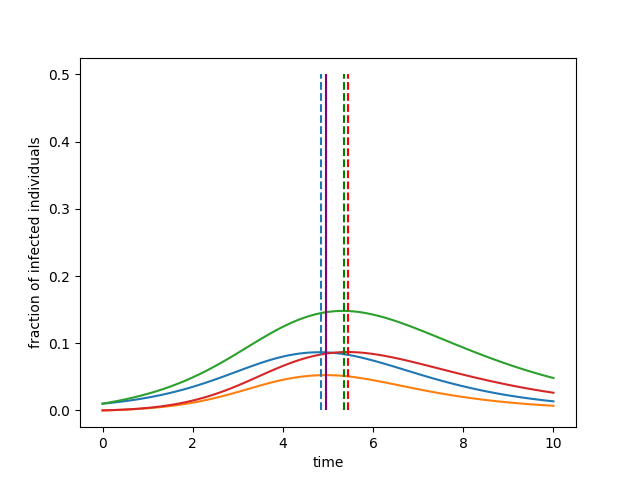} 
        \end{subfigure}
        \hfill
        \begin{subfigure}[h]{0.3\textwidth}
            \centering
            \includegraphics[width=\textwidth]{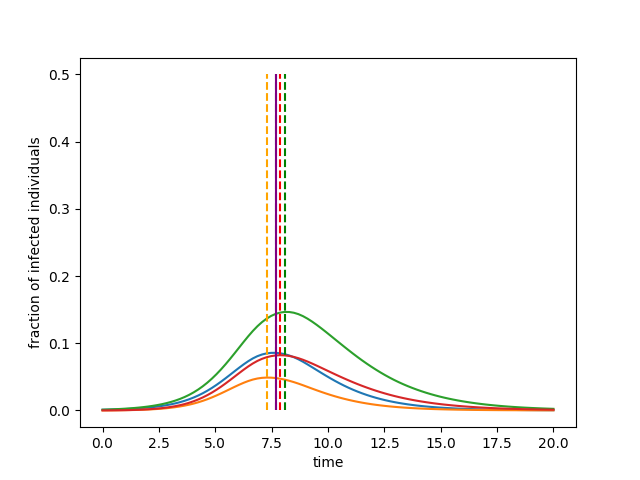} 
        \end{subfigure}
        \caption{Simulations for two communities of different sizes ($s_1(0) = 2/3, s_2(0) = 1/3$) and $W = \begin{pmatrix}
            5 & 1 \\ 1 & 10
        \end{pmatrix}$ so the expected number of edges within each community is the same.}
    \end{subfigure}
    \caption{A simulation of our system of ordinary differential equations on a two-community stochastic block model. We vary parameter values of $\eps$, the initial fraction of infected individuals going from $0.1, 0.01, 0.001$ from the left column to the right column; $W$, the contact matrix; and $s_1(0)$ and $s_2(0)$, the initial proportion of susceptible individuals in each community. For all simulations, $\irt = \rrt = 1/2$.}
    \label{fig:ode_sims}
\end{figure}

    Notice that in all of the simulations, in each community the  peak of the fraction of infected individuals occurs after the  peak in the number of active half-edges for that community, and furthermore $x_k$ stays below $i_k$ for all times $t>0$.  As the following lemma shows, this is not a numerical coincidence, but a direct consequence of the differential equations for $\s$, $\x$ and $\i$.

    \begin{lemma}\label{lem:I-after-X} Assume that $W$ is irreducible, that 
$s_k(0)>0$ for all $k\in [K]$, that $x(0)>0$, and that $x_k(0)=i_k(0)$. Then $x_k(t)<i_k(t)$ for all $t>0$.  If we assume in addition that
$\left.\frac {dx_k}{dt}\right|_{t=0}>0$, then the first local maximum of $x_k$, occurs after the first local maximum of $i_k$.  
\end{lemma}
    
While the details of the proof are a little tedious, the basic idea is simple and based on the observation that
    $$
    \frac {di_k}{dt}
-\frac {dx_k}{dt}=(\irt+\rrt)x_k- \rrt i_k, $$
showing that at least near $t=0$, where $\frac{x_k}{i_k}$ starts at $1$, $x_k$ grows slower than $i_k$.  To prove our statement on the local maxima, all one needs to show is that the first point where $\frac{x_k}{i_k}$ falls below $\frac{\rrt}{\irt+\rrt}$ happens after the first local maximum of $x_k$.  We will prove this fact in Appendix~\ref{app:diff-equ}, where we give the complete proof of Lemma~\ref{lem:I-after-X}.

In Figure \ref{fig:stochastic_sims} we visualize the LLN behavior for the SIR epidemic on a stochastic block model with two communities. The figures contain 100 trajectories of the fraction of infected vertices in simulated epidemics with 
\[
    W = \begin{pmatrix} 5 & 1 \\ 1 & 10 \end{pmatrix}, \quad \eps = 0.01, \quad \irt = \rrt = 1/2.
\]
We sampled networks with size $n = 200, 2000$ and $10,000$ nodes split equally amongst the two communities with $n_1 = n_2 = n/2$, where $n_1$ and $n_2$ are the number of nodes in community $1$ and $2$, respectively. A realization $G$ of the stochastic block model with parameters $W, (n_1, n_2)$ was drawn and then an epidemic, in which an $\eps$ fraction of the population in community $1$ was initially infected and the remaining nodes were all susceptible, was run on $G$. The solid lines in the figure depict the corresponding ODE solutions for $i_1(t)$ and $i_2(t)$. The stochastic epidemic was simulated using algorithms from \cite{Miller2019}.

\begin{figure}[H]
    \begin{subfigure}[b]{\textwidth}
        \centering
        \begin{subfigure}[b]{0.3\textwidth}
            \centering
            \includegraphics[width=\textwidth]{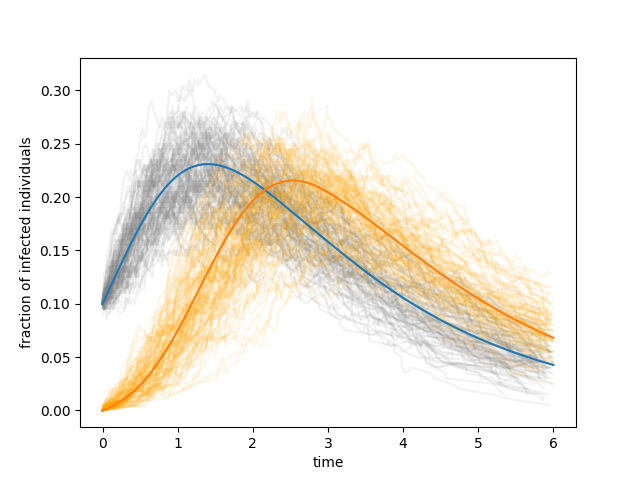}
            \caption{$n_1 = 100, n_2 = 100$}
        \end{subfigure}
        \hfill
        \begin{subfigure}[b]{0.3\textwidth}
            \centering
            \includegraphics[width=\textwidth]{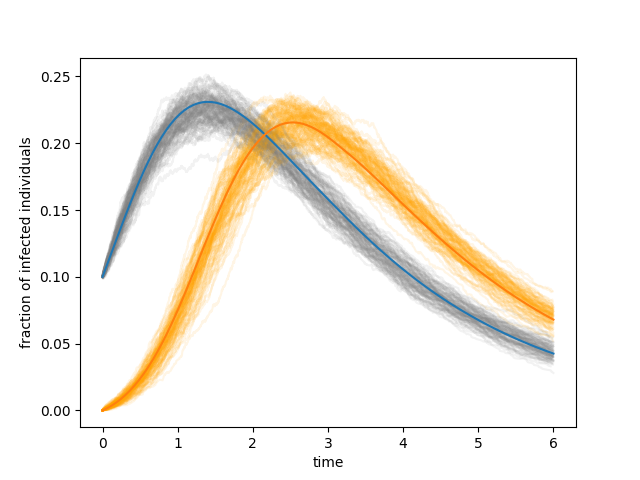} 
            \caption{$n_1 = 1000, n_2 = 1000$}
        \end{subfigure}
        \hfill
        \begin{subfigure}[b]{0.3\textwidth}
            \centering
            \includegraphics[width=\textwidth]{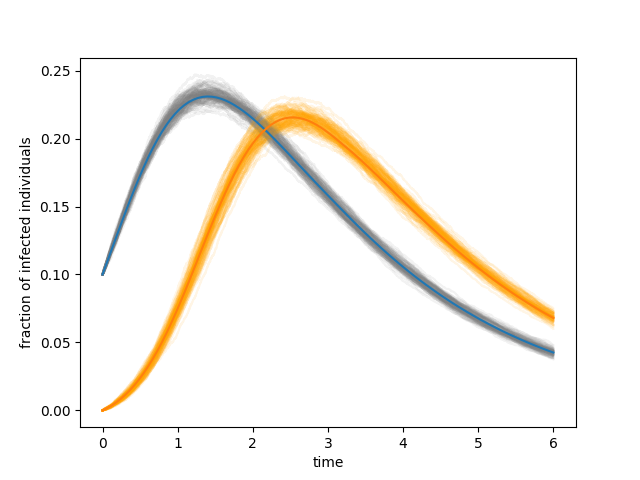}
            \caption{$n_1 = 5000, n_2 = 5000$}
        \end{subfigure}
    \end{subfigure}
    \caption{Simulations for one hundred stochastic SIR trajectories of the fraction of infected individuals on a stochastic block model with two communities of the same size. From left to right, we increase the number of nodes in the network keeping the ratio $\frac{n_1}{n_2}$ the same. The solid blue and orange curves are the solution to the differential equations with the same parameters.}
    \label{fig:stochastic_sims}
\end{figure}

\subsection{Force of the Infection}
From a practical perspective, it is hard to estimate $\x(t)$, since this quantity is generated from a coupling of the SIR epidemic on the SBM and is difficult to interpret in the context of the epidemic. However, what can be observed is the force of infection. In the epidemiology literature, force of infection (also called hazard rate or incidence rate) refers to the rate at which susceptible individuals in a population are infected \cite{vynnycky2010}. In the standard SIR setting, the force of infection on the susceptible population is given by $\irt \frac{I}n$, while in our setting, the force of infection on susceptible vertices in community $k$ is given by $\irt \sum_\ell\frac{X_\ell}{nD_k} W_{\ell k}$. Again, observe that the force of infection is determined by the active half-edges and \textit{not} by the number of infected individuals in each community. Denoting the limit of the force of the infection on individuals of with label $k$ by $F_k$, we see that in terms of the vector $\boldsymbol F=(F_k)_{k\in [K]}$ the limiting dynamics of the epidemic is described by
\begin{equation}\frac {d\s(t)} {dt}=-\diag(\boldsymbol F(t))\s(t)
\end{equation}
\begin{equation}\label{dFdt-matrix}
    \frac{d\boldsymbol F(t)}{dt} = (\irt + \rrt)\boldsymbol F(t)(C^T(t)-1)
\end{equation}
with initial conditions for $\boldsymbol F$ given by $F_k(0)=\irt \sum_\ell i_\ell(t) W_{\ell k}$, and $C^T(t)$ denoting the transpose of $C(t)$. Since $C^T(t)$ has the same eigenvalues as $C(t)$, this shows that our discussion of the herd immunity threshold from the last subsection can equivalently be formulated in terms of the observable force of the infection, rather than the unobservable number of active half-edges.

\subsection{Pair-Approximation}
In this paper, we couple the stochastic epidemic on the stochastic block model to a dynamic process where the graph and the infection are revealed at the same time, allowing us to derive differential equations for the stochastic epidemic and to understand the parameters that drive the infection. An alternative, {\it a priori} not mathematically rigorous, approach to deriving a system of limiting ordinary differential equations for the stochastic epidemic is to draw a realization of the graph first and then use a mean-field approximation, relying on the number of edges connecting individuals in different compartments. Below, we derive a system of ODEs using the heuristics of one such mean-field approach: the pair-approximation (cf. Ch.7 of \cite{Andersson2000} or, Ch. 4 of \cite{kiss2017mathematics}). Note that this approach is conjectured to be exact for locally tree like graphs, but to our knowledge, this has not been established rigorously \citep*{Sharkey2015exact, Kiss2015networkloops, kiss2017mathematics, kiss2022necessarysufficientconditionsexact}.


Our quantities of interest remain the expected number of susceptible, infected, and recovered vertices in community $k$ at time $t$.  Let $G \sim \sbm(\V, W)$ be a realization of the stochastic block model on $n$ vertices, let $u\in [n]$ be the labels of the nodes of $G$, and let $g$ be the adjacency matrix of $G$. For $A \in \{S, I, R\}$,
let $I_u^A$ be the indicator function that $u$ is susceptible, infected, or recovered, respectively, and let
\begin{equation}\label{eq:pairwise_def}
    [A_k]_n = \sum_{\substack{u \in [n] \\ k(u) = k}} I_u^A, \quad [A_kB_\ell]_n = \sum_{\substack{u, v \in [n] \\ k(u) = k \\ k(v) = \ell}}I_u^A g_{uv}I_v^B, \quad [A_kB_\ell C_m]_n \sum_{\substack{u, v, w\in [n], u \neq w \\ k(u) = k \\ k(v) = \ell \\ k(w) = m}} I_u^A g_{uv} I_v^B g_{vw} I_w^C.
\end{equation}
We say an edge $uv$ is of type $A_kB_\ell$ if $u$ has the community label $k$ and is in state $A$ and $v$ has label $\ell$ and is in state $B$, i.e., if $uv$ contributes to $[A_kB_\ell]$, and similarly for a type $A_mB_kC_\ell$ path $uvw$.

Assume that all quantities in \eqref{eq:pairwise_def} when normalized by $n$ converge to deterministic limits as $n \to \infty$. Observing that the rate at which a given susceptible vertex $v$ becomes infected is equal to the number of edges joining $v$ to an infected vertex, and that the rate at which infected vertices recover, is equal to  $\gamma$, we heuristically obtain 
the following system of differential equations:
\begin{equation}\label{eq:pair-S}
    \frac{d}{dt}[S_k] = - \irt \sum_{\ell \in [K]} [I_\ell S_k]
\end{equation}
\begin{equation}\label{eq:pair-I}
    \frac{d}{dt}[I_k] =\irt \sum_{\ell \in [K]} [I_\ell S_k] -\rrt [I_k]  =- \frac{d}{dt}[S_k] -\rrt [I_k]
\end{equation}
with errors which should be $o(n)$ as $n\to\infty$.  With a little bookkeeping, this heuristic also gives
\begin{equation}\label{eq:pair_approx}
    \frac{d}{dt}[I_\ell S_k] = -\rrt [I_\ell S_k] + \irt \sum_{m\in [K]} [I_m S_\ell S_k] - \irt\left( [I_\ell S_k]+\sum_{m \in [K]} [I_m S_k I_\ell]  \right)
\end{equation}
\begin{equation}\label{eq:pair_approx'}
    \frac{d}{dt}[S_\ell S_k] = -2 \irt \sum_{m \in [K]}[I_m S_\ell S_k].
\end{equation}
We give a heuristic explanation for \eqref{eq:pair_approx}, leaving the heurisic derivation of \eqref{eq:pair_approx'} to the reader.
At rate $\rrt [I_\ell S_k]$, the infected endpoint in an $I_\ell S_k$ edge recovers and the number of $I_\ell S_k$ edges decreases by one. At rate $\irt[I_m S_\ell S_k]$ for any $m \in [K]$, the infected vertex in an $I_m S_\ell S_k$ infects its $S_\ell$ neighbor and adds an $I_\ell S_k$ edge. At rate $\irt[I_m S_k I_\ell]$ for any $m \in [K]$, an $I_m S_k I_\ell$ set of edges becomes an $I_m I_k I_\ell$ set of edges and removes an $S_k I_\ell$ edge. Finally, at rate $\irt[I_\ell S_k]$ the $I_\ell$ node infects the $S_k$ node and removes an $I_\ell S_k$ edge.

In the so-called pair-approximation, it is then assumed that we can make the approximation
\[
    [A_mB_kC_\ell] = \frac{[A_mB_k][B_kC_\ell]}{[B_k]},
\]
giving us the following equations

\begin{equation}\label{eq:pair_approx-IS}
    \frac{d}{dt}[{I_\ell S_k}] = -(\irt + \rrt) [I_\ell S_k] + \irt \sum_{m \in [K]} \left(\frac{[I_m S_\ell][S_\ell S_k]}{[S_\ell]} - \frac{[I_m S_k][S_k I_\ell]}{[S_k]}\right)
\end{equation}
\begin{equation}\label{eq:pair_approx-SS}
    \frac{d}{dt}[S_\ell S_k] = -2 \irt \sum_{m \in [K]}\frac{[I_m S_\ell][S_\ell S_k]}{[S_\ell]}.
\end{equation}
Finally, under the assumption that the initially susceptible vertices of type $k$ are a random set of size $[S_k]$, we can assume that at time $0$,
$$
[S_kS_\ell]_{t=0}=\frac {W_{k\ell}}n [S_k]_{t=0}[S_\ell]_{t=0}.
$$
On the other hand, combining \eqref{eq:pair-S} and \eqref{eq:pair_approx-SS} and using the symmetry $k\leftrightarrow \ell$ of $[S_kS_\ell]$,  we get
\begin{equation}
\frac d{dt}\left([S_kS_\ell]-\frac {W_{k\ell}}n [S_k][S_\ell]\right)=
     - \irt \sum_{m \in [K]}\left(\frac{[I_m S_\ell]}{[S_k]}+\frac{[I_m S_k]}{[S_\ell]}\right)\left(
   [S_\ell S_k]-\frac {W_{k\ell}}n[S_\ell][S_k]\right)
\end{equation}
which we can integrate to get that
\begin{equation}\label{eq:ss_approx}
    [S_\ell S_k] =[ S_\ell][ S_k] \frac{W_{k\ell}}{n}
\end{equation}
for all $t$.

Let $[a_k] = n_{-1}[A_k]$ for $A \in {S, I, R}$ and let $[y_k] = n^{-1}\sum_{\ell \in [K]} [I_\ell S_k]$. The pair approximation and \eqref{eq:ss_approx} imply the following system of differential equations:
\begin{equation}\label{eq:s_pair_approx}
    \frac{d}{dt}[s_k] = -\irt [y_k]
\end{equation}
\begin{equation}\label{eq:i_pair_approx}
    \frac{d}{dt}[i_k] =- \frac{d}{dt}[s_k] -\rrt [i_k]
\end{equation}
\begin{equation}\label{eq:x_pair_approx}
    \frac{d}{dt}[y_k] = -(\irt + \rrt) [y_k] + \irt  \sum_{\ell}[y_\ell]
        W_{\ell k}[s_k] - \frac{\irt}{[s_k]}[y_k]^2.
\end{equation}
As we will see, the time evolution for $s_k$ and $i_k$ following from these equations is identical to the time evolution derived rigorously in this paper.  To show this, we consider the force of the infection, $$F_k=-\frac 1{[s_k]}  \frac{d}{dt}[s_k]
=\irt\frac{[y_k]}{[s_k]}$$ and use \eqref{eq:ss_approx} and \eqref{eq:x_pair_approx} to derive the differential equation
\begin{align*}
\frac {dF_k}{dt}
&=-\irt\frac  {[y_k]}{[s_k]^2}\frac {d[s_k]}{dt}+\irt\frac 1{[s_k]}\frac {d[y_k]}{dt}
=\left(\irt\frac{[y_k]}{[s_k]}\right)^2
 -(\irt + \rrt)\left(\irt \frac{[y_k]}{[s_k]}\right)
 + \irt^2  \sum_{\ell}[y_\ell]
 W_{\ell k}
 - \left(\irt\frac{[y_k]}{[s_k]}\right)^2\\
 &=
 -(\irt + \rrt)F_k
 + \irt  \sum_{\ell}F_\ell[s_\ell]
 W_{\ell k},
    \end{align*}
which is the time evolution for the force of infection derived rigorously  in the previous section.

\subsection{Extensions}
In this work, we have given a rigorous derivation of the limiting system of ordinary differential equations for the stochastic SIR epidemic on the stochastic block model. While there are some general alternative approaches for approximating the epidemic trajectories using e.g. the forward and backward branching processes of the model (see \cite{barbour2013approximating, bhamidi2014front}) or local graph limits (see \cite{alimohammadi2023epidemic, cocomello2023exact}), to our knowledge this paper is first to rigorously establish a system of differential equations for the limiting behavior allowing for general (including sublinear) starting conditions. Our law of large numbers results give convergence of the stochastic epidemic to the limiting ODE for all $t$ finite and infinite, and demonstrate that the infection curves in different communities are driven by different parameters. In particular, we prove a law of large numbers for the $t$ to infinity limit (final size) of the epidemic. To prove the final size result, we introduce a novel method using the idea of herd immunity and the $\Reff(t)$ threshold. Our system of ODEs implies the system of ODEs derived using the pair-approximation, justifying the use of mean-field approximations in this setting. 

In our model, we consider a constant rate of infection $\irt$ across all communities, however a natural extension would be to vary the rate of infection by community and consider a matrix $\bm{\irt}=(\irt_{k\ell})_{k,\ell\in [K]}$ and vector $\bm \rrt=(\rrt_k)_{k\in [K]}$ instead, where $ \irt_{k\ell}$ is the rate at which an infected individual in community $k$ infects its neighbors in community $\ell$, and $ \rrt_k$ is the rate at which vertices in community $k$ recover. In such a model where the transmission rates vary by community, we can also derive a set of differential equations based on the notion of active half-edges. Let $X_{k \ell}$ be the number of active half-edges going out of community $k$ with an endpoint (not yet realized) in community $\ell$. 
Then at any given time
\begin{itemize}
    \item at rate $\rrt_k X_{k\ell}$, we lose one active half-edge contributing to $X_{k\ell}$ 
    \item at rate $\irt_{k\ell} X_{k\ell}$, we explore one of the active half-edges contributing to $X_{k\ell}$.
\end{itemize}
When we explore an active half-edge contributing to $X_{k\ell}$, $X_{k\ell}$ decreases by $1$. Now, let us consider how $X_{k\ell}$ increases. Note that at rate $\irt_{j k} X_{j k}$ a half-edge in community $j$ with (future) endpoint in community $k$ is explored. With probability $S_k/n_k$ the exploration leads to an infection, and in this case $X_{k\ell}$ increases by $\Pois(D_{k\ell})$, where we recall that $D_{k\ell} = W_{k\ell}n_\ell/n$ is the expected number of half-edges from community $k$ into community $\ell$. Therefore, 
\begin{itemize}
    \item at an overall rate $\frac{S_k}{n_k}\sum_{\ell}  \irt_{\ell k} X_{\ell k}$, $S_k$ goes down by $1$ and $I_k$ goes up by $1$
    \item at an overall rate $(\irt_{k\ell} +  \rrt_k) X_{k\ell}$, $X_{k\ell}$ goes down by 1, and
    \item at an overall rate $\sum_j \ \irt_{jk} X_{jk} \frac{S_k}{n_k}$, $X_{k\ell}$ goes up by $\Pois(D_{k\ell})$.
\end{itemize}
Furthermore, at time $t=0$,
$
\E[X_{k\ell}(0)]=D_{k\ell} I_k(0)
$
Normalizing
\[
    \hat s_k(t) = \frac{S_k(t)}{n}, \quad \hat i_k(t) = \frac{I_k(t)}{n}, \quad \text{and} \quad \hat x_{k\ell}(t) = \frac{X_{k\ell}(t)}{nD_{k\ell}},
\]
the expected rates of change for $\hat s_k, \hat x_{k\ell},$ and $\hat i_k$ are given by
\[
    -\hat s_k\sum_j  \irt_{j k} W_{j k} \hat x_{j k}, \quad -( \irt_{k\ell} + \rrt_k)\hat x_{k\ell} + \hat s_k\sum_j  \irt_{j k}W_{j k} \hat x_{j k}, \quad -\rrt_k \hat i_k +\hat s_k\sum_j  \irt_{j k} W_{j k}\hat x_{j k},
\]
respectively.
The methods developed in Section~\ref{sec:lln} immediately imply a LLN for finite time horizons, leading to the differential equations
\begin{equation}
    \frac{ds_k}{dt} = - s_k \sum_j  \irt_{j k} W_{j k}  x_{j k}
\end{equation}
\begin{equation}
    \frac{dx_{k \ell}}{dt} = -( \irt_{k\ell} + \rrt_k)x_{k\ell} + s_k  \sum_j  \irt_{j k} W_{j k}  x_{j k}
\end{equation}
\begin{equation}
    \frac{di_k}{dt} = -\rrt_k i_k-  \frac{ds_k}{dt}
\end{equation}
with initial conditions $x_{k\ell}(0)=i_k(0)$.
The herd immunity threshold is now the threshold for which a new infection with $x_{k\ell}(t_0)=i_k(t_0)$ does not grow by more than a constant.  As the reader may easily verify, this threshold is now governed by the spectral radius 
$\Reff(t)$  of the matrix
$$
C_{k\ell}(t)=\frac{\irt_{k\ell}}{\irt_{k\ell}+\rrt_k}
W_{k\ell}s_\ell(t).
$$
While the technical details will be more complicated, we believe that our techniques will again allow one to prove an LLN which is uniform in the time horizon, but we have not worked out the details.

We further believe that our methods should allow one to analyze a dynamic version of the stochastic block model where susceptible-infected edges are rewired to susceptible-susceptible edges as in \cite{ball22:sir-rewire}. 
Finally, while we believe that the study of SIR dynamics on many random graph models will be amenable to the technology developed here, others, as e.g., an LLN for the time evolution and final size of the SIR epidemic on dynamically grown models such as preferential attachment might be more challenging.


\bibliography{references}
\nocite{andersson1998limit}
\nocite{newman2002spread}
\nocite{ball2008network}
\newpage

\appendix
\section*{\hfil Appendix \hfil}
\section{Branching Process Results}\label{app:branching}
In this appendix, we derive several  results about discrete time, multi-type branching processes used in the main body of this paper. 

\subsection{Preliminaries}
We start by establishing a few facts for multi-type branching processes which are usually formulated in the more restricted setting of ``positively regular'' branching processes, see, e.g., \cite{harris63:branching-process-theory}, but hold more generally, often with proofs which are straightforward extensions of those in the literature.

We  assume a finite type space $[K]$ and the associated state-space $\N_0^K$, with
 the state corresponding to a single individual of type $k$ represented by the unit vector $\e_k\in \N_0^K$ in direction $k$.
 We  use $\vec Z_m$ to denote the state of the branching process in generation $m$ and $\P_k$ to denote the probability distribution of $(\vec Z_m)_{m\geq 0}$ when starting from $\vec Z_0=\e_k$, $$\P_k(\cdot)=\P(\cdot\mid \vec Z_0=\e_k).  $$
We define $G_k(\cdot)$ to be
the generating function of the offspring distribution $p_k(\z)=\P_k(\vec Z_1=\z)$, and more generally define the generating function of  $\vec Z_m$ by
$$
G^{(m)}_k(\vec s)=\sum_{\z} \P_k(\vec Z_m=\z) \prod_{i=1}^Ks_i^{z_i}. 
$$
Setting $\vec G(\vec s)=(G_1(\vec s),\dots,G_K(\vec s) )$ and
$\vec G^{(m)}(\vec s)=(G_1^{(m)}(\vec s),\dots,G_K^{(m)}(\vec s) )$,
it is easy to see that
$
G^{(m)}_k(\vec s)=G_k(\vec G^{(m-1)}(\vec s)),
$
showing that $\vec G^{(n)}(\cdot)$ is the $n^\text{th}$ iterate of
$\vec G(\cdot)=\vec G^{(1)}(\cdot)$, the generating function for the offspring distribution.  Finally, we define the vector of extinction probabilities via
$$q_k=\P_k(\vec Z_m=0\text{ for some }m<\infty).$$

\begin{lemma}\label{lem:transience}
     Let $(\vec Z_m)_{m\geq 0}$ be a
    branching process such that all states in  $\N_0^K\setminus\{\vec 0\}$ are transient.  If $0\leq s_k<1$ for all $k$, then $\vec G^{(m)}(\s)\to \q$, where $\q=(q_1,\dots,q_k)$ is the vector of extinction probabilities.
\end{lemma}
\begin{proof}

Expressing  $q_k$ as the limit of $\P_k(\vec Z_m=0)=G_k^{(m)}(\vec 0)$, we need to prove that
$
G_k^{(m)}(\vec s)-G_k^{(m)}(\vec 0)\to 0$. 
To this end, we decompose the sum representing this difference into a finite sum plus a tail,
\begin{equation}
G_k^{(m)}(\vec s)-G_k^{(m)}(\vec 0)=\sum_{\z\neq \vec 0}
\P_k(\vec Z_m=\z) s_1^{z_1} \dots s_k^{z_k}\leq 
\sum_{\z\neq \vec 0:\norm{\z}_1 \leq R} \P_k(\vec Z_m=\z)  + O(\max_k s_k^{R}).
\end{equation}
Since all $\z\neq \vec 0$ are transient, the terms in the first sum go to $0$ as $m\to\infty$ for any fixed $R$.  Choosing first $R$ large enough, and then $m$ large enough, the right hand side can be made as small as desired, proving the lemma.
\end{proof}

To state the next lemma, we 
define a branching process to be irreducible if for all $i,j\in [K]$, there exists an  $m<\infty$ such that $\P(\vec Z_m\cdot \e_j>0\mid \vec Z_0=\e_i)>0$.

\begin{lemma}\label{lem:irred-bp-sols}
    Let $(\vec Z_m)_{m\geq 0}$ be an irreducible
    branching process such all states in  $\N_0^K\setminus\{\vec 0\}$ are transient.  If  there exists a solution $\s\neq \vec 1$ to $\vec G(\vec s) = \vec s$ in $[0,1]^K$, then it is equal to $\vec q$, and $q_k < 1$ for all $k$.
\end{lemma}

\begin{proof}
In view of the previous lemma, all we need to show is that if $\vec s = \vec G(\vec s)$ and $\vec s\neq 1$, then $s_k<1$ for all $k$.  Assume w.l.o.g. that $s_1<1$, and fix $k$. By irreducibility, there exists an $m$ such that $\P_k(\vec Z_m\cdot \e_1 > 0)>0$, implying that
    the generating function $G_k^{(m)}(\vec s)$  depends explicitly on $s_1$ and
    \begin{align*}
        G_k^{(m)}(\vec s) = \sum_{\vec z} \P_k(\vec Z_m = \vec z ) s_1^{z_1} \dots s_K^{z_K} 
        &\leq \sum_{\vec z, z_1 = 0} \P_k(\vec Z_m = \vec z ) + \sum_{\vec z, z_1 > 0} \P_k(\vec Z_m = \vec z )s_1^{z_1} \\
        &< \sum_{\vec z, z_1 = 0} \P_k(\vec Z_m = \vec z ) + \sum_{\vec z, z_1 > 0}\P_k(\vec Z_m = \vec z ) = 1.
    \end{align*}
    But, since $\vec s$ is a solution of $\vec s=\vec G(\vec s)$, it is also a solution of $\vec s = \vec G^{(m)} (\vec s)$ implying $s_k<1$ and hence the claim.
\end{proof}

To use these lemmas, we will need to establish that all states $\z\neq\vec 0$ are transient.  While not the most general lemma one can prove, the following will be sufficient for us.

\begin{lemma}\label{lem:irred-transience}
    Let $\mathbb S\subset\N_0^K$ be a set of states $\z$ such that $\P(\vec Z_1=\vec z|\vec Z_0=\vec z)=1$, and assume that for all $\vec z\notin \mathbb S$ there exists a state $\vec z'\in \mathbb S$ such that $\P(\vec Z_1=\z'|\vec Z_0=\z)>0$. Then all states $\z\notin \mathbb S$ are transient, so in particular $\lim_{m\to\infty}\P_k(\vec Z_m=\z)=0$ for all $\vec z\notin \mathbb S$ and all $k\in [K]$.
\end{lemma}
\begin{proof}
Fix $\z\notin \mathbb S$ and $\z'\in \mathbb S$ such that
$\P(\vec Z_1=\z'\mid \vec Z_0=\vec z)=\eps>0$, and let $T_\z$ be the first $m$ such that $\vec Z_m$ revisits $\z$ when starting at $\z$, i.e., let $T_\z=\min\{m>0\colon \vec Z_m=\z\}$ conditioned on $\vec Z_0=\z$, where the minimum is defined to be $\infty$ if there exists no such $m$.  Then $\P(T_\z<\infty\mid \vec Z_0=\z)$ is the probability that the chain revisits $\z$ at least once, and
$$
\P(T_\z=\infty\mid \vec Z_0=\z)\geq \P(\vec Z_1=z'\mid \vec Z_0=\z)=\eps>0.
$$
On the other hand, by the strong Markov property of $(\vec Z_m)_{m\geq 0}$, the probability distribution of $\vec Z_m$ after the stopping time $T_\z$ is independent of the history up to time $T_\z$ and identical to the original distribution (after the obvious time shift by $T_\z$).  As a consequence, the probability that the chain revisits $\z$ at least $k$ times is equal to 
$$\P(T_\z<\infty\mid \vec Z_0=\z)^k\leq (1-\eps)^k\to 0\quad\text{as}\quad k\to\infty,$$
showing that $\vec z$ is transient.
\end{proof}

\begin{remark}\label{rem:Poisson-BP-->assumptions}
The forward and backward branching processes of the SIR model introduced in the main part of this paper clearly obey the assumption of Lemma~\ref{lem:irred-transience}, since any state has non-zero probability of going extinct in one step by the properties of the Poisson distribution.
Furthermore, both are irreducible if $W$ is irreducible, allowing us to use both Lemma~\ref{lem:transience} and Lemma~\ref{lem:irred-bp-sols} in this case.
\end{remark}

\begin{lemma}\label{lem:subcritical-tails}
Let $\vec Z=\sum_m\vec Z_m$ be the total size of the branching process $(Z_m)_{m\geq 0}$ with offspring mean matrix $
C_{k\ell}=\E[\vec Z_1\cdot\e_\ell\mid \vec Z_0=\e_k]
$.  If $C$ has spectral radius $\rho(C)<1$ and  all components of $\vec Z_1$ have exponential tails, then all components of $\vec Z$ have exponential tails as well.
\end{lemma}

\begin{proof}
Let $\tilde\rho$ be such that $\rho(C)<\tilde\rho<1$, and let  $\vec v$ be a left eigenvector for the eigenvalue $\rho(C)$, chosen to have all positive entries (which is possible by Perron-Fr\"obenius).  To prove the lemma, we will first prove 
the existence of a constant $\alpha_0$ such that
\begin{equation}\label{eq:off_dist_upper_bound}
        \E\left[e^{\alpha \v \cdot \vec Z_m} \mid \vec Z_{m-1}\right]     \leq 
        e^{ \alpha \tilde\rho \vec v \cdot \vec Z_{m-1}}\quad\text{for all}\quad\alpha\leq \alpha_0,
    \end{equation}
and then use this bound to show that
\begin{equation} 
\label{exp-mom-exists}\E\left[e^{\alpha \v \cdot \vec Z}\mid \vec Z_0=\e_k\right]\leq e^{\frac \alpha{1-\tilde\rho}\v\cdot\e_k}
\quad\text{for all}\quad\alpha\leq \alpha_0(1-\tilde\rho),
\end{equation}
establishing the existence of the moment-generating function of $\vec Z$ in some neighborhood around $0$, and hence proving the claim.

   Let $\vec N_k$ be a random variable such that $\P(\vec N_k = \vec z) = p_k(\vec z)$. By the assumption that the offspring distribution has exponential tails, there exists some constant $\zeta<\infty$ such that for all $k$ and all sufficiently small $\alpha$
    \begin{equation}
    \label{eq:off_dist_tail}
        \E\left[e^{\alpha \vec v \cdot \vec N_k} ( \vec v \cdot \vec N_k)^2\right] \leq \zeta.
        \end{equation}
 Combined with the bound $e^x\leq 1+x+\frac{x^2}2e^x$, valid for all $x\geq 0$, and the fact that
$\E\left[\vec v \cdot \vec N_k\right] = \v C \vec e_k =\rho(C)\v\cdot \e_k$, this implies that
    \begin{equation}
    \label{eq:taylor_expansion}
        \E\left[e^{\alpha  \vec v \cdot \vec N_k}\right] \leq 1 + \alpha \E\left[\vec v \cdot \vec N_k\right] + \frac{\alpha^2}{2}\E\left[(\vec v \cdot \vec N_k)^2 e^{\alpha \vec v \cdot \vec N_k}\right]
        \leq 1 + \alpha \rho(C)\vec v \cdot \vec e_k + \frac{\alpha^2}{2}\zeta.
    \end{equation}
For $\alpha$
sufficiently small, we thus get that
    \begin{equation}
        \E\left[e^{\alpha  \vec v \cdot \vec N_k}\right] \leq 1 + \alpha \left(\rho(C)\v \cdot\vec e_k + \frac{\alpha \zeta}{2}\right) \leq 1 + \alpha \tilde\rho \v \cdot \vec e_k \leq e^{\alpha\tilde\rho\v \cdot \vec e_k}.
    \end{equation}
  Furthermore, letting $\vec N_{k, i} \sim \vec N_k$ i.i.d., we have
    \begin{align*}
        \E\left[e^{\alpha \v \cdot \vec Z_m} \mid \vec Z_{m-1}\right] 
        = \prod_{\ell \in [K]} \prod_{i=1}^{(\vec Z_{m-1})_\ell} \E\left[e^{\alpha \v \cdot \vec N_{\ell, i}}\right]        \leq \prod_{\ell \in [K]} \prod_{i=1}^{(\vec Z_{m-1})_\ell} e^{ \alpha \tilde\rho \v \cdot \vec e_\ell}
        = e^{ \alpha \tilde\rho \vec v \cdot \vec Z_{m-1}}.
    \end{align*}
This gives the bound \eqref{eq:off_dist_upper_bound}.

As a consequence of \eqref{eq:off_dist_upper_bound} applied twice in a row, we we get  that for $\alpha(1+\tilde\rho)\leq \alpha_0$,
  \begin{align*}
        \E\Big[e^{\alpha \v \cdot (\vec Z_m+\vec Z_{m-1})} \Big| \vec Z_{m-2}\Big] = \E\Big[\E[e^{\alpha \v \cdot \vec Z_{m}}\mid \vec Z_{m-1}]e^{\alpha \v \cdot \vec Z_{m-1}}
       \Big| \vec Z_{m-2}\Big]
        \leq e^{ \alpha (\tilde\rho+\tilde\rho^2) \vec v \cdot \vec Z_{m-2}}.
    \end{align*}
Continuing by induction, this proves that for $\alpha\sum_{i=0}^{m-1}\tilde\rho^i\leq \alpha_0$,
 \begin{align*}
        \E\Big[e^{\alpha \sum_{i=0}^m \v \cdot\vec Z_i} \Big| \vec Z_0=\e_k\Big] 
        \leq e^{ \alpha\v\cdot\e_k\sum_{i=0}^m\tilde\rho^i}.
    \end{align*}
   The bound 
    \eqref{exp-mom-exists} follows by
    by monotone convergence.
\end{proof}

\subsection{Backward Branching Process}

In this section, we formally define the  backwards branching process and use it to prove Lemma~\ref{lem:sinfty-continuity}.The process has label space $[K+1]$, with $k=K+1$ representing an infected state and labels $k\in [K]$ representing susceptible vertices of type $k$, and state space  $\N_0^{K+1}$. We will use the notation
$\bar \z=(\vec z,z_{K+1})=(z_1,\dots, z_{K+1})$ for vectors in $\N_0^{K+1}$, and similarly for vectors $\bar \s\in [0,1]^{K+1}$.  The offspring distribution of the backward process is then defined 
as follows.

Let $a_k\geq 0$ and let $p_k(\cdot)$ be the Poisson offspring distribution
$$
p_k(\z)=\prod_{\ell=1}^K\P(\Pois( C_{k\ell})=z_\ell).  
$$
For $k\in [K]$, we set
$$
\bar p_k(\bar\z)=\begin{cases}
e^{-a_k}p_k(\z)&\text{if }    z_{K+1}=0\\
1-e^{-a_k}&\text{if } z_{K+1}=1,
\end{cases}
$$
and for $k=K+1$ we set
$$
\bar p_{K+1}(\bar\z)=1\text{ if } \bar\z=(0,\dots,0,1)
\text{ and } 0 \text{ otherwise.}
$$
We will denote the generating function of this offspring distribution by $\bar{\vec F}$,  the state of the process in generation $m$ by  $\bar{\vec Z}_m$, and its generating function by ${\bar{\vec F}}^{(m)}$. It is easy to express $\bar{\vec F}$ in terms of the generating function $\G$ of the offspring distribution for the Poisson branching process, giving
$$
\bar { F}_k(\bar\s)=\begin{cases}
{e^{-a_k}}G_k(\s)  + \left(1-e^{-a_k}\right)s_{K+1}&\text{if }k\in [K] \\
s_{K+1}&\text{if } k=K+1.
\end{cases}
$$
If $s_{K+1}=0$, this expression, as well as the expression for the iterates takes a particular simple form, giving in particular
$$
\bar {F}_k^{(m)}((\s,0))=\begin{cases}
F_k^{(m)}(\s) &\text{if }k\in [K] \\
0&\text{if } k=K+1,
\end{cases}
$$
where $F^{(m)}$ is the $m^\text{th}$ iterate of
$$
F_k(\s)=e^{-a_k}G_k(\s).$$
Note that the function $F_k$ already appeared in the implicit equation 
\eqref{qk-with-x} for the final proportion of susceptible vertices, which we can now write in the form
\begin{equation}\label{eqn:implicit-equation}
    s_k= F_k(\s).
\end{equation}
It is clear that the survival probability of the backward branching process starting from the state $\bar{\vec Z}_0=\e_{K+1}$ is one.  The following lemma gives an expression for the extinction (and equivalently survival) probability starting from $e_k$ for $k\in [K]$.

\begin{lemma}\label{lem:backwards-survival}
      If $0\leq s_k<1$ for all $k\in [K]$, then $\vec F^{(m)}(\s)\to \q(\a)$, where $\q(a)=(q_1(\a),\dots,q_k(\a))$ is the vector of extinction probabilities, i.e.,  
      $q_k(\a)=\P(\bar{\vec Z}_m=0\text{ for some }m<\infty\mid \bar{\vec Z}_0=\e_k).$
   \end{lemma}

\begin{proof}
    First, we note that by the same arguments as those used in the proof of Lemma~\ref{lem:irred-transience}, all states $\bar \z=(\z,z_{K+1})$ with $\bar\z\neq 0$ are transient. To see this, set $T_{\bar\z}$ to be the first time the chain revisits the state $\bar\z$,  and note that
    \[
     \P(T_{\bar\z}=\infty)\geq   \P(\bar{\vec Z}_1 = (0, \dots, 0, z_{K+1} \mid \vec Z_0 = \bar\z) > 0
    \]
   since the probability that Poisson offspring distribution has $0$ children is bounded away from $0$.
   
Using the transience of the states $\bar\z=(\z,z_{K+1})$ with $\z\neq 0$, we may then proceed as in the proof of Lemma~\ref{lem:transience} to conclude that
\begin{align*}
F_k^{(m)}(\vec s)-\P_k(\bar{\vec Z}_m=0) 
&=
\bar F_k^{(m)}((\vec s,0))-\bar F_k^{(m)}((\vec 0,0))\\
&=\sum_{\z\neq \vec 0}
\P_k(\bar{\vec Z}_m=(\z,0)) s_1^{z_1} \dots _k^{z_k}\leq 
\sum_{\z\neq \vec 0:\norm{\z}_1 \leq R} \P_k(\bar{\vec Z}_m=\z)  + O(\max_k s_k^{R}).
\end{align*}
As before, the right hand side can be made arbitrary small by choosing first $R$ and then $m$ large enough, showing that the left hand side goes to $0$ as $m\to\infty$.   
 \end{proof}

\begin{lemma}[Uniqueness of Solution to Implicit Equation]
   Assume that $W$ is irreducible, and that $\sum_\ell a_\ell>0$.  Then the implicit equations \eqref{eqn:implicit-equation} are uniquely solved by $\s=\q(\a)$, where $\q(\a)$ is the vector of survival probabilities for the backward branching process, and $q_k(\vec a)<1$ for all $k$.
   \end{lemma}
\begin{proof}
In view of the last lemma, all we need to prove is that any solution $\s$ of \eqref{eqn:implicit-equation} is smaller than $1$ component-wise.  To this end, we first note that by the fact that $G_k(\s)\leq 1$ and our assumption on $\a$, there exists at least one component $s_\ell<1$.  Next we use that $\vec s=\vec F^{(m)}(\vec s)$ and 
 $\vec F\leq \vec G$ component wise imply that
\begin{equation}\label{s-leq-Gm(s)}
    s_k\leq G_k^{(m)}(\vec s).
\end{equation}
 But $s_\ell<1$ and irreducibility of the Poisson branching process imply that there exists an $m$ such that $G_k^{(m)}(\vec s)<1$, see the proof of Lemma~\ref{lem:irred-bp-sols}.  Since $k\in [K]$ was arbitrary, this completes the proof.
\end{proof}

 \begin{remark}\label{lem:upper-bound-implicit-eqn}
 It follows immediately from the above proof that
        $\vec q({\vec a}) \leq \vec q$ component-wise, where $\vec q$ is the smallest solution of the equation $\vec q = \vec G(\vec q)$ (which by Lemma~\ref{lem:irred-bp-sols} is also the vector of survival probabilities of the Poisson branching process with mean matrix $C$).  Indeed, by \eqref{s-leq-Gm(s)}, $\q(\vec a)\leq \vec G^{(m)}(\q(\a))$, and by Lemma~\ref{lem:transience}, the right hand side converges to $\q$.
    \end{remark}

 With these preparations, we are now ready to prove prove Lemma~\ref{lem:sinfty-continuity}. 
\begin{proof}[Proof of Lemma~\ref{lem:sinfty-continuity}]

 Assume that $(\tilde \s(0),\tilde \x(0))\to (\s(0),\x(0))$ along some sequence.  Let $(\tilde C,\tilde{\a})$ be the corresponding values for the matrix $C$ and the vector $\a$, and let $ \q(\tilde C, \tilde \a)$ be the corresponding sequence of survival probabilities.  We want to show that $ \q(\tilde C, \tilde \a)\to \q$, where $\q$ is the vector of survival probabilities of the Poisson branching process with mean matrix $C$.
Consider some subsequence $(C^j,\vec a^j)_{j \in \N}$  with   $(C^j,\vec a^j) \to (C,\vec 0)$ such that the limit of this subsequence $\wt{\vec q} := \lim_{j\to\infty} \vec q(C^j,{\vec a^j})$ exists, and let
$F_k^j(\cdot)$ be the corresponding sequence of generating funtions. Then
\[
    \wt q_k = 
    \lim_{j\to\infty} q_k(C^j,{\vec a^j}) 
    = \lim_{j\to\infty} F^j_k(\vec q(C^j{\vec a^j})) = G_k(\wt \q),
\]
    where the last equality is true by the fact that $F_k^j(\vec s)$ is jointly continuous in $\s$ and the parameters $C^j,\a^j$.
Thus $\wt{\vec q}$ satisfies $\wt{\vec q} = \vec G(\wt {\vec q})$. To complete the proof, we consider two cases.  If the spectral radius $\rho(C)$ of $C$ is at most one, we know that the survival probability of the Poisson branching process is $0$ and the implicit equation has only one solution, $\wt\q=\vec 1$, proving the desired convergence.  If $\rho(C)>1$, we can find some $\eps>0$ such that $\rho(C(1-\eps))>1$, and by the convergence of $C^j$, some $j_0$
such that for $j\geq j_0$, $C^j\geq (1-\eps) C$.  Since the survival probabilities are monontone in the entries of the mean matrix, and $q_k(C^j,\vec a^j)\leq q_k(C^j,0)$ by 
Remark~\ref{lem:upper-bound-implicit-eqn}, we conclude that
$\wt{\vec q} \leq \vec q_\eps$ component-wise, where $\q_\eps$ is survival probability of the Poisson branching process with mean matrix $(1-\eps)C$.  Since the latter has a strictly positive survival probability, we conclude that $\wt{\vec q}<\vec 1 $ component wise.
Therefore, $\wt{\vec q} = \vec q$ by uniqueness of the non-trivial solution to $\vec q = \vec G(\vec q)$ for probability generating functions. This gives the desired continuity result.
\end{proof}

\section{Analysis of the Differential Equations for
the SIR Epidemic on the SBM}
\label{app:diff-equ}

We start this appendix with the proof of Lemma~\ref{lem:dfq_limit}.
\begin{proof}[Proof of Lemma~\ref{lem:dfq_limit} ]
We will write the differential equation $d \bm y_t/{dt} = \bm b(\bm y_t)$ in the explicit form 
\eqref{eq:s_dfq_1}, \eqref{eq:x_dfq_1}, and \eqref{eq:i_dfq_1} and note that $\y_t\in U_0$ for all $t<\infty$ if $\y_0\in U_0$,
see Remark~\ref{rmk:si-sx-decreasing}.

By  \eqref{eq:s_dfq_1}, $d s_k(t)/dt \leq 0$ for all $t$.  Since $s_k(t)$ is non-negative, statement \ref{cond:s1-limit}  follows by monontone convergence.
Next we combined
\eqref{eq:s_dfq_1} and \eqref{eq:x_dfq_1} to conclude that
\begin{equation}\label{eq:s-x-down}
   \frac{dx_k(t)}{dt} + \frac{ds_k(t)}{dt} =  -(\irt + \rrt) x_k(t) \leq 0 
\end{equation}
and \eqref{eq:s_dfq_1} and  \eqref{eq:i_dfq_1} to conclude that
\begin{equation}\label{eq:s-i-down}
   \frac{di_k(t)}{dt} + \frac{ds_k(t)}{dt} =  - \rrt i_k(t) \leq 0.
\end{equation}
Monotone convergence then implies the existence of the limits $\lim_{t\to\infty}(x_k + s_k)(t)$ and $\lim_{t\to\infty}(i_k + s_k)(t)$.  Combined with the already established existence of the limit for $s_k$, this gives the existence of the limits in \ref{cond:x-limit} and \ref{cond:i-limit}.

To prove that $x_k(\infty)=0$, suppose towards contradiction that $ x_k(\infty) > 0$. Then there exist some constants $t_0<\infty$ and
$\zeta>0$ such that 
$x_k(t) \geq \zeta$ for all $t\geq t_0$.  
Inserted into
\eqref{eq:s-x-down}, this implies that
$$
x_k(t) + s_k(t)\leq x_k(t_0) + s_k(t_0)-(\irt+\rrt)\zeta(t-t_0)
$$
for all $t\geq t_0$, which is a contradiction since
$\y_t\in U_0$ so in particular $x_k(t)+s_k(t)\geq 0$ for all $t<\infty$.  This proves that $x_k(\infty)=0$. 
 The proof that $i_k(\infty)=0$ is identical.

 Next we prove \ref{sinfty>0}.  We  
 integrate \eqref{eq:s_dfq_1} to express $s_k(t)$ as
\begin{equation}\label{s-integrated}
    s_k(t) = s_k(0)\exp\left\{-\irt \sum_\ell \chi_\ell(t)W_{\ell k}\right\},
\end{equation}
where 
 \begin{equation}\label{chi-integrated}
\chi_\ell(t)= \int_{0}^t x_\ell(\tau) d\tau =\frac{x_\ell(0) + s_\ell(0)-s_\ell(t)-x_\ell(t)}{\irt + \rrt}.
\end{equation}
Here the identity in \eqref{chi-integrated} defines $\chi_\ell(t)$ and the second follows by
integrating the differential equation \eqref{eq:s-x-down} from $0$ to $t$.  This shows that $\chi_\ell(t)$ is bounded uniformly in $t$ implying claim \ref{sinfty>0}.

To prove \ref{x>0forallt}, we first note that $\frac{d}{dt}x_k(t)\geq -(\irt+\rrt)x_k(t)$, showing that $x_k(t)>0$ for all finite $t\geq t_0$ if $x_k(t_0)>0$, and similarly for $i_k(t)$.  Consider the set $\cal K(t)$ of indices $k$ such that $x_k(t)>0$ at time $t\geq 0$.  We want to prove that $\cal K(t)=[K]$ for all $t>0$. To this end, we will prove that if $\cal K(t)\neq [K]$, then we can find an $\ell\notin\cal K(t)$ such that $\ell\in\cal K(t')$ for all $t'>t$.
To see this, we use that by irreducibly of $W$, we can find $\ell\notin\cal K(t)$ and $k\in \calK(t)$ such that $W_{k\ell }>0$, which shows that
$$
\frac{dx_\ell}{dt}=-(\irt+\rrt)x_\ell(t)+\sum_{k'}x_{k'}(t)W_{k'\ell}s_\ell(t)>x_{k}(t)W_{k\ell}s_\ell(t)>0.
$$
This shows that  $\ell\in\cal K(t')$ for $t'>t$ close enough to $t$, and by our previous observation that $x_\ell(t'')>0$ for all $t''\geq t'$ if $x_\ell(t')>0$, we get $\ell\in\cal K(t')$ for all $t'>t$.  Starting from the observation that $\cal K(0)\neq\emptyset$,
this implies by induction that $\cal K(t)=[K]$ for all $t>0$, i.e,
$x_k(t)>0$ for all $k\in [K]$ and all $t>0$.
With the help of \eqref{eq:i_dfq_1} (and the fact that $i_k(t)>0$ implies $i_k(t')>0$ for all $t'>t$) this in turn implies that $i_k(t)>0$ for all $k\in [K]$ and all $t>0$.  Finally, 
$x_k(t)>0$ for all $0<t<\infty$ implies that $ds_k/dt<0$ with the help of \eqref{eq:s_dfq_1}.
\end{proof}

Next we state and prove a lemma alluded to in the discussion section of this paper, showing that  eventually, all components of $\x(t)$ decay exponentially fast.

\begin{lemma}\label{lem:x-exp-decay}
Assume that $W$ is irreducible and that $s_k(0)>0$ for all $k\in [K]$, let $\Rinf$ be the largest eigenvalue of the matrix $C(\infty) = \frac{\irt}{\irt + \rrt} W \diag(\s(\infty))$, and let $\v_\infty$ be the corresponding left eigenvector, chosen to have all non-negative entries. If $x(0)>0$, then $\Rinf<1$ and
    $$
    \lim_{t\to\infty}\x(t)e^{(\eta+\rrt)(1-\Rinf)t}=D\v_\infty,
    $$
    and
     $$
    \lim_{t\to\infty}\frac{d\x(t)}{dt}e^{(\eta+\rrt)(1-\Rinf)t}=-D(\eta+\rrt)(1-\Rinf)\v_\infty,
    $$
   where $D$ is a positive constant depending on the parameters of the model and the initial conditions (and the chosen normalization for $\v_\infty$).
\end{lemma}

\begin{proof}
Recall the vector form of the differential equations for the number of active half-edges \eqref{dxdt-matrix} in the SIR on SBM.
     By     
     the results from Section~\ref{sec:herd-immunity}, see in particular Lemma~\ref{lem:dfq_subcritical} and its proof, $\Rinf<1$, and by the bound    
     \eqref{uniform-s-bound}, the entries of $C(t)$ differ from those of $C(\infty)$ by an amount which is  decaying exponentially fast in $t$.  Combined with the Perron-Frobenius theorem, which  implies that $\x(t_0)e^{(\irt+\rrt)(C(\infty)-\Rinf)t}$ converges to $\v_\infty$ times the scalar product of $\x(t_0)$ with a suitably normalized right eigenvector and errors which are also exponentially small in $t$, this implies the first statement. Inserted back into the differential equation for $\x(t)$, this gives the second.
\end{proof}

We close this appendix with the proof of 
Lemma~\ref{lem:I-after-X}.

\begin{proof}[Proof of  
Lemma~\ref{lem:I-after-X}.]
  Rescaling  time so that $\irt+\rrt=1$ (and $\irt=1-\rrt$),
and using the differential equations for $i_k$ and $x_k$, we write the difference of the derivative of $i_k$ and $x_k$ as  
\begin{equation}\label{di-dx}
     \frac {di_k}{dt}-\frac {dx_k}{dt}=x_k- \rrt i_k. 
\end{equation}

To prove the fact that $x_k<i_k$ for all $t>0$, we first analyze the 
 ratio of
${x_k}$ and ${i_k}$ for small $t$, distinguishing the case $x_k(0)>0$ and $x_k(0)=0$.
  If $x_k(0)=i_k(0)>0$, we 
Taylor expand $x_k(t)$ and $i_k(t)$
as $x_k(t)=x_k(0)+x_k'(0)t+O(t^2)$ and 
$i_k(t)=x_k(0)+x_k'(0)t + (i_k'(0)-x_k'(0))t+O(t^2)$ where we used that $i_k(0)=x_k(0)$.  Using the fact that 
 $i_k'(0)-x'_k(0)=x_k(0)-\rrt i_k(0)=\irt x_k(0)$ by \eqref{di-dx}, this implies that
$$
\frac{x_k}{i_k}=1-\irt t+O(t^2).
$$
when $x_k(0)=i_k(0)>0$.
If  $x_k(0)=i_k(0)=0$ we have that
$x_k'(0)=i_k'(0)=(\x(0) C(0))_k>0$ by the irreducibility of $C(0)$ and the fact that $x(0)>0$, and
$i_k''(0)-x_k''(0)=x_k'(0)-\rrt i_k'(0)=\irt x_k'(0)$.
As a consequence, we may
expand $\frac{x_k}{i_k}$ as
$$
\frac{x_k}{i_k}=1-\frac{i_k-x_k}{i_k}
=1-\frac{\irt x_k'(0) \frac {t^2}2}{x_k'(0)t}(1+O(t))=
1-\frac\irt 2 t +O(t^2),
$$
showing that in both cases $\rrt<\frac{x_k}{i_k}<1$ for all sufficiently small, positive $t$.

To prove that 
 $$\frac{x_k}{i_k}<1\quad\text{for all}\quad t>0,
 $$
we rewrite the differential equations for $x_k$ and $i_k$ as
$$
\frac d{dt}x_k=(Z_k-1)x_k
\quad\text{and}\quad \frac d{dt}i_k=Z_kx_k-\rrt i_k
$$
where
$$
Z_k=\frac{(\x C(t))_k}{x_k}.
$$
(Note that for all $t>0$, both  $Z_k$ and  the ratio $\frac{x_k}{i_k}$ are well defined since $i_k>0$ and $x_k>0$ by Lemma~\ref{lem:diff-eq-sol-cts-finite-t}).  Next, we calculate the derivative of the ratio $\frac{x_k}{i_k}$, giving
$$
\frac {d}{dt}\frac{x_k}{i_k}
=
\frac{x_k}{i_k}\left(\frac{x'_k}{x_k}-\frac{i'_k}{i_k}\right)=
\frac{x_k}{i_k}
\left( Z_k-1+\rrt-Z_k\frac{x_k}{i_k}\right)
=\frac{x_k}{i_k}
\left( Z_k\left(1-\frac{x_k}{i_k}\right)-\irt\right).
$$
Assume now towards a contradiction that
there exists a time $0<t<\infty$ such that $\frac{x_k}{i_k}=1$, and let $t^*$ be the first such time.  For $t=t^*$, we then have that 
$$
\frac {d}{dt}\frac{x_k}{i_k}
=1 \cdot 
\left( Z_k\cdot 0 -\irt\right)=-\irt,
$$
which implies that shortly before $t^*$,
 $\frac{x_k}{i_k}>1$.  Since $\frac{x_k}{i_k}$ is continuous on $(0,\infty)$ and was strictly smaller than $1$ for small positive $t$, there must be another time $t'$ such that $\frac{x_k}{i_k}=1$, a contradiction.

To formalize the statement about the local maxima, we let $t_0$ be the location of the first local maximum of $x_k$, which means that $\frac{dx_k}{dt}\geq 0$ for all $t\leq t_0$, with equality if and only if $t=t_0$. Our goal will be to prove that  $\frac{di_k}{dt}>0$ for $0\leq t\leq t_0$, including $t=t_0$.

We start by proving this fact for $t=0$.  To this end, we note that  if $i_k(0)=x_k(0)=0$, then
$i'_k(0)=(\x(0)C(0)_k>0$ as we already observed earlier, while for
$i_k(0)=x_k(0)>0$ the
fact that  $x_k'(0)\geq 0$ and
 \eqref{di-dx}  imply that
$i_k'(0)\geq (1-\irt)x_k(0)>0$.   

 Next consider the quantity $Z_k$ for $0<t\leq t_0$.  Since the derivative of $x_k$ is non-negative for these $t$, we conclude that $Z_k\geq 1$, and combined with the fact that $\frac{x_k}{i_k}<1$ for all $0<t<\infty$, we conclude that
$$
\frac {d}{dt}\frac{x_k}{i_k}
\geq 
\frac{x_k}{i_k}
\left(1-\frac{x_k}{i_k}-\irt\right)=\frac{x_k}{i_k}
\left(\rrt-\frac{x_k}{i_k}\right)
\quad\text{for all}\quad 0<t\leq t_0.
$$
Recalling that $\rrt<\frac{x_k}{i_k}$ for all sufficiently small $t>0$, let $t_\rrt$ be the first time that
$\rrt=\frac{x_k}{i_k}$ (if there is no such time, we set 
$t_\rrt=\infty$).  In view of \eqref{di-dx}, we want to show that $t_\rrt>t_0$, since this would imply that
$\frac{di_k}{dt}>\frac{dx_k}{dt}\geq 0$ for $0\leq t\leq t_0$.  

Assume towards a contradiction that $t_\rrt\leq t_0$, and define
 $\eps(t)=\log \frac{x_k(t)}{\rrt i_k(t)}$.  Then
$0<\eps(t)\leq \eps(0)=-\log \rrt$ for $0<t<t_\rrt$ and
$\eps(t_\rrt)=0$. 
Furthermore
$$
\frac {d}{dt}\eps(t)
\geq \rrt
\left(1-e^{\eps(t)}\right)\geq -\rrt\eps(t)e^{\eps(t)}
\geq -\rrt^2\eps(t)
\quad\text{for all}\quad 0<t\leq t_\rrt,
$$
implying that
$$
0=\eps(t_\rrt)\geq e^{-\rrt^2  t_\rrt}\eps(0)>0,$$
a contradiction.
This completes the proof of the lemma.

\end{proof}

\end{document}